\documentclass[11pt]{amsart}
\usepackage{amsfonts,latexsym,amsmath, amssymb}
\usepackage{mathrsfs,MnSymbol}
\usepackage{url,color}
\usepackage{upgreek}
\usepackage{fancyhdr}
\usepackage{hyperref}
\usepackage{times}
\usepackage{scalefnt}
\usepackage{url}
\usepackage{cite}
\usepackage{amsmath}

\newcommand\myuline{\bgroup\markoverwith
	{\textcolor{gray}{\rule[0.4ex]{2pt}{3pt}}}\ULon}

\newcommand{\bea}{\begin{eqnarray}}
\newcommand{\eea}{\end{eqnarray}}
\def\beaa{\begin{eqnarray*}}
\def\eeaa{\end{eqnarray*}}
\def\ba{\begin{array}}
\def\ea{\end{array}}
\def\be#1{\begin{equation} \label{#1}}
\def\eeq{\end{equation}}

\def\lV{\lVert}
\def\rV{\rVert}

\def\div{\mathrm{div}}
\def\d{\partial}

\def\b{{\beta}}
\def\be{{\beta}}
\def\ga{\gamma}

\def\ep{\epsilon}

\def\Om{\Omega}

\def\nab{\nabla}

\def\FF{{\mathcal F}}
\def\EE{{\mathcal E}}

\def\bfR{{\mathbf R}}

\def\R{{\mathbb{R}}}
\def\C{{\mathbb{C}}}

\def\N{{\mathbb N}}
\def\Z{{\mathbb{Z}}}

\def\rr{{\varrho}}

\def\P{{\mathbb P}}

\newcommand{\<}{  \langle   }
\renewcommand{\>}{  \rangle   }

\newtheorem{theorem}{Theorem}[section]
\newtheorem{lemma}[theorem]{Lemma}
\newtheorem{proposition}[theorem]{Proposition}
\newtheorem{corollary}[theorem]{Corollary}

\numberwithin{equation}{section}

\numberwithin{equation}{section}

\topmargin - .23 in

\begin{document}

%\scalefont{1.10}

\title[Global regulairity for hyperbolic liquid crystal]
{Small data global regularity for simplified 3-D Ericksen-Leslie's compressible hyperbolic liquid crystal model}

\author[J. Huang]{Jiaxi Huang$^1$}
\email{\href{mailto:jiaxih@mail.ustc.edu.cn}{jiaxih@mail.ustc.edu.cn}}

\author[N. Jiang]{Ning Jiang$^{2}$}
\email{\href{mailto:njiang@whu.edu.cn}{njiang@whu.edu.cn}}

\author[Y. Luo]{Yi-Long Luo$^3$}
\email{\href{mailto:luo-yl@whu.edu.cn}{luo-yl@whu.edu.cn}}

\author[L. Zhao]{Lifeng Zhao$^1$}
\email{\href{mailto:zhaolf@ustc.edu.cn}{zhaolf@ustc.edu.cn}}

\address{$^1$School of Mathemtical Sciences, University of Science and Technology of China, Hefei 230026, Anhui, P. R. China}

\address{$^2$School of Mathemtical and Statistics, Wuhan University, Wuhan 430072, P. R. China}

\address{$^3$Department of Mathematics, City University of Hong Kong, Kowloon Tong, Hong Kong, People's Republic of China}

\subjclass[2010]{}

\keywords{Global regularity, hyperbolic, compressible liquid crystal}

%\thanks{Acknowledgments.}
\begin{abstract}
In this article, we consider the Ericksen-Leslie's hyperbolic system for compressible liquid crystal model in three spatial dimensions.  Global regularity for small and smooth initial data near equilibrium is proved for the case that the system is a nonlinear coupling of compressible Navier-Stokes equations with wave map to $\mathbb{S}^2$.  Our argument is a combination of vector field method and Fourier analysis. The main strategy to prove global regularity relies on an interplay between the control of high order energies and decay estimates, which is based on the idea inspired by the method of space-time resonances. In particular the different behaviors of the decay properties of the density and velocity field for compressible fluids at different frequencies play a key role.
\end{abstract}
\maketitle

%\centerline{\today}

\setcounter{tocdepth}{2}
%\pagenumbering{roman} \tableofcontents \newpage
%\pagenumbering{arabic}

\section{Introduction}
The hydrodynamic theory of liquid crystals was established by Ericksen \cite{Ericksen-1961-TSR, Ericksen-1987-RM, Ericksen-1990-ARMA} and Leslie \cite{Leslie-1968-ARMA, Leslie-1979} in the 1960's (see also Section 5.1 of \cite{Lin-Liu-2001} ). The compressible hyperbolic Ericksen-Leslie's system was derived in  Leslie's paper \cite{Leslie-1968-ARMA} which  consists of the following equations of the velocity field $u(x,t)\in \mathbb{R}^3$ and the orientation field $d(x,t)\in \mathbb{S}^2$, and $(x,t)\in \R^3\times \R^+$:
\begin{equation}\label{PHLC}
\begin{aligned}
\left\{ \begin{array}{l}
\d_t \rho+\div (\rho u)=0,\\
\partial_t (\rho u) +\div(\rho u \otimes u)  + \nabla P = \div(\Sigma_1+\Sigma_2+\tilde{\sigma}), \\
\rho \ddot{d} = \Delta d+\Gamma d +\lambda_1(\dot{d}+Bd)+\lambda_2Ad.
\end{array}\right.
\end{aligned}
\end{equation}

In the above system, $P=P(\rho)$ is the pressure, the superposed dot denotes the material derivative $\partial_t + u \cdot\nabla$, the symbol $\otimes$ denotes the tensor product for two vectors with entries $(a\otimes b)_{ij}=a_ib_j$ for $1\leq i,j\leq 3$, and
\begin{equation*}
\begin{aligned}
A = \tfrac{1}{2}(\nabla u + \nabla^\top u)\,,\quad B= \tfrac{1}{2}(\nabla u - \nabla^\top u)\,,\end{aligned}
\end{equation*}
represent the rate of strain tensor and skew-symmetric part of the strain rate, respectively. We also define $N = \dot d + B d$ as the rigid rotation part of director changing rate by fluid vorticity. Here $A_{ij} = \tfrac{1}{2} (\partial_j u_i + \partial_i u_j)$, $ B_{ij} = \tfrac{1}{2} (\partial_j u_i - \partial_i u_j) $, $(B d)_i =B_{ki} d_k$, and $(\nabla d \odot \nabla d)_{ij} = \partial_i d_k \partial_j d_k$.
The notations $\Sigma_1,\Sigma_2$, and $\tilde{\sigma}$ are as follows:
\begin{equation}\label{Extra-Sress-sigma}
\begin{aligned}
&\Sigma_1:=\tfrac{1}{2}\mu_1(\nab u+\nab^{\top}u)+\mu_2\div u\ I,\\
&\Sigma_2:=\tfrac{1}{2}|\nab d|^2I-\nab d\odot \nab d,\\
&\tilde\sigma_{ji}=  \nu_1 d_k d_p A_{kp}  d_i d_j + \nu_2  d_j N_i  + \nu_3 d_i N_j  + \nu_5 A_{ik}d_k d_j   + \nu_6 d_i A_{jk}d_k \,.
\end{aligned}
\end{equation}
These coefficients $\nu_i (1 \leq i \leq 6, i \neq 4)$ which may depend on material and temperature, are usually called Leslie coefficients, and are related to certain local correlations in the fluid. Usually, the following relations are frequently used, see \cite{Ericksen-1961-TSR, Leslie-1968-ARMA, Wu-Xu-Liu-ARMA2013}.
\begin{equation}\label{Coefficients-Relations}
\lambda_1=\nu_2-\nu_3\,, \quad\lambda_2 = \nu_5-\nu_6\,,\quad \nu_2+\nu_3 = \nu_6-\nu_5\,.
\end{equation}
The first two relations are necessary conditions in order to satisfy the equation of motion identically, while the third relation is called {\em Parodi's relation}, which is derived from Onsager reciprocal relations expressing the equality of certain relations between flows and forces in thermodynamic systems out of equilibrium. Under Parodi's relation, we see that the dynamics of a compressible nematic liquid crystal flow involve five independent Leslie coefficients in \eqref{Extra-Sress-sigma}. Furthermore, in \eqref{PHLC}, the Lagrangian multiplier $\Gamma$ is (which ensures the geometric constraint $|d|=1$):
\begin{equation}\label{Lagrange-Multiplier}
\Gamma = - \rho |\dot{d}|^2 + |\nabla d|^2 - \lambda_2 d^\top A d\, .
\end{equation}
We remark that the last two terms in the third equation of \eqref{PHLC} is the so-called kinematic transport, i.e.
\begin{equation}\label{kinematic transport}
  g= \lambda_1 (\dot{d} + B d) + \lambda_2 A d\,,
\end{equation}
which represents the effect of the macroscopic flow field on the microscopic structure. The material coefficients $\lambda_1$ and $\lambda_2$ reflect the molecular shape and the slippery part between the fluid and particles. The first term represents the rigid rotation of the molecule , while the second term stands for the stretching of the molecule by the flow. We remark that in \cite{Leslie-1968-ARMA}, Leslie also derived the incompressible model which can be written as:
\begin{equation}\label{IncompressibleLC}
\begin{aligned}
\left\{ \begin{array}{l}
\partial_t u + u\cdot\nabla u -\tfrac{1}{2}\mu_1 \Delta u + \nabla P = -\div(\nab d\odot \nab d )+\tilde{\sigma}, \\
\div u=0\,,\\
\rho_1 \ddot{d} = \Delta d+\Gamma d +\lambda_1(\dot{d}+Bd)+\lambda_2Ad\,,
\end{array}\right.
\end{aligned}
\end{equation}
where the Lagrangian multiplier $\Gamma = - \rho_1 |\dot{d}|^2 + |\nabla d|^2 - \lambda_2 d^\top A d\,.$ In the last equation of \eqref{IncompressibleLC}, the constant $\rho_1 > 0$ is called the inertial constant. To distinguish them with their parabolic analogues, which have no $\ddot{d}$ terms in \eqref{PHLC} and \eqref{IncompressibleLC}, we call \eqref{PHLC} and \eqref{IncompressibleLC} Ericksen-Leslie's hyperbolic system for compressible and incompressible liquid crystal, respectively.

As one of the important special cases, parabolic Ericksen-Leslie system ($\rho_1=0$, $\lambda_1=-1$)  has been extensively studied since the mid 80's. The static analogue of the parabolic Ericksen-Leslie's system is the so-called Oseen-Frank model, whose mathematical study was initiated from Hardt-Kinderlehrer-Lin \cite{Hardt-Kinderlehrer-Lin-CMP1986}. Since then there have been many works in this direction. In particular, the existence and regularity or partial regularity of the approximation (usually Ginzburg-Landau approximation as in \cite{Lin-Liu-CPAM1995}) dynamical Ericksen-Leslie's system was started by the work of Lin and Liu in \cite{Lin-Liu-CPAM1995}, \cite{Lin-Liu-DCDS1996} and \cite{Lin-Liu-ARMA2000}. The simplest system preserving the basic energy law can be obtained by neglecting the Leslie stress and by specifying some elastic constants. In 2-D case, the existence of global weak solutions with at most a finite number of singular times were proved by Lin-Lin-Wang \cite{Lin-Lin-Wang-ARMA2010}. Recently, Lin and Wang proved global existence of weak solution for 3-D case with the initial director field lying in the hemisphere in \cite{Lin-Wang-CPAM2016}. For the general parabolic Ericksen-Leslie's system,  local well-posedness was proved by Wang-Zhang-Zhang in \cite{Wang-Zhang-Zhang-ARMA2013}, and in \cite{Huang-Lin-Wang-CMP2014} existence of global solutions and regularity in $\mathbb{R}^2$ was established by Huang-Lin-Wang. We refer to the references therein for more results for the parabolic Ericksen-Leslie's system.

The study of hyperbolic Ericksen-Leslie's system is at the infancy compared with its parabolic analogue. In hyperbolic Ericksen-Leslie's system, $\rho_1>0$ and \eqref{PHLC} is a compressible Navier-Stokes equation coupled with a wave map with target $\mathbb{S}^2$. In one dimensional case, suppose $\tilde{\sigma}=0$ and $u=0$, the system \eqref{PHLC} can be reduced to a so-called nonlinear variational wave equation. Zhang and Zheng  studied systematically the dissipative and energy conservative solutions \cite{Zhang-Zheng-CPDE2001,Zhang-Zheng-ARMA2010}. For the multidimensional case, there has been some progress on the hyperbolic system of liquid crystal. Very recently, De Anna and Zarnescu \cite{DeAnna-Zarnescu-2016} considered the inertial Qian-Sheng model of liquid crystals. They derived the energy law and proved the local well-posdedness for bounded initial data and global well-posedness under the assumptions that the initial data is small in suitable norm and the coefficients satisfy some further damping property. Furthermore, for the inviscid version of the Qian-Sheng model, Feireisl-Rocca-Schimperna-Zarnescu proved the global existence of the {\em dissipative solution} in \cite{FRSZ-2016} which is inspired from that of incompressible Euler equation defined by P-L. Lions \cite{Lions-1996}.

There are growing interest in the study of hyperbolic Ericksen-Leslie's system. Recently, the second and third named authors of the current paper started the research of the well-posedness of Ericksen-Leslie's hyperbolic system in the context of classical solutions: first for the incompressible model \eqref{IncompressibleLC}  in \cite{Jiang-Luo-2018} and then with Tang for the compressible model \eqref{PHLC} in \cite{Jiang-Luo-Tang-2019}. More precisely, in \cite{Jiang-Luo-2018} and \cite{Jiang-Luo-Tang-2019}, under some natural constraints on the Leslie coefficients which ensure the basic energy law is dissipative, they showed the local-in-time existence and uniqueness of the classical solution to the system \eqref{PHLC} with finite initial energy. Furthermore, with an additional assumption on the coefficients which provides a damping effect, i.e. $\lambda_1 < 0$, and the smallness of the initial energy, the unique global classical solution was established. Here we remark that both in \cite{Jiang-Luo-2018} and \cite{Jiang-Luo-Tang-2019} the assumption $\lambda_1 < 0$ plays a crucial role in the global well-posedness. The limiting behaviors when the inertia constant $\rho_1 \rightarrow 0$  were considered in \cite{Jiang-Luo-Tang-NA, Jiang-Luo-Tang-Zarnescu, Jiang-Luo-2019}

When $\lambda_1=0$, which corresponds to the model without kinematic transport, the wellposedness is much more subtle analytically because there are no obvious damping effect. For incompressible model \eqref{IncompressibleLC}, Cai-Wang \cite{CW} recently made progress for the simplified Ericksen-Leslie system, namely, the case with  $\nu_i = 0, i=1\,,\cdots\,,6$, $i\neq 4$. They proved the global regularity of \eqref{IncompressibleLC} near the constant equilibrium by employing the vector field method. In \cite{HJLZ-incomp}, the authors of the current paper considered the more general incompressible model \eqref{IncompressibleLC}: still $\nu_2=\nu_3=0$, but $0 \neq \nu_5=\nu_6 >-\nu_4$, and $0 \neq \nu_1 > -2(\nu_4+\nu_5)$. Of course, from \eqref{Coefficients-Relations}, we still have $\lambda_1=\lambda_2=0$, i.e. the kinematic transport vanishes. They proved the global regularity for the more general case (but still special, compared with the most general case, say, \eqref{Extra-Sress-sigma} and \eqref{Coefficients-Relations}) near the constant equilibrium $(u,d)= (\vec{0},\vec{i})$.  The argument of \cite{HJLZ-incomp} is a combination of vector field method and Fourier analysis inspired by the so-called space-time resonance method.

In the current paper, we address the small data global regularity of the Ericksen-Leslie's hyperbolic system for compressible model, i.e. \eqref{PHLC}. We consider the simplest case: $\nu_1=\nu_2=\nu_3=\nu_5=\nu_6=0$, hence $\lambda_1=\lambda_2=0$, and the pressure $P(\rho)$ obeys the $\gamma$-law, i.e. $P(\rho)=a\rho^{\gamma}$ with $\gamma\geq 1,a>1$. The aim of this paper is to prove the global regularity of this simplest case near the constant equilibrium $(\rho,u,d)= (1,\vec{0},\vec{i})$. More precisely, we study the Ericksen-Leslie's hyperbolic liquid crystal model in the following form: $\mu_1,\mu_2>0$,
\begin{equation}        \label{ori_sys}
\left\{
\begin{aligned}
\d_t \rho+\nab \cdot(\rho u)&=0,\\
\partial_t u+u\cdot\nabla u+\tfrac{\nabla P(\rho)}{\rho} & =\tfrac{1}{\rho}\big[\tfrac{\mu_1}{2}\Delta u+(\tfrac{\mu_1}{2}+\mu_2)\nab\div u+\tfrac{1}{2}\nab|\nab d|^2-\sum_{j=1}^3\d_j(\nabla d \cdot \d_j d)\big], \\
\rho\ddot{d}-\Delta d & =(-\rho|\dot{d}|^2+|\nabla d|^2)d.
\end{aligned}
\right.
\end{equation}
on $\R^3\times \R^+$ with the constraint $|d|=1$. Compared with the incompressible model, the decay of the density $\rho(x,t)$ and velocity field $u(x,t)$ in compressible model is much more subtle.

\subsection{The main theorem}
To state our main theorem, we need some notations. Define the perturbed angular momentum operators by
\begin{equation*}
\tilde{\Om}_i u=\Om_i u+A_i u,\ \tilde{\Om}_i d=\Om_i d,
\end{equation*}
where $\Om=(\Om_1,\Om_2,\Om_3)$ is the rotation vector-field $\Om=x\wedge\nab$ and $A_i$ is defined by
\begin{equation}
A_1=e_2\otimes e_3-e_3\otimes e_2,\ \ A_2=e_3\otimes e_1-e_1\otimes e_3,\ \ A_3=e_1\otimes e_2-e_2\otimes e_1.
\end{equation}
We define the scaling vector-field $S$ by
\begin{equation*}
S=t\d_t+x_i\d_{x_i}.
\end{equation*}

Let
\begin{equation*}
\Gamma\in \{\d_t,\d_1,\d_2,\d_3,\tilde{\Om}_1,\tilde{\Om}_2,\tilde{\Om}_3\}
\end{equation*}
and $Z^a=S^{a_1}\Gamma^{a'}$, where $a=(a_1,a'):=(a_1,a_2,\cdots,a_8)\in\Z_+^8$, $\Gamma^{a'}=\Gamma^{a_2}\Gamma^{a_3}\cdots\Gamma^{a_8}$, we define
\begin{equation*}
\rr=\rho-1,\ \rr^{(a)}=Z^{a}\rr,\ u^{(a)}:=Z^a u,\ \ d^{(a)}:=Z^a d.
\end{equation*}

The main result of this paper is as follows:
\begin{theorem}           \label{Ori_thm}
	Assume that $N_0:=60,N_1:=6,h:=6$, there is $\epsilon_0 > 0$, such that $(\rho_0,u_0,d_0,d_1)$ are initial data near equilibrium $(1,\vec{0},\vec{i},\vec{0})$ satisfying the smallness assumptions
	\begin{equation}            \label{MainAss_dini}
	\sup_{|a|\leq N_1} \{\lV  \rr_0^{(a)}\rV_{H^{N(a)}}+\lV  u_0^{(a)}\rV_{H^{N(a)}}+\lV \nab d_0^{(a)}\rV_{H^{N(a)}}+\lV d_1^{(a)}\rV_{H^{N(a)}}\}\leq \epsilon_0,
	\end{equation}
	where $N(a)=N_0-|a|h$ for $0\leq|a|\leq N_1$. Then there exists a unique global solution $(\rr,u,d)$ of the system (\ref{ori_sys}) with initial data
	\begin{equation*}
	\rr(0)=\rr_0,\ \ u(0)=u_0,\ \ d(0)=d_0,\ \d_t d(0)=d_1,
	\end{equation*}
	satisfies the energy bounds
	\begin{align*}
	\sup_{0\leq t\leq T}\sup_{|a|\leq N_1} \{\lV(\rr^{(a)},  u^{(a)})(t)\rV_{H^{N(a)}}+ &\lV(\nab\rr^{(a)},\nab u^{(a)})\rV_{L^2([0,t]:H^{N(a)-1}\times H^{N(a)})}\}\\
	&+\sup_{0\leq t\leq T}\{\lV \nab d(t)\rV_{H^{N(0)}}+\lV \d_t d(t) \rV_{H^{N(0)}}\}\lesssim \ep_0,\\
	\sup_{0\leq t\leq T}\sup_{1\leq|a|\leq N_1}\{\lV \nab d^{(a)}(t)\rV_{H^{N(a)}}+&\lV \d_t d^{(a)}(t)\rV_{H^{N(a)}}\}\lesssim \ep_0(1+T)^{\bar{\delta}}.
	\end{align*}
	for any $T\in[0,\infty)$, where $\bar{\delta}<10^{-15}$ depends on $\ep_0$, $N_0$ and $N_1$.
\end{theorem}

In order to analyze the system on the Fourier side, we introduce the angles $\phi_1$ and $\phi_2$. Precisely, assume that $\phi_1$ represents the angle between $x$-axis and projection of $d$ onto $x-y$ plane and $\phi_2$ represents the angle between $d$ and $x-y$ plane, i.e
\begin{equation*}
d=(\cos\phi_1\cos\phi_2,\sin\phi_1\cos\phi_2,\sin\phi_2).
\end{equation*}
Then (\ref{ori_sys}) can be rewritten as the $(\rr,u,\phi)$-system:
\begin{equation}                   \label{Main_Sys}
\left\{\begin{aligned}
&\d_t \rr+\nab\cdot u=F_1,\\
&\partial_t u-\tfrac{\mu_1}{2}\Delta u-(\tfrac{\mu_1}{2}+\mu_2)\nab\div u+a\gamma\nab\rr  =F_2,\\
&\d_t^2 \phi-\tfrac{1}{1+\rr}\Delta\phi= F_3,
\end{aligned}
\right.
\end{equation}
where $\phi:=(\phi_1,\phi_2)^{\top}$, and the nonlinear terms $F_1,F_2,F_3$ are
\begin{align*}
F_1:=&-u\cdot\nab\rr-\rr\nab\cdot u,\\
F_2:=&-u\cdot\nab u+(a\gamma-\tfrac{P'(1+\rr)}{1+\rr})\nab\rr+\tfrac{-\rr}{1+\rr}\big[\tfrac{\mu_1}{2}\Delta u+(\tfrac{\mu_1}{2}+\mu_2)\nab\div u\big]\\
&+\tfrac{1}{1+\rr}\big[\tfrac{1}{2}\nab(|\nab\phi|^2)-\sum_{j=1}^3\d_j(\nab\phi \d_j \phi) \big]+Err2,\\
F_3:=&-\d_t u\cdot\nab\phi-2u\cdot\nab\d_t\phi-u\cdot\nab(u\cdot\nab\phi)+Err3,
\end{align*}
and the error terms are
\begin{align*}
Err2:=-\tfrac{1}{1+\rr}\big[\tfrac{1}{2}\nab(\sin^2 \phi_2|\nab\phi_1|^2)-\sum_{j=1}^3\d_j(\sin^2\phi_2 \nab\phi_1 \d_j \phi_1) \big],\\
Err3:=(2\tan \phi_2(\dot{\phi_1}\dot{\phi_2}-\tfrac{\nab\phi_1\nab\phi_2}{1+\rr},
\tfrac{1}{2}\sin 2\phi_2(-\dot{\phi_1}^2+\tfrac{|\nab\phi_1|^2}{1+\rr}))^{\top}.
\end{align*}

Now it suffices to consider the system (\ref{Main_Sys}). For simplicity, we define
\begin{equation*}
\rr^{(a)}=Z^a \rr,\ u^{(a)}=Z^a u,\ \phi^{(a)}:=Z^a \phi,
\end{equation*}
Then we establish the following result:
\begin{theorem}\label{Main_thm}
	With the notation and hypothesis in Theorem \ref{Ori_thm}, the initial data $(\rr_0,u_0,\phi_{1,0},\phi_{1,1},\phi_{2,0},\phi_{2,1})$ satisfies
	\begin{equation}            \label{MainAss_ini}
	    \sup_{|a|\leq N_1} \{\lV (\rr_0^{(a)},  u_0^{(a)})\rV_{H^{N(a)}}+\sum_{i=1}^2(\lV \nab \phi_{i,0}^{(a)}\rV_{H^{N(a)}}+\lV \phi_{i,1}^{(a)}\rV_{H^{N(a)}})\}\leq \epsilon_0.
	\end{equation}
	Then there exists a unique global solution $(\rr,u,\phi)$ of the system (\ref{Main_Sys}) with initial data
	\begin{equation*}
	\rr(0)=\rr_0,\ u(0)=u_0,\ \ \phi_i(0)=\phi_{i,0},\ \d_t \phi_i(0)=\phi_{i,1},
	\end{equation*}
	satisfies the energy bounds
	\begin{align*}
	\sup_{0\leq t\leq T}\sup_{|a|\leq N_1} \{\lV  (\rr^{(a)},u^{(a)})(t)\rV_{H^{N(a)}}+ &\lV(\nab\rr^{(a)},\nab u^{(a)})\rV_{L^2([0,t]:H^{N(a)-1}\times H^{N(a)})}\}\\
	&+\sup_{0\leq t\leq T}\{\lV \d_t \phi(t)\rV_{H^{N(0)}}+\lV \nab \phi(t)\rV_{H^{N(0)}}\} \lesssim \ep_0,\\\label{Bdphia}
	\sup_{0\leq t\leq T}\sup_{1\leq|a|\leq N_1}\{\lV \d_t \phi^{(a)}(t)\rV_{H^{N(a)}}+&\lV \nab \phi^{(a)}(t)\rV_{H^{N(a)}}\}\lesssim \ep_0(1+T)^{\bar{\delta}}.
	\end{align*}
	for any $T\in[0,\infty)$, where $\bar{\delta}<10^{-15}$ depends on $\ep_0$, $N_0$ and $N_1$.
\end{theorem}

\subsection{Main ideas}
The main strategy to prove global regularity relies on an interplay between the control of high order energies and decay estimates, which is based on the method of space-time resonances developed by Germain, Masmoudi and Shatah \cite{GMS-09,GMS-12-Water,GMS-12-Sch}.
\\

\noindent{\em Dispersive estimates and decay estimates.} In the proof the Theorem \ref{Main_thm}, we need various decay estimates of $\rr,u$, and $\phi$. The decay of $\phi$ is easily obtained by dispersive estimates of $e^{it|\nab|}$. However, for the decay of $\rr$ and $u$, the linear coupling between $\rr$ and $u$ in the first two equations in (\ref{Main_Sys_rr,v,w}) brings the first difficulty: If we only use the dissipation, the decay of $\nab\rr$ and $\nab u$ may be $\<t\>^{-1+2\delta}$, which does not allow for the high order energy estimates of $\phi$ with slow growth. Therefore, we have to investigate the wave propagation of $\rr$ and $u$ to obtain the following faster decay rates:
\begin{equation}       \label{GoodDec-rho,u}
\lV (\nab\rr,\nab u)\rV_{L^{\infty}}\lesssim \ep_1\<t\>^{-1-\delta}.
\end{equation}

In this paper, in order to get around the first difficulty, and obtain the decay estimate (\ref{GoodDec-rho,u}), we introduce the ``compressible part" $v$ and divergence free part $w$ of $u$. Precisely, let
\begin{equation}          \label{U}
v:=|\nab|^{-1}\nab\cdot u,\ \ w:=\P u=(I+\nab(-\Delta)^{-1}\nab\cdot)u.
\end{equation}
Conversely, $u$ can be expressed in terms of $v$ and $w$,
\begin{equation*}
u=-|\nab|^{-1}\nab v+w.
\end{equation*}
Hence, we can also rewrite (\ref{Main_Sys}) as
\begin{equation}                   \label{Main_Sys_rr,v,w}
\left\{\begin{aligned}
&\d_t \rr+|\nab|v=F_1,\\
&\d_t v-(\mu_1+\mu_2)\Delta v-a\gamma|\nab|\rr=|\nab|^{-1}\nab\cdot F_2,\\
&\partial_t w-\tfrac{\mu_1}{2}\Delta w=\P F_2,\\
&\d_t^2 \phi-\tfrac{1}{1+\rr}\Delta\phi= F_3.
\end{aligned}
\right.
\end{equation}
From this, we observe that $w$ satisfies a parabolic equation, the decay estimates of $\nab w$ are easily obtained by Duhamel's formula and heat semi-group. Hence, it suffices to consider the decay of $\nab\rr$ and $\nab v$. For this we define the solution of linear system
\begin{equation*}
\left\{ \begin{aligned}
&\d_t \tilde{\rr}+|\nab|\tilde{v}=0,\\
&\d_t \tilde{v}-(\mu_1+\mu_2)\Delta \tilde{v}-a\gamma|\nab|\tilde{\rr}=0,\\
&(\tilde{\rr}(0),\tilde{v}(0))=(f,g),
\end{aligned}\right.
\end{equation*}
by
\begin{equation}          \label{def-B}
(\tilde{\rr},\tilde{v})^{\top}=e^{tB}(f,g)^{\top},
\end{equation}
where the linear operator $e^{tB}$ defined as
\begin{equation}          \label{e-tB-Sym}
\begin{aligned}
\FF (e^{tB}(f,g)^{\top})(\xi)=&e^{tB(\xi)} (\widehat{f},\widehat{g})^{\top}(\xi) \\
:=&
\left( \begin{aligned}
&\tfrac{\Lambda_+(\xi) e^{\Lambda_-(\xi)t}-\Lambda_- e^{\Lambda_+(\xi) t}}{\Lambda_+(\xi)-\Lambda_-(\xi)} & |\xi|\tfrac{ e^{\Lambda_-(\xi) t}- e^{\Lambda_+(\xi) t}}{\Lambda_+(\xi)-\Lambda_-(\xi)}\\
&a\gamma|\xi|\tfrac{ e^{\Lambda_-(\xi) t}- e^{\Lambda_+(\xi) t}}{\Lambda_+(\xi)-\Lambda_-(\xi)}& \tfrac{\Lambda_+(\xi) e^{\Lambda_+(\xi)t}-\Lambda_-(\xi) e^{\Lambda_-(\xi) t}}{\Lambda_+(\xi)-\Lambda_-(\xi)}
\end{aligned}  \right)
\left( \begin{aligned}
\widehat{f}(\xi)\\
\widehat{g}(\xi)
\end{aligned}  \right),
\end{aligned}
\end{equation}
and
\begin{equation}            \label{Lamda+-_def}
\Lambda_{\pm}(\xi):=\left\{\begin{aligned}
&-\tfrac{\mu}{2}|\xi|^2\pm i\sqrt{a\gamma|\xi|^2-\tfrac{\mu^2}{4}|\xi|^4},\ \ \ \ \ \ |\xi|\leq \tfrac{2\sqrt{a\gamma}}{\mu},\ \mu=\mu_1+\mu_2,\\
&-\tfrac{\mu}{2}|\xi|^2\pm \sqrt{\tfrac{\mu^2}{4}|\xi|^4-a\gamma|\xi|^2},\ \ \ \  \ \ \ |\xi|> \tfrac{2\sqrt{a\gamma}}{\mu}.
\end{aligned}\right.
\end{equation}
We divide the above linear solution into the low-frequency case and high-frequency case, i.e
\begin{equation}          \label{decomposition-e-tB}
e^{tB}(f,g)^{\top}=P_{\leq k}e^{tB}(f,g)^{\top}+P_{>k}e^{tB}(f,g)^{\top},
\end{equation}
Then by Fourier analysis, we may have the following linear estimates:
\begin{equation}          \label{Lin-1}
\lV e^{tB}\nab^l(f,g)^{\top}\rV_{L^{\infty}} \lesssim t^{-\frac{3}{2p}-\frac{l}{2}}\lV(f,g)\rV_{L^p}+e^{-ct}\lV(f,g)\rV_{H^{l+2}},\ 1\leq p\leq 2.
\end{equation}
For low-frequency term,
\begin{equation}          \label{Lin-2}
\lV P_{\leq k}e^{tB}\nab(f,g)^{\top}\rV_{L^{\infty}}\lesssim \<t\>^{-5/4} 2^k\lV |\nab|^{-1}(f,g)\rV_{L^2},
\end{equation}
and for high-frequency term,
\begin{equation}           \label{Lin-3}
\lV P_{> k}e^{tB}\nab(f,g)^{\top}\rV_{L^{\infty}}\lesssim \<t\>^{-3/2}2^{-13k^-}\lV (f,g)\rV_{L^{\infty}}.
\end{equation}
The decomposition (\ref{decomposition-e-tB}) and a similar linear estimates (\ref{Lin-3}) when $k=0$ were investigated by Li and Zhang \cite{LiZh}. Here when we choose $k\in \Z$ satisfying $2^k\approx \<t\>^{-c}$ for sufficient small $c>0$, the decay estimate (\ref{GoodDec-rho,u}) can be obtained by Duhamel's formula, (\ref{Lin-1}), (\ref{Lin-2}) and (\ref{Lin-3}).
\\

\noindent{\em Energy estimates.} We would like to control the increment of both high order Sobolev norms and weighted norms for solutions of (\ref{Main_Sys}). In fact, by the dissipation of $\rr$ and $u$ and the decay estimates, we can obtain that the energies of $\rr^{(a)},u^{(a)}$ for $0\leq |a|\leq N_1$, and $\phi$ are bounded. However, for the energy estimates of $\phi^{(a)}$, $1\leq |a|\leq N_1$, the presence of time resonances without matching null structure brings a second difficulty. More precisely, in proving the energy estimates, the main difficulty are caused by the quadratic terms $\rr^{(b)}\Delta \phi^{(c)}$ and $u^{(b)}\cdot\nab\d_t\phi^{(c)}$.

To get around the main difficulty, we discuss the term $\rr^{(b)}\Delta\phi^{(c)}$ in detail. Let
\begin{equation*}
\Phi:=\d_t\phi+i|\nab|\phi,\ \ \Phi^{(a)}:=\d_t\phi^{(a)}+i|\nab|\phi^{(a)}.
\end{equation*}
In order to control the increment of $\phi^{(a)}$ for $1\leq |a|\leq N_1$, by $\phi$-equation one needs to consider the space-time integrals
\begin{equation}          \label{HardTerm0}
\int_0^t\int_{\R^6}m(\xi,\eta) \overline{ \widehat{P_{k}\Phi^{(a)}}}(\xi) \widehat{P_{k_1}\rr^{(b)}}(\xi-\eta)\widehat{P_{k_2}\Phi^{(c)}}(\eta)d\xi d\eta ds,
\end{equation}
for any $t\in[0,T]$, $k,k_1,k_2\in\Z$, and $m(\xi,\eta):=(1+\xi^{2N(a)}) |\eta|$. Notice that it is not possible to estimate (\ref{HardTerm0}) by moving the absolute value inside the time integral, due to the slow decay of $u^{(b)}$ in $L^{\infty}$. So we need to integrate by parts in time or $\eta$, for this define the profiles
\begin{equation*}
\Psi:=e^{-it|\nab|}\Phi,\ \ \Psi^{(a)}:=e^{-it|\nab|}\Phi^{(a)}.
\end{equation*}
Then after rewriting the functions $ \Phi^{(a)}$, $\Phi^{(c)}$ in terms of the variables $\Psi^{(a)}$, $\Psi^{(c)}$, we need to consider the space-time integral
\begin{equation}          \label{HardTerm}
\int_0^t\int_{\R^6}e^{-is(|\xi|-|\eta|)}m(\xi,\eta) \overline{ \widehat{P_{k}\Psi^{(a)}}}(\xi) \widehat{P_{k_1}\rr^{(b)}}(\xi-\eta)\widehat{P_{k_2}\Psi^{(c)}}(\eta)d\xi d\eta ds,
\end{equation}
We observe that the above integral admits the time resonant frequencies
\begin{equation*}
\mathscr{T}=\{(\xi,\eta):|\xi|-|\eta|=0,m(\xi,\eta)\neq 0\},
\end{equation*}
but the set of space resonances is null, i.e
\begin{equation*}
\mathcal{S}=\{(\xi,\eta):\nab_{\eta}(|\xi|-|\eta|)=0,m(\xi,\eta)\neq 0\}=\emptyset.
\end{equation*}
Hence, for the contribution of high-low interaction, i.e $k_1\sim k,k_2<k-5$ and high-high interaction, i.e $k_1\sim k_2>k+5$, we may integrate by parts in time and then use the decay of $\d_s \rr^{(b)}$ and $\d_s\Psi^{(c)}$ to control (\ref{HardTerm}). For the contribution of low-high interaction, i.e $k\sim k_2,k_1\leq k+6$, we may use integration by parts in $\eta$ and the decay of $\lV\nab \rr^{(b)}\rV_{L^2}$ and $\Phi^{(c)}$ to control the increment of (\ref{HardTerm}).
\\

\noindent{\em $L^2$-bounds on the derivatives of $\widehat{\Psi^{(a)}}$.} In the above arguments, we also need the following bootstrap assumptions
\begin{equation*}
\lV \FF^{-1}(|\xi|\d_{\xi}\widehat{\Psi^{(a)}})\rV_{H^{N(a)}}\lesssim \ep_1\<t\>^{H(a+1)\delta},\ \ |a|\leq N_1-1,
\end{equation*}
which corresponds to weighted bounds in the physical space and can be linked to the energy estimates by
\begin{equation*}
\lV \FF^{-1}(|\xi|\d_{\xi}\widehat{\Psi^{(a)}})\rV_{H^{N(a)}}\lesssim \lV S\Psi^{(a)}\rV_{H^{N(a)}}+\lV \Om\Psi^{(a)}\rV_{H^{N(a)}}+\lV \Psi^{(a)}\rV_{H^{N(a)}}+\lV t\d_t\Psi^{(a)}\rV_{H^{N(a)}}.
\end{equation*}

\subsection{Notations} For $N\geq 0$, $1\leq p\leq\infty$, let $L^p,H^N,\dot{H}^N,W^{N,p}$ and $\dot{W}^{N,p}$ denote the standard Sobolev spaces on $\R^3$, i.e
\begin{equation*}
L^p:=L^p(\R^3),\ H^N:=H^N(\R^3),\ \dot{H}^N:=\dot{H}^N(\R^3),\ W^{N,p}:=W^{N,p}(\R^3),\ \dot{W}^{N,p}:=\dot{W}^{N,p}(\R^3).
\end{equation*}
For any $k\in\Z$,  let $k^+:=\max(k,0)$ and $k^-:=\min(k,0)$. For any $x\in\R$, denote $\<x\>:=\sqrt{1+x^2}$. For any two numbers $M_1$ and $M_2$ and a absolute constant $C$, we denote
\begin{equation*}
M_1\lesssim M_2,\ M_2\gtrsim M_1,\ \mathrm{if}\ M_1\leq CM_2.
\end{equation*}

Here we will use the following multi-index
\begin{equation*}
n:=(\alpha,\beta,\gamma)\in\N^3.
\end{equation*}
Denote
$$|n|:=\alpha+\beta+\gamma,\ \ \d^n:=\d^{\alpha}_{x_1}\d^{\beta}_{x_2}\d^{\gamma}_{x_3}.$$

\subsection{Outline}
In section 2, we fix notations and state the main bootstrap proposition \ref{Main_Prop}. We also state several lemmas, such as dispersive linear bounds, Hardy-type estimate and weighted $L^{\infty}-L^2$ estimate. In section 3, we use the bootstrap assumptions to prove various decay estimates of $\rr,u$ and $\phi$. Then we start the proof of the main bootstrap proposition \ref{Main_Prop} in section 4, 5, and 6, where we obtain improved energy estimates (\ref{Main_Prop_result1}) and the $L^2$ bounds (\ref{Main_Prop_result2}) on the derivation of $\widehat{\Psi}$.

\section{Preliminaries and the main proposition}

In this section, we start by summarizing our main definitions and notations.

\subsection{Some analysis tools}
The Fourier transform of $f$ is defined as follows:
\begin{equation*}
\mathcal{F}(f)(\xi):=\widehat{f}(\xi)=\int_{\R^3}e^{-ix\cdot\xi}f(x)dx.
\end{equation*}
We use $\mathcal{F}^{-1}(f)$ to denote the inverse Fourier transform of $f$. Fix an even smooth function $\varphi:\R \rightarrow [0,1] $ supported in $[-2,2]$ and equal to 1 in $[-1,1]$. For simplicity of notation, we also let $\varphi:\R^3 \rightarrow [0,1]$ denote the corresponding radial function on $\R^3$. For any $k\in \Z$, $I\subset \R$, let
\begin{gather*}
\varphi_k(x):=\varphi(x/2^k)-\varphi(x/2^{k-1}),\ \ \varphi_I(x):=\sum_{m\in I\cap \Z}\varphi_m(x),\\
\varphi_{\leq k}(x):=\sum_{l\leq k}\varphi_{l}(x),\ \varphi_{> k}(x):=\sum_{l> k}\varphi_{l}(x).
\end{gather*}
The frequency projection operator $P_k$, $P_I$, $P_{\leq k}$ and $P_{>k}$ is defined by the Fourier multiplier $\varphi_k(\xi)$, $\varphi_I(\xi)$, $\varphi_{\leq k}(\xi)$ and $\varphi_{>k}(\xi)$, i.e.
\begin{gather*}
\widehat{P_k f}(\xi)=\varphi_k(\xi)\widehat{f}(\xi),\ \ \ \widehat{P_I f}(\xi)=\varphi_I(\xi)\widehat{f}(\xi),\\
\widehat{P_{\leq k} f}(\xi)=\varphi_{\leq k}(\xi)\widehat{f}(\xi)\ \mathrm{and} \ \widehat{P_{> k} f}(\xi)=\varphi_{> k}(\xi)\widehat{f}(\xi).
\end{gather*}
Moreover, we have the following Bernstein inequality: For any $k\in \Z$,
\begin{equation}        \label{BernIneq}
	\lV P_k f\rV_{L^q}\lesssim 2^{3k(\frac{1}{p}-\frac{1}{q})}\lV P_k f\rV_{L^p},\ \ \ 1\leq p\leq q\leq \infty.
\end{equation}

Let
\begin{equation*}
\mathcal{J}:=\{(k,j)\in\Z\times\Z_+:k+j\geq 0\}.
\end{equation*}
For any $(k,j)\in\mathcal{J}$, let
\begin{equation*}
\tilde{\varphi}_j^{(k)}(x):=\left\{
\begin{aligned}
&\varphi_{\leq -k}(x),\ \mathrm{if}\ k+j=0\ \mathrm{and}\ k\leq 0,\\
&\varphi_{\leq 0}(x),\ \mathrm{if}\ j=0\ \mathrm{and}\ k\geq 0,\\
&\varphi_{j}(x),\ \mathrm{if}\ k+j\geq 1\ \mathrm{and}\ j\geq 1,
\end{aligned}
\right.
\end{equation*}
then for any $k\in\Z$ fixed, $\sum_{j\geq \max(-k,0)}\tilde{\varphi}_j^{(k)}=1$. For any $(k,j)\in\mathcal{J}$, we define the operator $Q_{jk}$ by
\begin{equation*}
(Q_{jk}f)(x):=\tilde{\varphi}_j^{(k)}(x)\cdot P_kf(x),
\end{equation*}
and denote
\begin{equation*}
f_{j,k}(x):=P_{[k-2,k+2]}Q_{jk}f(x).
\end{equation*}
Then $P_k f$ for fix $k\in\Z$ can be decomposed as
\begin{equation}      \label{Decom_Pkf}
P_k f(x)=P_{[k-2,k+2]}P_kf(x)=P_{[k-2,k+2]}\sum_{j\geq \max(-k,0)}\tilde{\varphi}_j^{(k)}(x)P_kf(x)=\sum_{j\geq \max(-k,0)}f_{j,k}(x).
\end{equation}
Moreover, for any $f\in L^2(\R^3)$ and $(k,j)\in \mathcal{J}$, let $\beta:=1/1000$, we have
\begin{equation}    \label{Dec_fjk}
\lV \widehat{f_{j,k}}\rV_{L^{\infty}}\lesssim \min\{2^{3j/2}\lV Q_{jk}f\rV_{L^2},\ 2^{j/2-k}2^{\beta(j+k)}\sum_{|\alpha|\leq 1}\lV Q_{jk}\Om^{\alpha}f\rV_{L^2}\},
\end{equation}
see \cite[Lemma 3.4.]{IoPa}.

We state several decay estimates for some linear operators.
\begin{lemma}[Linear decay estimates]
	(\romannumeral1) For any $f\in \mathcal{S}$ and $l\in \N$, we have
	\begin{equation}        \label{Dec_heat}
	\lV e^{t\Delta}|\nab|^l f\rV_{L^{\infty}}\lesssim t^{-\frac{3}{2p}-\frac{l}{2}} \lV f\rV_{L^p},\ \ \ \  1\leq p\leq \infty.
	\end{equation}
	(\romannumeral2) For any $f,g\in \mathcal{S}$ and $l\in \N$, for the operator (\ref{def-B}) we have
	\begin{gather}          \label{e-tB_Lin}
	\lV e^{tB} |\nab|^l(f,g)^{\top}\rV_{L^{\infty}}\lesssim t^{-\frac{3}{2p}-\frac{l}{2}}\lV (f,g)\rV_{L^p}+e^{-ct} \lV |\nab|^l(f,g)\rV_{H^{2}},\ \ 1\leq p\leq 2,\\ \label{etB_L2-L2}
	\lV e^{tB} |\nab|^l(f,g)^{\top}\rV_{L^2}\lesssim t^{-\frac{l}{2}}\lV (f,g)\rV_{L^2}+e^{-ct} \lV |\nab|^l(f,g)\rV_{L^{2}},\\         \label{e tB_Lp-Lp}
	\lV P_{>k_0}e^{tB}(f,g)^{\top}\rV_{L^p}\lesssim \<t\>^{-3/2}2^{-13k_0^-}\lV P_{>k_0}(f,g)\rV_{L^p},\ \ {\rm for }\ 2\leq p\leq \infty, k_0\in\Z.
	\end{gather}
\end{lemma}
\begin{proof}
	The proof of (\ref{Dec_heat}) is standard. Here we prove the bound (\ref{e-tB_Lin})-(\ref{e tB_Lp-Lp}) in detail.
	
	\textbf{Step 1. We prove the bound (\ref{e-tB_Lin}).}	
	From the definition of $\Lambda_{\pm}$ (\ref{Lamda+-_def}), we easily check that
	\begin{equation}       \label{e tB}
	|e^{tB(\xi)}|\lesssim \left\{ \begin{aligned}
	&e^{-c|\xi|^2t},\ \ \ \text{for }|\xi|\leq \tfrac{2\sqrt{a\gamma}}{\mu},\\
	&e^{-ct},\ \ \ \text{for }|\xi|> \tfrac{2\sqrt{a\gamma}}{\mu},
	\end{aligned}\right.
	\end{equation}
	where $c>0$ small. Then for $t\leq 1$, using Sobolev embedding it follows that
	\begin{equation*}
	\lV e^{tB}|\nab|^l (f,g)\rV_{L^{\infty}}\lesssim \lV e^{tB} |\nab|^l(f,g)\rV_{H^2}\lesssim \lV  |\nab|^l(f,g)\rV_{H^{2}}.
	\end{equation*}
	For $t>1$, let $j_0\in\Z$ satisfy $2^{j_0}\approx\frac{2\sqrt{a\gamma}}{\mu}$, from Young's inequality, Sobolev embedding, Hausdorff-Young inequality and (\ref{e tB}) we obtain for $1\leq p\leq 2$ and $q^{-1}+p^{-1}=1$
	\begin{align*}
	\lV e^{tB} |\nab|^l(f,g)\rV_{L^{\infty}}\lesssim & \lV \varphi_{\leq j_0}(|\nab|)e^{tB}|\nab|^l  (f,g)\rV_{L^{\infty}}+\lV \varphi_{> j_0}(|\nab|)e^{tB} |\nab|^l (f,g)\rV_{L^{\infty}}\\
	\lesssim & \lV \FF^{-1} (\varphi_{\leq j_0}(|\xi|)e^{tB(\xi)}|\xi|^l)\rV_{L^q}\lV(f,g)\rV_{L^p}+\lV \varphi_{> j_0}(|\nab|)e^{tB} \<\nab\>^2 |\nab|^l (f,g)\rV_{L^2}\\
	\lesssim & \lV e^{-c|\xi|^2 t}|\xi|^l\rV_{L^p}\lV(f,g)\rV_{L^p}+ \lV \varphi_{> j_0}(|\xi|)e^{tB(\xi)} \rV_{L^{\infty}} \lV\<\nab\>^2  |\nab|^l(f,g)\rV_{L^2}\\
	\lesssim & t^{-\frac{3}{2p}-\frac{l}{2}}\lV  (f,g)\rV_{L^p}+e^{-ct}\lV\<\nab\>^2 |\nab|^l (f,g)\rV_{L^2}.
	\end{align*}
	This concludes the proof of (\ref{e-tB_Lin}).
	
    \textbf{Step 2. We prove the bound (\ref{etB_L2-L2}).}

    By (\ref{e tB}) and Plancherel's Theorem we get this bound
    \begin{align*}
    \lV e^{tB}|\nab|^l (f,g)^{\top}\rV_{L^2}=&\lV e^{tB(\xi)}|\xi|^l (\widehat{f}(\xi),\widehat{g}(\xi)^{\top}\rV_{L^2}\\
    \lesssim &\lV \varphi_{\leq j_0}(\xi) e^{-c|\xi|^2t}|\xi|^l\rV_{L^{\infty}}\lV(f,g)\rV_{L^2}+\lV \varphi_{>j_0}(\xi) e^{-ct}\rV_{L^{\infty}}\lV|\nab|^l(f,g)\rV_{L^2}\\
    \lesssim &t^{-l/2}\lV(f,g)\rV_{L^2}+e^{-ct}\lV|\nab|^l(f,g)\rV_{L^2}.
    \end{align*}
	
	\textbf{Step 3. We prove the bound (\ref{e tB_Lp-Lp}).}
	
	The following bound
	\begin{equation*}
	\lV P_{>j_0} e^{tB}(f,g)^{\top}\rV_{L^p}\lesssim e^{-ct} \lV (f,g)\rV_{L^p},
	\end{equation*}
	for some $c>0$ small is given in \cite[Lemma 4.3]{LiZh}. Then it suffices to show that
	\begin{equation*}
	\lV P_{[k_0,j_0]} e^{tB}(f,g)^{\top}\rV_{L^p}\lesssim \<t\>^{-3/2}2^{-13k_0} \lV (f,g)\rV_{L^p}.
	\end{equation*}
	By Young's inequality we have
	\begin{align*}
	\lV P_{[k_0,j_0]}e^{tB}(f,g)^{\top}\rV_{L^p}=&\lV \mathcal{F}^{-1}(\varphi_{[k_0,j_0]}e^{tB(\xi)})\ast (f,g)^{\top}\rV_{L^p}\\
	\lesssim & \lV \mathcal{F}^{-1}(\varphi_{[k_0,j_0]}e^{tB(\xi)})\rV_{L^1}\lV(f,g)\rV_{L^p}.
	\end{align*}
	Now it suffices to show the bound
	\begin{equation*}
	\lV \mathcal{F}^{-1}(\varphi_{[k_0,j_0]}e^{tB(\xi)})\rV_{L^1}\lesssim \<t\>^{-3/2}2^{-13k_0}.
	\end{equation*}
	By integration by parts in frequency, we have
	\begin{align*}
	\lV \mathcal{F}^{-1}(\varphi_{[k_0,j_0]}e^{tB(\xi)})\rV_{L^1}=&\lV(\varphi_{<0}(x)+\varphi_{\geq 0}(x)) \mathcal{F}^{-1}(\varphi_{[k_0,j_0]}e^{tB(\xi)})\rV_{L^1}\\
	\lesssim & \lV |x|^{-2}\varphi_{<0}(x)\rV_{L^1(\R^3)}\lV\nab_{\xi}^2(\varphi_{[k_0,j_0]}(\xi)e^{tB(\xi)})\rV_{L^1}\\
	&+\lV |x|^{-4}\varphi_{\geq 0}(x)\rV_{L^1(\R^3)}\lV\nab_{\xi}^4(\varphi_{[k_0,j_0]}(\xi)e^{tB(\xi)})\rV_{L^1}\\
	\lesssim & \lV\nab_{\xi}^2(\varphi_{[k_0,j_0]}(\xi)e^{tB(\xi)})\rV_{L^1}+\lV\nab_{\xi}^4(\varphi_{[k_0,j_0]}(\xi)e^{tB(\xi)})\rV_{L^1}.
	\end{align*}
	From the definition of $e^{tB(\xi)}$ (\ref{e-tB-Sym}), we have
	\begin{align*}
	m(\xi):=\frac{\Lambda_+(\xi) e^{\Lambda_-(\xi)t}-\Lambda_-(\xi) e^{\Lambda_+(\xi) t}}{\Lambda_+(\xi)-\Lambda_-(\xi)}
	= \big( \frac{\mu}{2}|\xi|^2 t\frac{\sin bt}{bt}+\cos bt \big)e^{-\frac{\mu}{2}|\xi|^2 t},
	\end{align*}
	where
	\begin{equation*}
	b:=\sqrt{a\gamma|\xi|^2-\frac{\mu^2}{4}|\xi|^4}.
	\end{equation*}
	With the restriction $\varphi_{[k_0,j_0]}(\xi)$, by direct computations we have
	\begin{gather*}
	\nab_{\xi}\big(\frac{\sin bt}{bt}\big)\lesssim |\xi|t^2\lesssim (|\xi|^2t)^2 2^{-3k_0},\\
	\nab_{\xi}^2\big(\frac{\sin bt}{bt}\big)\lesssim |\xi|^2t^4+t^2\lesssim (1+|\xi|^2t)^4 2^{-6k_0},\\
	\nab_{\xi}^3\big(\frac{\sin bt}{bt}\big)\lesssim |\xi|^3t^6+|\xi|t^4+t^2\lesssim (1+|\xi|^2t)^6 2^{-9k_0},\\
	\nab_{\xi}^4\big(\frac{\sin bt}{bt}\big)\lesssim |\xi|^4t^8+|\xi|^2t^6+t^4+t^2\lesssim (1+|\xi|^2t)^8 2^{-12k_0}.
	\end{gather*}
	Similarly, we have
	\begin{gather*}
	\nab_{\xi}^l(\cos bt)\lesssim (1+|\xi|^2t)^{2l}2^{-3lk_0},\\
	\nab_{\xi}^l e^{-\frac{\mu}{2}|\xi|^2 t}\lesssim (1+|\xi|^2 t)^l 2^{-lk_0} e^{-\frac{\mu}{2}|\xi|^2 t}.
	\end{gather*}
	From these estimates, we
	\begin{align*}
	&\lV\nab_{\xi}^2(\varphi_{[k_0,j_0]}(\xi)m(\xi))\rV_{L^1}+\lV\nab_{\xi}^4(\varphi_{[k_0,j_0]}(\xi)m(\xi))\rV_{L^1}\\
	\lesssim & \int_{\R^3} \varphi_{[k_0,j_0]}(\xi)(1+|\xi|^2t)^9 2^{-13k_0}e^{-\frac{\mu}{2}|\xi|^2 t}d\xi \lesssim \<t\>^{-3/2}2^{-13k_0}.
	\end{align*}
	For other symbols in $e^{tB(\xi)}$, we have the same estimates. This completes the proof of (\ref{e tB_Lp-Lp}).

\end{proof}

We also need the following Hardy-type estimate involving localization in frequency and space.
\begin{lemma}[Lemma 3.5, \cite{IoPa}]       \label{AkBk_lem}
	For $f\in L^2(\R^3)$ and $k\in\Z$ let
	\begin{equation*}
	F_k(f):=\lV P_k f\rV_{L^2}+\sum\limits_{l=1}^3 \lV \varphi_k(\xi)\d_{\xi_l}\widehat{f}(\xi)\rV_{L^2},\ \ B_k(f):=\big[\sum\limits_{j\geq \max(-k,0)}2^{2j}\lV Q_{jk}f\rV_{L^2}^2\big]^{1/2}.
	\end{equation*}
	Then, for any $k\in\Z$,
	\begin{equation*}
	B_k\lesssim \left\{
	\begin{aligned}
	&\sum_{|k'-k|\leq 4}F_{k'},\ \mathrm{if}\ k\geq 0,\\
	&\sum_{k'\in\Z}F_{k'}2^{-|k-k'|/2}\min\{1,2^{k'-k}\},\ \mathrm{if}\ k\leq 0.
	\end{aligned}
	\right.
	\end{equation*}
\end{lemma}

\begin{lemma}[Lemma 2.2, \cite{Wxc}]    \label{ContSymb}
	If $f:\R^{6}\rightarrow\C$ and $k_1,k_2\in\Z$, then the following estimate holds:
	\begin{equation*}
	\lV \int_{\R^{6}} f(\xi_1,\xi_2) e^{i(x_1\cdot\xi_1+x_2\cdot\xi_2)}\varphi_{k_1}(\xi_1)\varphi_{k_2}(\xi_2)d\xi_1\xi_2\rV_{L^1_{x_1, x_2}}\lesssim \sum\limits_{m=0}^{7}\sum\limits_{j=1}^2 2^{mk_j}\lV\partial_{\xi_j}^m f\rV_{L^{\infty}}.
	\end{equation*}
\end{lemma}

To obtain the energy estimates, one often needs to analyze the symbols. Define a class of symbol as follows
\begin{equation*}
\mathcal{S}^{\infty}:=\{m:\R^6\rightarrow\C,\ m\ \text{is continous and } \lV\mathcal{F}^{-1}m\rV_{L^1}<\infty\},
\end{equation*}
whose associated norms are defined as
\begin{equation*}
\lV m\rV_{\mathcal{S}^{\infty}}:=\lV \mathcal{F}^{-1}m\rV_{L^1},
\end{equation*}
and
\begin{equation*}
\lV m\rV_{\mathcal{S}^{\infty}_{k,k_1,k_2}}:=\lV m(\xi,\eta)\varphi_k(\xi)\varphi_{k_1}(\xi-\eta)\varphi_{k_2}(\eta)\rV_{\mathcal{S^{\infty}}}.
\end{equation*}
Then we have

\begin{lemma}[Bilinear estimate, \cite{IoPu}]
	Given $m\in\mathcal{S}^{\infty}$ and two well-defined functions $f_1,\ f_2$, then the following estimate holds:
	\begin{gather}   \label{Symbol_Pre}
	\lV \mathcal{F}^{-1}\big(\int_{\R^3}m(\xi,\eta)\widehat{f}_1(\xi-\eta)\widehat{f}_2(\eta)d\eta\big)(x)\rV_{L^r}\lesssim\lV m\rV_{\mathcal{S}^{\infty}} \lV f_1\rV_{L^p} \lV f_2\rV_{L^q},\ \ \frac{1}{r}=\frac{1}{p}+\frac{1}{q}.
	\end{gather}
\end{lemma}

\begin{lemma}[Hardy-type inequality and interpolation inequality]    \label{Hardy}
	Given $1\leq p\leq \infty$, $N\in\N$, and two well-defined functions $f$, $g$, then the following estimates holds:
	\begin{gather}    \label{Hardy_1}
	\lV fg\rV_{\dot{H}^N}\lesssim \sum_k 2^{Nk^+}\lV P_k f\rV_{\infty} \lV g\rV_{\dot{H}^N},\\   \label{Hardy_2}
	\lV fg\rV_{\dot{H}^{N}}\lesssim \lV f\rV_{\dot{H}^N} \lV g\rV_{\infty}+\lV f\rV_{\infty} \lV g\rV_{\dot{H}^{N}},\\\label{chazhi}
	\lV f\rV_{L^p}\lesssim \lV f\rV_{L^{\infty}}^{1-2/p}\lV f\rV_{L^2}^{2/p},\ {\rm for}\ 2\leq p\leq \infty.
	\end{gather}
\end{lemma}
\begin{proof}
	The proof of the first two bounds is standard by Littlewood-Paley decomposition, see Lemma 2.5 in \cite{HJLZ-incomp} for detail. The last bound is obtained by H\"{o}lder immediately.
\end{proof}

Finally, we need the following weighted $L^{\infty}-L^2$ estimate and weighted $H^N$ estiamtes.
\begin{lemma}[Lemma 3.3, \cite{Si}]Let $f\in H^2(\R^3)$, $r=|x|$, then there holds
	\begin{equation}   \label{rv}
	\lV \<r\>f\rV_{L^{\infty}}\lesssim \sum_{|\alpha|\leq 1}\lV\d_r\tilde{\Om}^{\alpha}f\rV_{L^2}^{1/2}\sum_{|\alpha|\leq 2}\lV\tilde{\Om}^{\alpha}f\rV_{L^2}^{1/2},
	\end{equation}
	provided the right hand side is finite.
\end{lemma}

\begin{lemma}
	For any $N\geq 0$, we have
	\begin{equation}           \label{xidxif}
	\lV\FF^{-1}(|\xi|\nab_{\xi}\widehat{f})\rV_{H^N}\lesssim \lV x\cdot \nab_x f\rV_{H^N}+\lV \Om f\rV_{H^N}+\lV f\rV_{H^N}.
	\end{equation}
\end{lemma}
\begin{proof}
	This bound is obtained by the relation
	\begin{equation*}
	|\xi|\nab_{\xi}=\frac{\xi}{|\xi|}(\xi\cdot\nab_{\xi})-\frac{\xi}{|\xi|}\wedge \Om(\xi),
	\end{equation*}
	where $\Om(\xi)=\xi\wedge\nab_{\xi}$.
\end{proof}

\subsection{The main bootstrap proposition}
By the similar argument to \cite[Section 2.1]{CW}, applying the vector fields $Z$ to the system (\ref{Main_Sys}), we obtain
\begin{equation}         \label{Main_Sys_VecFie}
\left\{\begin{aligned}
&\d_t Z^a\rr+\nab\cdot Z^a u= F_1^{(a)},\\
&\d_t Z^a u -\tfrac{\mu_1}{2}\Delta (S-1)^{a_1}\Gamma^{a'} u-(\tfrac{\mu_1}{2}+\mu_2)\nab\div(S-1)^{a_1}\Gamma^{a'}u+\tfrac{P'(1+\rr)}{1+\rr} \nab Z^a \rr=F_2^{(a)},\\
&\d_t^2Z^a\phi-\tfrac{1}{1+\rr}\Delta Z^a\phi=F_3^{(a)},
\end{aligned}
\right.
\end{equation}
where the nonlinear terms are
\begin{equation}           \label{F1a}
F_1^{(a)}:=-\sum_{b+c=a}C^b_a (Z^b u\cdot\nab Z^c\rr+Z^b\rr \cdot\nab Z^c u),
\end{equation}
\begin{equation}        \label{F2a}
\begin{aligned}
F_2^{(a)}:=&-\sum_{b+c=a}C^b_a Z^b u\cdot\nab Z^c u-\sum_{b+c=a,|b|\geq 1}C^b_a Z^b\Big(\tfrac{P'(1+\rr)}{1+\rr}\Big)\nab Z^c\rr\\
&-\sum_{b+c=a}C^b_aZ^b \Big(\tfrac{\rr}{1+\rr}\Big)\Big[ \tfrac{\mu_1}{2}\Delta (S-1)^{(c_1)}\Gamma^{(c')} u+(\tfrac{\mu_1}{2}+\mu_2)\nab\div(S-1)^{c_1}\Gamma^{c'}u \Big]\\
&+\sum_{b+c+e=a}C^{b,c}_aZ^b \Big(\tfrac{1}{1+\rr}\Big)\Big[\tfrac{1}{2} \nab\big( \nab(S-1)^{c_1}\Gamma^{c'}\phi\cdot\nab(S-1)^{e_1}\Gamma^{e'}\phi \big)\\
&\ \ \ \ \ \ -\d_j\big(  \nab(S-1)^{c_1}\Gamma^{c'}\phi\cdot\d_j(S-1)^{e_1}\Gamma^{e'}\phi \big) \Big]+(S+1)^{a_1}\Gamma^{a'}Err2,\\
\end{aligned}
\end{equation}
\begin{equation}      \label{F3a}
\begin{aligned}
F^{(a)}_3:=&-\sum_{b+c=a,|b|\geq 1}C^b_a Z^b\Big(\tfrac{1}{1+\rr}\Big)\Delta Z^c \phi -\sum_{b+c=a}C_a^b(\d_t Z^b u\cdot\nab Z^c\phi+2Z^b u\cdot\nab\d_t Z^c\phi)\\
&-\sum_{b+c+e=a}C_a^{b,c}Z^bu\cdot\nab(Z^cu\cdot\nab Z^e\phi)+(S+2)^{a_1}\Gamma^{a'}Err3.
\end{aligned}
\end{equation}
and the constants are
\begin{equation*}
C_a^b:=\tfrac{a!}{b!(a-b)!},\ C_a^{b,c}:=\tfrac{a!}{b!c!(a-b-c)!}.
\end{equation*}
In order to prove the decay estimates of $\nab\rr^{(a)}$ and $\nab u^{(a)}$, we define
\begin{equation*}
v^{(a)}:=|\nab|^{-1}\nab\cdot u^{(a)}=|\nab|^{-1}\nab\cdot Z^a u,\ w^{(a)}:=\P u^{(a)}=\P Z^a u.
\end{equation*}
Then the system (\ref{Main_Sys_VecFie}) can be rewritten as
\begin{equation}         \label{Main_Sys_rr,v,wVecFie}
\left\{\begin{aligned}
&\d_t \rr^{(a)}+|\nab|v^{(a)}= F_1^{(a)},\\
&\d_t v^{(a)}-(\mu_1+\mu_2)\Delta v^{(a)}-a\gamma|\nab| \rr^{(a)}=|\nab|^{-1}\nab\cdot \tilde{F}^{(a)}_2,\\
&\d_t w^{(a)} -\tfrac{\mu_1}{2}\Delta w^{(a)}=\P \tilde{F}_2^{(a)},\\
&\d_t^2\phi^{(a)}-\tfrac{1}{1+\rr}\Delta \phi^{(a)}=F_3^{(a)},
\end{aligned}
\right.
\end{equation}
where the nonlinear term $\tilde{F}_2^{(a)}$ is
\begin{align*}
\tilde{F}_2^{(a)}:=\sum_{l=1}^{a_1}C_{a_1}^l(-1)^l\big( \tfrac{\mu_1}{2}\Delta S^{a_1-l}\Gamma^{a'}u+(\tfrac{\mu_1}{2}+\mu_2)\nab\div S^{a_1-l}\Gamma^{a'}u \big)+(a\gamma-\tfrac{P'(1+\rr)}{1+\rr})\nab\rr^{(a)}+F_2^{(a)}.
\end{align*}

To state the main proposition we define the normalized solution $\Phi$ and its profile $\Psi$ by
\begin{equation*}
\begin{aligned}
\Phi(t):=(\d_t+i|\nab|)\phi(t),\ \  \Psi(t):=e^{-it|\nab|}\Phi(t),
\end{aligned}
\end{equation*}
which will play an important role in the energy estimates. The function $\phi$ can be recovered from the normalized variable $\Phi$ by the formulas
\begin{equation*}
\d_t \phi=(\Phi+\overline{\Phi})/2,\ \ |\nab|\phi=(\Phi-\overline{\Phi})/2i.
\end{equation*}
More generally, we define $\Phi^{(a)}$ and $\Psi^{(a)}$ for $1\leq |a|\leq N_1$ by
\begin{equation}         \label{profile_phia}
\begin{aligned}
\Phi^{(a)}(t):=(\d_t+i|\nab|)\phi^{(a)}(t),\ \ \Psi^{(a)}(t):=e^{-it|\nab|}\Phi^{(a)}(t).
\end{aligned}
\end{equation}

Theorem \ref{Main_thm} is a consequence of the following proposition:
\begin{proposition}   \label{Main_Prop}
    Assume that $(\rr,u,\phi)$ is a solution to (\ref{Main_Sys}) on some time interval $[0,T]$, $T\geq 1$ with initial data satisfying the assumptions (\ref{MainAss_ini}). Assume also that, for any $t\in[0,T]$, the solution satisfies the bootstrap hypothesis
    \begin{gather}           \label{Main_Prop_Ass1}
        \sup\limits_{|a|\leq N_1,t\in[0,T]} \{\lV(\rr^{(a)},  u^{(a)})\rV_{H^{N(a)}}+\lV (\nab\rr^{(a)},\nab u^{(a)})\rV_{L^2([0,t]:H^{N(a)-1}\times H^{N(a)})}     +\<t\>^{-H(a)\delta}\lV \Phi^{(a)}\rV_{H^{N(a)}}\}\leq \epsilon_1,   \\\label{Main_Prop_Ass2}
        \sup_{|a|\leq N_1-1,t\in[0,T]}\<t\>^{-H(|a|+1)\delta}\lV\mathcal{F}^{-1}(|\xi|\d_{\xi}\widehat{\Psi^{(a)}})\rV_{H^{N(|a|+1)}} \leq \ep_1,
    \end{gather}
    where $\ep_1=\ep_0^{2/3}$, $\delta=10^{-15}$,
    \begin{gather*}
    H(0)=0;\ H(a)=1,\ \mathrm{for}\ 1\leq |a|\leq N_1-2;\ H(a)=2,\ \mathrm{for}\ N_1-1\leq|a|\leq N_1.\\
    N(a)=N_0-|a|h,\ N_0=60,\ h=6.
    \end{gather*}
    Then the following improved bounds hold, for any $t\in[0,T]$,
    \begin{gather}\label{Main_Prop_result1}
        \sup\limits_{|a|\leq N_1,t\in[0,T]} \{\lV  (\rr^{(a)},v^{(a)})\rV_{H^{N(a)}}+\lV (\nab\rr^{(a)}, \nab v^{(a)})\rV_{L^2([0,t]:H^{N(a)-1}\times H^{N(a)})} +\<t\>^{-H(a)\delta}\lV \Phi^{(a)}\rV_{H^{N(a)}}\}\lesssim \epsilon_0,   \\\label{Main_Prop_result2}
        \sup_{|a|\leq N_1-1,t\in[0,T]}\<t\>^{-H(|a|+1)\delta}\lV\mathcal{F}^{-1}(|\xi|\d_{\xi}\widehat{\Psi^{(a)}})\rV_{H^{N(|a|+1)}} \lesssim \ep_0.
    \end{gather}
\end{proposition}

The bounds (\ref{Main_Prop_Ass2}) and (\ref{Main_Prop_result2}) are our main $L^2$ bounds on the derivatives of the profile $\Psi^{(a)}$ in the Fourier space. They correspond to weighted bounds in the physical space which play an important role in the energy estimates.

Before we proceed, as in \cite[Corollary 2.8]{HJLZ-incomp}, Lemma \ref{AkBk_lem} and the bootstrap assumptions of Proposition \ref{Main_Prop} yield
\begin{corollary}
	Under the assumptions of Proposition \ref{Main_Prop},  we have
	\begin{equation}       \label{Qjk_Psi}
	\Big(\sum_{(k,j)\in\mathcal{J}}2^{2N(|a|+1)k^++2k}2^{2j} \lV Q_{jk}\Psi^{(a)}\rV_{L^2}^2\Big)^{1/2} \lesssim \ep_1\<t\>^{H(|a|+1)\delta},
	\end{equation}
	for any $|a|\leq N_1-1$.
\end{corollary}

Once Proposition $\ref{Main_Prop}$ is proved, Theorem $\ref{Main_thm}$ follows directly from the standard continuity argument. The rest of this paper focuses on the proof of Proposition $\ref{Main_Prop}$. The key ingredients include Proposition $\ref{Prop_Ene_Sob}-\ref{Prop_Ene_Weiphi}$ and Proposition $\ref{Prop_Psia}$.

\section{Decay of density, velocity field and angles}
In this section, we give the various decay estimates of $\rr$, $u$ and $\Phi$, which will be useful in the energy estimates in the next sections.

\subsection{Decay of $\phi$} In order for the decay estimates of $\phi$, the following frequency localized linear dispersive estimate is necessary.

\begin{lemma}[Lemma 3.1, \cite{HJLZ-incomp}] Frequency localized linear decay estimate] For any $ t\in[1,T]$ and $k\in \Z$, we have
	\begin{equation}    \label{Dis}
	\lV e^{it|\nab|} P_k f\rV_{L^{\infty}}\lesssim \lV \widehat{P_kf}\rV_{L^{\infty}}(t^{-1}2^{2k}+t^{-1+\delta}2^{2k+\delta k})+\lV \nab_{\xi}\widehat{P_kf}\rV_{L^2}(t^{-1}2^{3k/2}+t^{-1+\delta}2^{3k/2+\delta k}).
	\end{equation}
\end{lemma}

\begin{lemma}                \label{Phi_inf}
	Under the bootstrap assumption (\ref{Main_Prop_Ass2}), for any $k\in\Z$, $t\in[0,T]$ and $|a|\leq N_1$ we have
	\begin{gather}     \label{Dec_Phi}
	\lV P_k\Phi^{(a)}(t)\rV_{L^{\infty}} \lesssim \ep_1 \<t\>^{-1+\delta+H(|a|+2)\delta} 2^{k^-/2}2^{-N(|a|+1)k^+ +(h/2+2)k^+},\\ \label{Dec_Phi/nab}
	\lV|\nab|^{-1}\Phi^{(a)}(t)\rV_{L^{\infty}}\lesssim \ep_1\<t\>^{-1/2+3\delta}.
	\end{gather}	
\end{lemma}
\begin{proof}
	On one hand, by (\ref{Dec_fjk}) and (\ref{Qjk_Psi}) when $k>0$,  we have
	\begin{equation}    \label{Psi_k>0}
	\begin{aligned}
	\lV \widehat{P_k\Psi^{(a)}}\rV_{L^{\infty}}\lesssim &\sum_{j\leq hk^+} 2^{3j/2}\lV Q_{jk}\Psi^{(a)}\rV_{L^2}+\sum_{j>hk^+}2^{j/2-k+\beta(j+k)}\sum_{|\alpha|\leq 1}\lV Q_{jk}\tilde{\Om}^{\alpha}\Psi^{(a)}\rV_{L^2}\\
	\lesssim& 2^{hk^+/2}\ep_1\<t\>^{H(|a|+2)\delta}2^{-N(|a|+1)k^+-k}+2^{-hk^+/2}\ep_1\<t\>^{H(|a|+2)\delta}2^{-N(|a|+2)k^+-k}\\
	\lesssim & \ep_1\<t\>^{H(|a|+2)\delta} 2^{-N(|a|+1)k^++hk^+/2-k}.
	\end{aligned}
	\end{equation}
	On the other hand, when $k\leq 0$ we have
	\begin{equation} \label{Psi_k<0}
	\begin{aligned}
	\lV \widehat{P_k\Psi^{(a)}}\rV_{L^{\infty}}\lesssim &\sum_{j\geq -k}2^{j/2-k+\beta(j+k)}\sum_{|\alpha|\leq 1}\lV Q_{jk}\tilde{\Om}^{\alpha}\Psi^{(a)}\rV_{L^2}\\
	\lesssim& \ep_1\<t\>^{H(|a|+2)\delta} 2^{-3k/2}.
	\end{aligned}
	\end{equation}
	Meanwhile, by (\ref{Main_Prop_Ass2}) and the definition of $P_k$, we also have
	\begin{equation} \label{Psi_2}
	2^k\lV \nab_{\xi}\widehat{P_k\Psi^{(a)}}(\xi)\rV_{L^2}= \lV \frac{\xi}{|\xi|}\varphi_k'(\xi) \widehat{\Psi^{(a)}}(\xi)\rV_{L^2}+2^k\lV \varphi_k(\xi)\nab_{\xi}\widehat{\Psi^{(a)}}\rV_{L^2}\lesssim \ep_1 \<t\>^{H(|a|+1)\delta}2^{-N(|a|+1)k^+}.
	\end{equation}
	Thus (\ref{Dec_Phi}) follows from (\ref{Psi_k>0}), (\ref{Psi_k<0}), (\ref{Psi_2}) and (\ref{Dis}).
	
	Next by (\ref{BernIneq}), H\"{o}lder inequality, (\ref{Main_Prop_Ass1}) and (\ref{Dec_Phi}), we have for $p\geq 2$
	\begin{align*}
	\lV|\nab|^{-1}\Phi^{(a)}(t)\rV_{L^{\infty}}&\lesssim \sum_k 2^{3k/p-k}\lV P_k \Phi^{(a)}(t) \rV_{L^p} \lesssim \sum_k 2^{3k/p-k}\lV P_k \Phi^{(a)}(t) \rV_{L^2}^{2/p} \lV P_k \Phi^{(a)}(t) \rV_{L^{\infty}}^{1-2/p}\\
	&\lesssim \ep_1 t^{(-1+\delta+H(|a|+2)\delta)(1-2/p)} \sum_k 2^{(2/p-1/2)k^-}2^{(1-2/p)(-N(|a|+2)+h/2+2)k^++k^+/p}.
	\end{align*}
	Choosing $p$ such that $2/p-1/2=\delta$, the bound (\ref{Dec_Phi/nab}) follows.
\end{proof}

As a consequence of Lemma \ref{Phi_inf}, we have
\begin{corollary}
	With the notations and hypothesis in Proposition \ref{Main_Prop}, we have
	\begin{equation}      \label{phi^2-Hna}
	\sum_{|b|+|c|\leq|a|}\lV\nab\phi^{(b)}\nab\phi^{(c)}\rV_{H^{N(a)}}\lesssim \ep_1^2\min\{\<t\>^{-1+4\delta},\<t\>^{-1+(2+|a|)\delta}\}.
	\end{equation}
\end{corollary}
\begin{proof}
	We first consider the case of $|a|\geq N_1-1$. Without loss of generality, we may assume that $|b|\leq |c|$. By (\ref{Hardy_1}) and (\ref{Dec_Phi}) we have
	\begin{align*}
	\lV\nab\phi^{(b)}\nab\phi^{(c)}\rV_{H^{N(a)}}\lesssim &\sum_{k}2^{N(a)k^+}\lV P_k\nab\phi^{(b)}\rV_{L^{\infty}}\lV \nab\phi^{(c)}\rV_{H^{N(a)}}\\
	\lesssim & \ep_1^2\<t\>^{-1+\delta+H(|b|+2)\delta+H(c)\delta}\sum_k 2^{\tfrac{k^-}{2}+ N(a)k^+-N(|b|+1)k^++(\tfrac{h}{2}+2)k^+}\\
	\lesssim &\ep_1^2\<t\>^{-1+4\delta}.
	\end{align*}
	When $|a|\leq N_1-2$, by (\ref{Hardy_2}) and (\ref{Dec_Phi}), we have
	\begin{align*}
	\lV\nab\phi^{(b)}\nab\phi^{(c)}\rV_{H^{N(a)}}\lesssim & \lV\nab\phi^{(b)}\rV_{H^{N(a)}}\lV\nab\phi^{(c)}\rV_{L^{\infty}}+\lV\nab\phi^{(b)}\rV_{L^{\infty}}\lV\nab\phi^{(c)}\rV_{H^{N(a)}}\\
	\lesssim & \ep_1^2\<t\>^{H(b)\delta-1+\delta+H(|c|+2)\delta}+\ep_1^2\<t\>^{-1+\delta+H(|b|+2)\delta+H(b)\delta},
	\end{align*}
	then we obtain
	\begin{equation*}
	\lV\nab\phi^{(b)}\nab\phi^{(c)}\rV_{H^{N(a)}}\lesssim \left\{\begin{aligned}
	&\ep_1^2\<t\>^{-1+4\delta},\ \ {\rm for\ }|a|=N_1-2,\\
	&\ep_1^2\<t\>^{-1+(2+|a|)\delta},\ \ {\rm for\ }|a|< N_1-2,
	\end{aligned}
	\right.
	\end{equation*}
	which is acceptable. This concludes the bound (\ref{phi^2-Hna}).
\end{proof}

\subsection{Basic decay estimates of density $\rr$ and velocity field $u$} Before proving the decay estimates of $\rr$ and $u$, we need the following bilinear estimates.
\begin{lemma}[Bilinear estimates]    \label{biL-lem}
	Under the bootstrap assumption (\ref{Main_Prop_Ass1}), let $G(\rr,\phi)$ be smooth function, and $|\d_{\rr}^{k}\d_{\phi}^l G|\leq C_{k,l}$ for $k+l\leq N$, $4\leq N\leq N(0)$, let $f$ and $g$ be Schwartz functions, $\gamma\geq 1$, we have
	\begin{align}       \label{biL-Ffg-WN1}
	&\lV G(\rr,\phi)fg\rV_{W^{N,1}}\lesssim \lV f\rV_{H^N}\lV g\rV_{H^N},\\        \label{biL-Ffg-HN}
	&\lV G(\rr,\phi)fg\rV_{H^{N}}\lesssim\lV fg\rV_{H^N}\lesssim  \lV f\rV_{H^N}\lV g\rV_{H^N},\\ \label{biL-rr}
	&\lV\big((1+\rr)^{\gamma-2}-1\big)\nab\rr\rV_{W^{N,1}\cap H^N}\lesssim \lV \rr\rV_{H^N}\lV\nab\rr\rV_{H^N},\\ \label{biL-rr-HN}
	&\lV \big((1+\rr)^{\gamma-2}-1\big)\nab\rr\rV_{H^N}\lesssim \lV \rr\rV_{W^{N/2+1,\infty}}\lV\nab\rr\rV_{H^N}.
	\end{align}
\end{lemma}
\begin{proof}
	The bound (\ref{biL-Ffg-WN1}) and (\ref{biL-Ffg-HN}) are obtained by H\"{o}lder, Sobolev embedding, and the bound
	\begin{equation*}
	\sum_{|n|\leq N/2}\lV \d^n G(\rr,\phi)\rV_{L^{\infty}}+\sum_{N/2<|n|\leq N}\lV \d^n G(\rr,\phi)\rV_{L^2}\lesssim (1+\lV\rr\rV_{H^N}+\lV\phi\rV_{H^N})^N\lesssim 1,
	\end{equation*}
	which follows from Sobolev embedding, $|\d_{\rr}^{k}\d_{\phi}^l G|\leq C_{k,l}$, and (\ref{Main_Prop_Ass1}).
	
	Next, we prove (\ref{biL-rr}) and (\ref{biL-rr-HN}). By the equality
	\begin{equation*}
	(1+\rr)^{\gamma-2}-1=(\gamma-2)\int_0^{\rr}(1+s)^{\gamma-3} ds
	\end{equation*}
	 and (\ref{Main_Prop_Ass1}), we get
	\begin{equation*}
	\lV (1+\rr)^{\gamma-2}-1\rV_{L^p}\lesssim \lV \rr\rV_{L^p}\lV (1+|\rr|)^{\gamma-3}\rV_{L^{\infty}}\lesssim \lV \rr\rV_{L^p},\ {\rm for }\ 1\leq p\leq \infty.
	\end{equation*}
	Moreover, by Sobolev embedding and (\ref{Main_Prop_Ass1}) we have
	\begin{align*}
	\sum_{1\leq |n|\leq N}\lV\d^n [(1+\rr)^{\gamma-2}-1]\rV_{L^2}\lesssim & \sum_{1\leq |n|\leq N}\sum_{\substack{n_1+\cdots+n_s=n,\\|n_i|\geq 1}}\lV(1+\rr)^{\gamma-2-s}\d^{n_1}\rr\cdots\d^{n_s}\rr\rV_{L^2}\\
	\lesssim &\lV \nab\rr\rV_{H^{N-1}}+\lV \nab\rr\rV_{H^{N-1}}^N\lesssim \lV \nab\rr\rV_{H^{N-1}},
	\end{align*}
	and
	\begin{equation*}
	\sum_{1\leq |n|\leq N/2}\lV\d^n [(1+\rr)^{\gamma-2}-1]\rV_{L^{\infty}}\lesssim \lV \rr\rV_{W^{N/2,\infty}}+\lV \rr\rV_{W^{N/2,\infty}}^{N/2}\lesssim \lV \rr\rV_{W^{N/2,\infty}}.
	\end{equation*}
	Then the bounds (\ref{biL-rr}) and (\ref{biL-rr-HN}) follow by H\"{o}lder, Sobolev embedding, and the above three bounds.

\end{proof}

We then prove various bounds on the nonlinearities $F_1$ and $F_2$. Indeed the nonlinearities $F_2$ can be rewritten as
\begin{equation*}
F_2=F_{2,I}+F_{2,II},
\end{equation*}
where
\begin{align*}
F_{2,I}:=&-u\cdot\nab u+(a\gamma-\frac{P'(1+\rr)}{1+\rr})\nab\rr-\frac{\rr}{1+\rr}(\frac{\mu_1}{2}\Delta u+(\frac{\mu_1}{2}+\mu_2)\nab\div u)\\
&-\frac{1}{(1+\rr)^2}[\frac{1}{2}\nab\rr(\cos^2\phi_2|\nab\phi_1|^2+|\nab\phi_2|^2)-\d_j \rr(\cos^2\phi_2\nab\phi_1\d_j\phi_1+\nab\phi_2\d_j\phi_2)],\\
F_{2,II}:=&\nab[\frac{1}{2(1+\rr)}(\cos^2\phi_2|\nab\phi_1|^2+|\nab\phi_2|^2)]-\d_j[\frac{1}{1+\rr} (\cos^2\phi_2\nab\phi_1\d_j\phi_1+\nab\phi_2\d_j\phi_2)].
\end{align*}

\begin{corollary}     \label{Non-Cor}
	Under the bootstrap assumptions (\ref{Main_Prop_Ass1}) and (\ref{Main_Prop_Ass2}), for any $N(0)-3\leq N\leq N(0)-1$, $t\in[0,T]$, we have
	\begin{align}           \label{F1,F2I-WN1}
		&\lV F_1 \rV_{W^{N,1}}+\lV F_{2,I}\rV_{W^{N,1}}\lesssim \ep_1(\lV \nab\rr\rV_{H^N}+\lV \nab u\rV_{H^{N+1}}),\\\label{F1,F2I-HN}
		&\lV F_1 \rV_{H^{N}}+\lV F_{2,I}\rV_{H^{N}}\lesssim (\lV(\rr,u)\rV_{W^{N(0)-4,\infty}}+\ep_1^2\<t\>^{-1+2\delta})(\lV \nab\rr\rV_{H^N}+\lV \nab u\rV_{H^{N+1}}),\\\label{F2II-HN}
		&\lV F_{2,II}\rV_{H^{N(0)-1}}\lesssim \lV\nab\phi\rV_{H^{N(0)}}\lV\nab\phi\rV_{L^{\infty}}\lesssim \ep_1^2\<t\>^{-1+2\delta}.
	\end{align}
\end{corollary}
\begin{proof}
	In view of $F_1$, $F_{2,I}$ and $F_{2,II}$, the Corollary is an consequence of Lemma \ref{biL-lem}, (\ref{Dec_Phi}) and the bootstrap assumption (\ref{Main_Prop_Ass1}).
	
	For the first two bounds, we only estimate the following two terms in detail
	\begin{equation*}
	\tfrac{\rr}{1+\rr}\nab^2 u,\ \ \ \ \tfrac{1}{(1+\rr)^2}\nab\rr(\cos^2\phi_2|\nab\phi|^2).
	\end{equation*}
	because other terms can be estiamted similarly. By (\ref{biL-Ffg-WN1}) and (\ref{Main_Prop_Ass1}), we have
	\begin{equation*}
	\lV \tfrac{1}{1+\rr}\rr\nab^2 u \rV_{W^{N,1}}\lesssim \lV \rr\rV_{H^N}\lV\nab^2 u\rV_{H^N}\lesssim \ep_1\lV\nab u\rV_{H^{N+1}},
	\end{equation*}
	and by (\ref{biL-Ffg-HN}) and (\ref{Hardy_2}) we have
	\begin{equation*}
	\lV \tfrac{1}{1+\rr}\rr\nab^2 u \rV_{H^N}\lesssim \lV \rr\nab^2 u\rV_{H^N}\lesssim \lV\rr\rV_{L^{\infty}}\lV\nab^2 u\rV_{H^N}+\lV\rr\rV_{\dot{H}^N}\lV \nab^2 u\rV_{L^{\infty}}\lesssim \lV(\rr,u)\rV_{W^{2,\infty}}(\lV\nab\rr\rV_{H^N}+\lV\nab u\rV_{H^{N+1}}).
	\end{equation*}
	These are acceptable for (\ref{F1,F2I-WN1}) and (\ref{F1,F2I-HN}). From (\ref{biL-Ffg-WN1}), (\ref{biL-Ffg-HN}) and (\ref{Dec_Phi}), we can also bound the second term by
	\begin{align*}
	\lV \tfrac{1}{(1+\rr)^2}\nab\rr(\cos^2\phi_2|\nab\phi|^2)\rV_{W^{N,1}\cap H^N}
	\lesssim &\lV \nab\rr\rV_{H^N}\lV \cos^2\phi_2|\nab\phi|^2\rV_{H^N}\\
	\lesssim & \lV \nab\rr\rV_{H^N}\lV |\nab\phi|^2\rV_{H^N}\\
	\lesssim &\lV \nab\rr\rV_{H^N}\lV \nab\phi\rV_{H^N}\lV\nab\phi\rV_{L^{\infty}}\\
	\lesssim &\ep_1^2\<t\>^{-1+2\delta}\lV\nab\rr\rV_{H^N},
	\end{align*}
	which is acceptable for (\ref{F1,F2I-WN1}) and (\ref{F1,F2I-HN}).
	
	Finally, the bound (\ref{F2II-HN}) can be easily obtained by (\ref{biL-Ffg-HN}), (\ref{Hardy_2}), and (\ref{Dec_Phi}). This completes the proof of the lemma.

\end{proof}

As a consequence of Corollary \ref{Non-Cor} and linear decay estimates (\ref{Dec_heat}), (\ref{e-tB_Lin}), and (\ref{etB_L2-L2}) we get the following decay estimates of $\rr,v$, and $w$.

\begin{lemma}         \label{rr,v,w-a=0_Lem}
	With the hypothesis in Proposition \ref{Main_Prop}, let $p=\frac{3}{2\delta}$, for any $t\in[0,T]$, we have
	\begin{align}    \label{rr,v,w_inf}
	&\lV (\rr,v,w)(t)\rV_{W^{N(0)-3,p}}\lesssim \ep_1\<t\>^{-3/4+\delta},\\ \label{drr,v,w_L2dec}
	&\lV\nab(\rr,v,w)\rV_{H^{N(0)-2}}\lesssim \ep_1\<t\>^{-1/2},\\\label{d2rr,v,w_L2dec}
	&\lV\nab^2(\rr,v,w)\rV_{H^{N(0)-5}}\lesssim \ep_1\<t\>^{-1+2\delta}.
	\end{align}
\end{lemma}
\begin{proof}
	For $0\leq t\leq 1$, (\ref{rr,v,w_inf})-(\ref{d2rr,v,w_L2dec}) are the consequences of Sobolev embedding and (\ref{Main_Prop_Ass1}). Now it suffices to prove the bounds (\ref{rr,v,w_inf})-(\ref{d2rr,v,w_L2dec}) for $t> 1$.
	
	\textbf{Proof of (\ref{rr,v,w_inf}) for $t>1$.} From (\ref{Main_Sys_rr,v,w}), it suffices to prove
	\begin{equation}
	\lV e^{tB} (\rr_0,v_0)^{\top}\rV_{W^{N(0)-3,p}}+\lV e^{t\Delta} w_0\rV_{W^{N(0)-3,p}}\lesssim \ep_0 t^{-3/4+\delta} \label{rr,v,w_inf_Lin}
	\end{equation}
	and
	\begin{equation}
	\begin{aligned}
	&\lV \int_0^t  e^{(t-s)B} (F_1,\frac{\nab\cdot F_2}{|\nab|})^{\top} ds\rV _{W^{N(0)-3,p}}+\lV \int_0^t  e^{(t-s)\Delta} \P F_2 ds \rV_{W^{N(0)-3,p}}\\
	\lesssim &\ep_1^2 t^{-1+2\delta}+\ep_1 \sup_{s\in[t/2,t]}\lV (\rr,v,w)(s)\rV_{W^{N(0)-3,p}},\label{rr,v,w_inf_NonLin}
	\end{aligned}
	\end{equation}
	then, by Duhamel's formula and (\ref{rr,v,w_inf_Lin}) and  (\ref{rr,v,w_inf_NonLin}) we get
	\begin{align*}
	&\sup_{t\in[1,T]}\<t\>^{3/4-\delta}\lV  (\rr,v,w)\rV_{W^{N(0)-3,p}}\lesssim \ep_1 +\ep_1\sup_{t\in[1,T]}\<t\>^{3/4-\delta}\lV  (\rr,v,w)\rV_{W^{N(0)-3,p}},
	\end{align*}
	which implies (\ref{rr,v,w_inf}) immediately.
	
	Now we prove the bounds (\ref{rr,v,w_inf_Lin}) and (\ref{rr,v,w_inf_NonLin}), respectively. The bound (\ref{rr,v,w_inf_Lin}) follows from (\ref{chazhi}), (\ref{e-tB_Lin}), (\ref{Dec_heat}) and (\ref{MainAss_ini}). (\ref{rr,v,w_inf_NonLin}) can be divided into two parts. By (\ref{e-tB_Lin}) and (\ref{Dec_heat}), we have
	\begin{equation*}        \label{Dec_rr_0-t/2}
	\begin{aligned}
	I_{[0,\tfrac{t}{2}]}:=&\lV \int_0^{t/2}  e^{(t-s)B} (F_1,\frac{\nab\cdot F_2}{|\nab|})^{\top} +e^{(t-s)\Delta} \P F_2 ds\rV _{W^{N(0)-3,p}}\\
	\lesssim & \int_0^{t/2} (t-s)^{-3/2+\delta} \lV (F_1, F_2^I)\rV_{W^{N(0)-3,1}}+e^{-c(t-s)} \lV(F_1, F_2)\rV_{H^{N(0)-1} }\\
	&+(t-s)^{-2+\delta} \lV \frac{1+\cos^2 \phi_2}{1+\rr}\nab\phi\nab\phi\rV_{W^{N(0)-3,1}}ds.
	\end{aligned}
	\end{equation*}
    Then (\ref{F1,F2I-WN1}), (\ref{F1,F2I-HN}) and (\ref{biL-Ffg-WN1}) imply
	\begin{equation}      \label{N1}
	\begin{aligned}
	I_{[0,\tfrac{t}{2}]}\lesssim &\int_0^{t/2} (t-s)^{-3/2+\delta} \ep_1 (\lV \nab\rr\rV_{H^{N(0)-1}}+\lV \nab u\rV_{H^{N(0)}})+e^{-c(t-s)} (\ep_1^2 +\ep_1 \lV\nab u\rV_{H^{N(0)}} )+\ep_1^2(t-s)^{-2+\delta}ds\\
	\lesssim & \ep_1^2 t^{-1+\delta}.
	\end{aligned}
	\end{equation}
	For the other part, by (\ref{chazhi}), (\ref{e-tB_Lin}), (\ref{Dec_heat}) and (\ref{BernIneq}), we have
	\begin{equation*}           \label{Dec_rr_t/2-t}
	\begin{aligned}
	I_{[\tfrac{t}{2},t]}:=&\lV \int_{t/2}^t  e^{(t-s)B} (F_1,\frac{\nab\cdot F_2}{|\nab|})^{\top}+e^{(t-s)\Delta}\P F_2 ds\rV _{W^{N(0)-3,p}}\\
	\lesssim & \int_{t/2}^{t-1} (t-s)^{-3/4+\delta} \lV(F_1, F_2^I)\rV_{H^{N(0)-1}}+ (t-s)^{-5/4+5\delta/3} \lV \frac{1+\cos^2 \phi_2}{1+\rr}\nab\phi\nab\phi\rV_{H^{N(0)-1}}\\
	&+e^{-c(t-s)} \lV(F_1, F_2)\rV_{H^{N(0)-1}}ds+\int_{t-1}^t \lV (F_1,F_2)\rV_{H^{N(0)-1}} ds
	\end{aligned}
	\end{equation*}
	Then by (\ref{F1,F2I-HN}), (\ref{F2II-HN}) and (\ref{Main_Prop_Ass1}), we get
	\begin{equation}       \label{N2}
	\begin{aligned}
	I_{[\tfrac{t}{2},t]}\lesssim &\int_{t/2}^{t-1} (t-s)^{-3/4+\delta} (\lV (\rr,u)\rV_{W^{N(0)-3,p}}+\ep_1^2 s^{-1+2\delta})(\lV \nab\rr\rV_{H^{N(0)-1}}+\lV \nab u\rV_{H^{N(0)}})\\
	&+  (t-s)^{-5/4+5\delta/3}\ep_1^2\<s\>^{-1+2\delta}+ e^{-c(t-s)} [\lV(\rr,u)\rV_{W^{N(0)-3,p}}(\lV\nab\rr\rV_{H^{N(0)-1}}+\lV\nab u\rV_{H^{N(0)}})+\ep_1^2\<s\>^{-1+2\delta}]ds\\
	&+\int_{t-1}^t \lV(\rr,u)\rV_{W^{N(0)-3,p}}(\lV\nab\rr\rV_{H^{N(0)-1}}+\lV\nab u\rV_{H^{N(0)}})+\ep_1^2\<s\>^{-1+2\delta}ds\\
	\lesssim & \ep_1\sup_{s\in[t/2,t]}\lV (\rr,u)(s)\rV_{W^{N(0)-3,p}}+\ep_1^2\<t\>^{-1+2\delta}.
	\end{aligned}
	\end{equation}
	Hence, the desired bound (\ref{rr,v,w_inf_NonLin}) follows from (\ref{N1}) and (\ref{N2}). Thus the proof of (\ref{rr,v,w_inf}) completes.

	\textbf{Proof of (\ref{drr,v,w_L2dec}) for $t>1$.} By Duhamel's formula, it suffices to prove that
	\begin{equation}         \label{drr,v,w_L2dec_Lin}
	\lV e^{tB}|\nab|(\rr_0,v_0)\rV_{H^{N(0)-2}}+\lV e^{t\Delta}|\nab| w_0\rV_{H^{N(0)-2}} \lesssim \ep_0t^{-1/2},
	\end{equation}
	and
	\begin{equation}      \label{drr,v,w_L2dec_NonLin}
	\int_0^t \lV e^{(t-s)B}(|\nab|F_1,\nab\cdot F_2)^{\top}\rV_{H^{N(0)-2}}+\lV e^{(t-s)\Delta}|\nab|\P F_2 \rV_{H^{N(0)-2}} ds\lesssim \ep_1 t^{-3/4+2\delta}.
	\end{equation}
	The bound (\ref{drr,v,w_L2dec_Lin}) follows by (\ref{etB_L2-L2}), (\ref{Dec_heat}) and (\ref{MainAss_ini}). Thus we only show (\ref{drr,v,w_L2dec_NonLin}). Using (\ref{etB_L2-L2}), (\ref{biL-Ffg-WN1}) and (\ref{biL-Ffg-HN}), it follows that
	\begin{align*}
	&\int_0^{t/2} \lV e^{(t-s)B}(|\nab|F_1,\nab\cdot F_2)^{\top}\rV_{H^{N(0)-2}} ds\\
	\lesssim & \int_0^{t/2} (t-s)^{-5/4} \lV (F_1,F_2^I)\rV_{W^{N(0)-2,1}}+(t-s)^{-7/4}\lV\nab\phi\rV_{H^{N(0)-2}}^2+e^{-c(t-s)}\lV(F_1,F_2)\rV_{H^{N(0)-1}} ds,
	\end{align*}
	and
	\begin{align*}
	&\int_{t/2}^t \lV e^{(t-s)B}(|\nab|F_1,\nab\cdot F_2)^{\top}\rV_{H^{N(0)-2}} ds\\
	\lesssim & \int_{t/2}^t \<t-s\>^{-1/2}\lV(F_1,F_2^I)\rV_{H^{N(0)-1}}+\<t-s\>^{-1}\lV\nab\phi\rV_{H^{N(0)}}^2 +e^{-c(t-s)}\lV(F_1,F_2)\rV_{H^{N(0)-1}} ds.
	\end{align*}
	Then by (\ref{F1,F2I-WN1}), (\ref{F1,F2I-HN}), (\ref{F2II-HN}), (\ref{rr,v,w_inf}), H\"{o}lder and (\ref{Main_Prop_Ass1}) we get
	\begin{align*}
	&\int_0^t \lV e^{(t-s)B}(|\nab|F_1,\nab\cdot F_2)^{\top}\rV_{H^{N(0)-2}} ds\\
	\lesssim & \int_0^{t/2} (t-s)^{-5/4}\ep_1(\lV \nab\rr\rV_{H^{N(0)-2}}+\lV\nab u\rV_{H^{N(0)-1}})+(t-s)^{-7/4}\ep_1^2 +e^{-c(t-s)}(\ep_1^2+ \ep_1\lV\nab u\rV_{H^{N(0)}}) ds\\
	&+\int_{t/2}^t \Big[\<t-s\>^{-1/2}\ep_1 s^{-3/4+\delta}(\lV \nab\rr\rV_{H^{N(0)-1}}+\lV\nab u\rV_{H^{N(0)}}) +\<t-s\>^{-1}\ep_1^2 s^{-1+2\delta} \\
	&\ \ \ \ \ \ \ \ \ \ \ \ +e^{-c(t-s)}(\ep_1 s^{-3/4+\delta}(\lV \nab\rr\rV_{H^{N(0)-1}}+\lV\nab u\rV_{H^{N(0)}})+\ep_1^2 s^{-1+2\delta})\Big]ds\\
	\lesssim & \ep_1^2 t^{-3/4}+\ep_1^2 t^{-2}+\ep_1^2t^{-3/4+2\delta}+\ep_1^2 t^{-1+3\delta}+\ep_1^2 t^{-3/4}\\
	\lesssim & \ep_1^2 t^{-3/4+2\delta}.
	\end{align*}
	Hence, the proof of (\ref{drr,v,w_L2dec}) for $\rr,v$ is obtained. Using similar argument, we can also get the bound (\ref{drr,v,w_L2dec}) for $w$. This completes the proof of (\ref{drr,v,w_L2dec}).
	
	\textbf{Proof of (\ref{d2rr,v,w_L2dec}) for $t>1$.} By Duhamel's formula, it suffices to prove
	\begin{gather}         \label{d2rr,v,w_L2dec_Lin}
	\lV e^{tB}\nab^2(\rr_0,v_0)\rV_{H^{N(0)-5}}+\lV e^{t\Delta}\nab^2 w_0\rV_{H^{N(0)-5}} \lesssim \ep_0t^{-1},\\
	     \label{d2rr,v,w_L2dec_NonLin}
	\int_0^t \lV e^{(t-s)B}\nab^2(F_1,|\nab|^{-1}\nab\cdot F_2)^{\top}\rV_{H^{N(0)-5}}+\lV e^{(t-s)\Delta}\nab^2\P F_2 \rV_{H^{N(0)-5}} ds\lesssim \ep_1 t^{-1+2\delta}.
	\end{gather}
	From (\ref{etB_L2-L2}), (\ref{Dec_heat}) and (\ref{MainAss_ini}), we obtain (\ref{d2rr,v,w_L2dec_Lin}). For (\ref{d2rr,v,w_L2dec_NonLin}), by (\ref{etB_L2-L2}) and (\ref{Dec_heat}) we have
	\begin{equation*}             \label{d2rrL2_0,t/2}
	\begin{aligned}
	\mathcal{I}_{[0,\tfrac{t}{2}]}:=&\int_0^{t/2} \lV e^{(t-s)B}\nab^2(F_1,|\nab|^{-1}\nab\cdot F_2)^{\top}\rV_{H^{N(0)-5}}+\lV e^{(t-s)\Delta}\nab^2\P F_2 \rV_{H^{N(0)-5}} ds\\
	\lesssim & \int_0^{t/2} (t-s)^{-7/4}\lV(F_1,F_{2,I})\rV_{W^{N(0)-5,1}}+(t-s)^{-9/4}\lV\nab\phi\rV_{H^{N(0)}}^2
	+e^{-c(t-s)}\lV (F_1,F_2)\rV_{H^{N(0)-3}}ds.
	\end{aligned}
	\end{equation*}
	Thus by Corollary \ref{Non-Cor}, H\"{o}lder and (\ref{Main_Prop_Ass1}) we get
	\begin{equation}         \label{d2rrL2_0,t/2-}
	\begin{aligned}
	\mathcal{I}_{[0,\tfrac{t}{2}]}\lesssim \int_0^{t/2} \ep_1t^{-7/4}(\lV\nab\rr\rV_{H^{N(0)-2}}+\lV\nab u\rV_{H^{N(0)-1}})+\ep_1^2t^{-9/4}+\ep_1^2 e^{-c(t-s)}ds
	\lesssim \ep_1^2t^{-5/4}.
	\end{aligned}
	\end{equation}
	Next, by (\ref{etB_L2-L2}) and (\ref{Dec_heat}) we have
	\begin{equation*}         \label{d2rrL2_t/2,t}
	\begin{aligned}
	\mathcal{I}_{[\tfrac{t}{2},t]}:=&\int_{t/2}^t \lV e^{(t-s)B}\nab^2(F_1,|\nab|^{-1}\nab\cdot F_2)^{\top}\rV_{H^{N(0)-5}}+\lV e^{(t-s)\Delta}\nab^2\P F_2 \rV_{H^{N(0)-5}} ds\\
	\lesssim & \int_{t/2}^t \<t-s\>^{-1}\lV (F_1,F_{2,I})\rV_{H^{N(0)-3}}+\<t-s\>^{-3/2}\lV (\nab\phi)^2\rV_{H^{N(0)-2}}+e^{-c(t-s)}\lV(F_1,F_2)\rV_{H^{N(0)-3}}ds.
	\end{aligned}
	\end{equation*}	
	Since (\ref{F1,F2I-HN}), (\ref{rr,v,w_inf}) and (\ref{drr,v,w_L2dec}) give
	\begin{align*}
	\lV (F_1,F_{2,I})\rV_{H^{N(0)-3}}\lesssim\ep_1^2\<t\>^{-5/4+\delta},
	\end{align*}
	by (\ref{F2II-HN}) and (\ref{Dec_Phi}) we can bound the left hand side of (\ref{d2rrL2_t/2,t}) by
	\begin{equation}        \label{d2rrL2_t/2,t-}
	\begin{aligned}
	\mathcal{I}_{[\tfrac{t}{2},t]}\lesssim &\int_{t/2}^t\<t-s\>^{-1}\ep_1^2\<s\>^{-5/4+\delta}+\<t-s\>^{-3/2}\ep_1^2\<s\>^{-1+2\delta}+e^{-c(t-s)}\ep_1^2\<t\>^{-1+2\delta}ds\lesssim \ep_1^2 \<t\>^{-1+2\delta}.
	\end{aligned}
	\end{equation}
	Hence, (\ref{d2rrL2_0,t/2-}) and (\ref{d2rrL2_t/2,t-}) imply the bound (\ref{d2rr,v,w_L2dec_NonLin}). This completes the proof of (\ref{d2rr,v,w_L2dec}).	
\end{proof}

We prove now the decay estimates of $\rr^{(a)},u^{(a)}$ for $1\leq |a|\leq N_1$ by induction. From the system (\ref{Main_Sys_rr,v,wVecFie}), we rewrite the nonlinearities $\tilde{F}_2^{(a)}$ as
\begin{equation*}
\tilde{F}_2^{(a)}=\tilde{F}_{2,I}^{(a)}+\tilde{F}_{2,II}^{(a)}+\tilde{F}_{2,III}^{(a)}+(S+1)^{a_1}\Gamma^{a'}Err2,
\end{equation*}
where
\begin{align*}
\tilde{F}_{2,I}^{(a)}:=&\sum_{l=1}^{a_1}C_{a_1}^l(-1)^l\big( \frac{\mu_1}{2}\Delta S^{a_1-l}\Gamma^{a'}u+(\frac{\mu_1}{2}+\mu_2)\nab\div S^{a_1-l}\Gamma^{a'}u \big),\\
\tilde{F}_{2,II}^{(a)}:=&(a\gamma-\frac{P'(1+\rr)}{1+\rr})\nab\rr^{(a)}-\sum_{b+c=a}C^b_a Z^b u\cdot\nab Z^c u-\sum_{b+c=a,|b|\geq 1}C^b_a Z^b\Big(\frac{P'(1+\rr)}{1+\rr}\Big)\nab Z^c\rr\\
&+\sum_{b+c=a}C^b_a  \nab Z^b \Big(\frac{\rr}{1+\rr}\Big)\Big[\frac{\mu_1}{2} \nab (S-1)^{(c_1)}\Gamma^{(c')} u+(\frac{\mu_1}{2}+\mu_2)\div(S-1)^{c_1}\Gamma^{c'}u \Big]\\
&-\sum_{b+c+e=a}C^{b,c}_a\Big[\frac{1}{2} \nab Z^b \Big(\frac{1}{1+\rr}\Big) \big( \nab(S-1)^{c_1}\Gamma^{c'}\phi\cdot\nab(S-1)^{e_1}\Gamma^{e'}\phi \big)\\
&\ \ \ \ \ \ -\d_j Z^b \Big(\frac{1}{1+\rr}\Big)\big(  \nab(S-1)^{c_1}\Gamma^{c'}\phi\cdot\d_j(S-1)^{e_1}\Gamma^{e'}\phi \big) \Big],
\end{align*}
and
\begin{align*}
\tilde{F}_{2,III}^{(a)}:=& -\sum_{b+c=a}C^b_a \nab \Big[Z^b \Big(\frac{\rr}{1+\rr}\Big)\Big(\frac{\mu_1}{2} \nab (S-1)^{(c_1)}\Gamma^{(c')} u+(\frac{\mu_1}{2}+\mu_2)\div Z^b \Big(\frac{\rr}{1+\rr}\Big)\div(S-1)^{c_1}\Gamma^{c'}u \Big)\Big]\\
&+\sum_{b+c+e=a}C^{b,c}_a \Big[\frac{1}{2} \nab \Big( Z^b \Big(\frac{1}{1+\rr}\Big) \nab(S-1)^{c_1}\Gamma^{c'}\phi\cdot\nab(S-1)^{e_1}\Gamma^{e'}\phi \Big)\\
&\ \ \ \ \ \ -\d_j \Big( Z^b \big(\frac{1}{1+\rr}\big)\cdot  \nab(S-1)^{c_1}\Gamma^{c'}\phi\cdot\d_j(S-1)^{e_1}\Gamma^{e'}\phi \Big) \Big].
\end{align*}

\begin{lemma}         \label{F1a,F2a_Lem}
	Under the bootstrap assumptions (\ref{Main_Prop_Ass1}) and (\ref{Main_Prop_Ass2}), for any $N(a)-3\leq N\leq N(a)-1$, $t\in[0,T]$, we have
	\begin{align*}
	\lV \tilde{F}_1^{(a)}\rV_{W^{N,1}}+\lV \tilde{F}_{2,II}^{(a)}\rV_{W^{N,1}}\lesssim & \ep_1\sum_{|b|\leq |a|}(\lV \nab\rr^{(b)}\rV_{H^{N}}+\lV \nab u^{(b)}\rV_{H^{N}}),\\
	\lV \tilde{F}_1^{(a)}\rV_{H^{N}}+\lV \tilde{F}_{2,II}^{(a)}\rV_{H^N}\lesssim & \lV(\rr^{(b)},u^{(b)})\rV_{W^{N(a)/2,\infty}}\lV (\nab\rr^{(c)},\nab u^{(c)})\rV_{H^N}\\
	&+\ep_1^2\<t\>^{-1+5\delta}\lV \nab\rr^{(b)}\rV_{H^N}.
	\end{align*}
\end{lemma}

\begin{lemma} With the hypothesis in Proposition \ref{Main_Prop}, let $p=\frac{3}{2\delta}$, for any $t\in[0,T]$, we have
	\begin{align} \label{rra,v,w_inf}
	&\lV  (\rr^{(a)},v^{(a)},w^{(a)})(t)\rV_{W^{N(a)-3,p}}\lesssim \ep_1\<t\>^{-3/4+(2+|a|)\delta},\ \ \ {\rm for}\ 1\leq|a|\leq N_1.\\	\label{drra,v,w_L2dec}
	&\lV\nab(\rr^{(a)},v^{(a)},w^{(a)})\rV_{H^{N(a)-2}}\lesssim \ep_1\<t\>^{-1/2+(2+|a|)\delta},\\\label{d2rra,v,w_L2dec}
	&\lV\nab^2(\rr^{(a)},v^{(a)},w^{(a)})\rV_{H^{N(a)-5}}\lesssim \ep_1\<t\>^{-1+(2+|a|)\delta}.
	\end{align}
\end{lemma}
\begin{proof}
	We prove the bounds (\ref{rra,v,w_inf})-(\ref{d2rra,v,w_L2dec}) by induction. Precisely, by Lemma \ref{rr,v,w-a=0_Lem}, we may assume that the bounds (\ref{rra,v,w_inf})-(\ref{d2rra,v,w_L2dec}) hold for $|a|\leq l-1\leq N_1-1$, then we prove (\ref{rra,v,w_inf})-(\ref{d2rra,v,w_L2dec}) for $|a|=l$. Now it suffices to prove these bounds for $t>1$, otherwise by (\ref{BernIneq}) and (\ref{Main_Prop_Ass1}) the bounds (\ref{rra,v,w_inf})-(\ref{d2rra,v,w_L2dec}) for $t\leq 1$ follow immediately.
	
	{\bf Proof of (\ref{rra,v,w_inf}) for $t>1$.}
	
	By Duhamel's formula, it suffices to prove
	\begin{gather}       \label{rra,v,w_inf_Lin}
	\lV e^{tB} (\rr^{(a)}_0,v^{(a)}_0)^{\top}\rV_{W^{N(a)-3,p}}+\lV  e^{t\Delta} w^{(a)}_0\rV_{W^{N(a)-3,p}}\lesssim \ep_0 t^{-3/4+\delta},  \\    \label{rra,v,w_inf_Lin2}
	\sum_{|b|\leq |a|-1}(\lV \int_0^t  e^{(t-s)B} (0,\Delta u^{(b)})^{\top} ds\rV _{W^{N(a)-3,p}}+\lV \int_0^t  e^{(t-s)\Delta} \P \Delta u^{(b)} ds \rV_{W^{N(a)-3,p}})\lesssim \ep_1 t^{-3/4+(2+|b|)\delta},
	\end{gather}
	\begin{equation}
	\begin{aligned}
	&\lV \int_0^t  e^{(t-s)B} (F_1^{(a)}, \tilde{F}_{2,II}^{(a)})^{\top} ds\rV _{W^{N(a)-3,p}}+\lV \int_0^t  e^{(t-s)\Delta}  \tilde{F}_{2,II}^{(a)} ds \rV_{W^{N(a)-3,p}}\\
	\lesssim &\ep_1^2 t^{-3/4+(2+|a|)\delta}+\ep_1 \sup_{s\in[t/2,t]}\lV (\rr^{(a)},v^{(a)},w^{(a)})(s)\rV_{W^{N(a)/2,p}},\label{rra,v,w_inf_NonLin}
	\end{aligned}
	\end{equation}
	\begin{equation}        \label{rra,v,w_inf_F2III}
	\begin{aligned}
	&\int_0^t\lV   e^{(t-s)B} (0, \tilde{F}_{2,III}^{(a)})^{\top} \rV _{W^{N(a)-3,p}}+\lV   e^{(t-s)\Delta}  \tilde{F}_{2,III}^{(a)} \rV_{W^{N(a)-3,p}} ds\\
	\lesssim & \ep_1\<t\>^{-3/4+(2+|a|)\delta}+\ep_1\sup_{s\in[t/2,t]}(\lV(\rr^{(a)},u^{(a)})(s)\rV_{W^{1,p}}),
	\end{aligned}
	\end{equation}
	and
	\begin{equation}     \label{rra,v,w_inf_Err2}
	\begin{aligned}
	&\int_0^t\lV   e^{(t-s)B} (0, (S+1)^{a_1}\Gamma^{a'}Err2)^{\top} \rV _{W^{N(a)-3,p}}+\lV   e^{(t-s)\Delta}  (S+1)^{a_1}\Gamma^{a'}Err2  \rV_{W^{N(a)-3,p}}ds\\
	\lesssim &\ep_1^2 \<t\>^{-1+(2+|a|)\delta}
	\end{aligned}
	\end{equation}

	Now we begin to prove the bounds (\ref{rra,v,w_inf_Lin})-(\ref{rra,v,w_inf_Err2}), respectively. The bound (\ref{rra,v,w_inf_Lin}) follows from (\ref{e-tB_Lin}), (\ref{Dec_heat}) and (\ref{MainAss_ini}) directly.
	
	{\bf Step 1. Proof of the bound (\ref{rra,v,w_inf_Lin2}).} By (\ref{e-tB_Lin}), (\ref{Dec_heat}) and inductive assumptions, we have for any $|b|\leq |a|-1$
	\begin{align*}
	&\lV \int_0^t  e^{(t-s)B} (0,\Delta u^{(b)})^{\top} ds\rV _{W^{N(a)-3,p}}+\lV \int_0^t  e^{(t-s)\Delta}  \Delta u^{(b)} ds \rV_{W^{N(a)-3,p}}\\
	\lesssim & \int_0^{t/2} (t-s)^{-7/4+\delta}\lV u^{(b)}\rV_{H^{N(a)-3}}ds+\int_{t/2}^t \<t-s\>^{-3/4+\delta}\lV \nab^2u^{(b)}\rV_{H^{N(a)-1}}ds\\
	&+\int_0^t e^{-c(t-s)}\lV \nab^2u^{(b)}\rV_{H^{N(a)-1}}ds\\
	\lesssim & \ep_1 t^{-3/4+\delta}+ \ep_1 t^{-1+(2+|b|)\delta}t^{1/4+\delta}+t^{-2}\int_0^{t/2}(t-s)^2 e^{-c(t-s)}\ep_1ds+\int_{t/2}^t e^{-c(t-s)}\ep_1\<s\>^{-1+(2+|b|)\delta}ds\\
	\lesssim &  \ep_1\<t\>^{-3/4+(2+|a|)\delta}.
	\end{align*}
	Thus, the bound (\ref{rra,v,w_inf_Lin2}) follows.
	
    {\bf Step 2. Proof of the bound (\ref{rra,v,w_inf_NonLin}).} By (\ref{e-tB_Lin}) and (\ref{Dec_heat}) we obtain
    \begin{align*}
    &\int_0^{t/2}\lV e^{(t-s)B} (F_1^{(a)}, \tilde{F}_{2,II}^{(a)})^{\top} \rV _{W^{N(a)-3,p}}+\lV e^{(t-s)\Delta} \tilde{F}_{2,II}^{(a)}\rV_{W^{N(a)-3,p}} ds\\
    \lesssim & \int_0^{t/2} (t-s)^{-3/2+\delta}\lV (F_1^{(a)},\tilde{F}_{2,II}^{(a)})\rV_{W^{N(a)-3,1}}+e^{-c(t-s)} \lV (F_1^{(a)},\tilde{F}_{2,II}^{(a)})\rV_{H^{N(a)-1}}ds,
    \end{align*}
    and
    \begin{align*}
    &\int_{t/2}^t\lV   e^{(t-s)B} (F_1^{(a)}, \tilde{F}_{2,II}^{(a)})^{\top}\rV _{W^{N(a)-3,p}}+\lV  e^{(t-s)\Delta} \tilde{F}_{2,II}^{(a)} \rV_{W^{N(a)-3,p}}ds\\
    \lesssim & \int_{t/2}^t \<t-s\>^{-3/4+\delta}\lV (F_1^{(a)},\tilde{F}_{2,II}^{(a)})\rV_{H^{N(a)-1}}+e^{-c(t-s)} \lV (F_1^{(a)},\tilde{F}_{2,II}^{(a)})\rV_{H^{N(a)-1}}ds\\
    \lesssim &\int_{t/2}^t \<t-s\>^{-3/4+\delta}\lV (F_1^{(a)},\tilde{F}_{2,II}^{(a)})\rV_{H^{N(a)-1}} ds.
    \end{align*}
    Then by Lemma \ref{F1a,F2a_Lem}, inductive assumptions, H\"{o}lder and (\ref{Main_Prop_Ass1}), we have
    \begin{align*}
    &\int_0^t\lV e^{(t-s)B} (F_1^{(a)}, \tilde{F}_{2,II}^{(a)})^{\top} \rV _{W^{N(a)-3,p}}+\lV e^{(t-s)\Delta}  \tilde{F}_{2,II}^{(a)}\rV_{W^{N(a)-3,p}} ds\\
    \lesssim & \int_0^{t/2} \ep_1(t-s)^{-3/2+\delta}\sum_{|b|\leq |a|} \lV \nab(\rr^{(b)},u^{(b)})\rV_{H^{N(a)-3}}+\ep_1e^{-c(t-s)}\sum_{|b|\leq |a|} \lV (\nab\rr^{(b)},\nab u^{(b)})\rV_{H^{N(a)-1}\times H^{N(a)}}ds\\
    &+\int_{t/2}^t \<t-s\>^{-3/4+\delta}\big[(\ep_1\<s\>^{-3/4+(2+|a|)\delta}+\lV(\rr^{(a)},u^{(a)})\rV_{W^{N(a)/2,\infty}})\lV (\nab\rr^{(c)},\nab u^{(c)})\rV_{H^{N(a)-1}\times H^{N(a)}}\\
    &\ \ \ \ \ \ \ \ +\ep_1^2\<s\>^{-1+5\delta}\sum_{|b|\leq |a|}\lV\nab\rr^{(b)}\rV_{H^{N(a)-1}}\big] ds\\
    \lesssim & \ep_1^2 t^{-1+\delta}+\ep_1(\ep_1\<t\>^{-3/4+(2+|a|)\delta}+\sup_{s\in[t/2,t]}\lV(\rr^{(a)},u^{(a)})(s)\rV_{W^{N(a)/2,p}})+\ep_1^3\<t\>^{-1+5\delta}\\
    \lesssim & \ep_1^2\<t\>^{-3/4+(2+|a|)\delta}+\ep_1 \sup_{s\in[t/2,t]}\lV(\rr^{(a)},u^{(a)})(s)\rV_{W^{N(a)/2,p}}.
    \end{align*}
    which yields the bounds (\ref{rra,v,w_inf_NonLin}).

    {\bf Step 3. Proof of the bound (\ref{rra,v,w_inf_F2III}) and (\ref{rra,v,w_inf_Err2}).} By the definition of $\tilde{F}_{2,III}^{(a)}$, we only estimate for any $|b|+|c|\leq |a|$ in detail
    \begin{equation}       \label{F2IIIa-Dec}
    \begin{aligned}
    &\int_0^t\Big(\lV   e^{(t-s)B} (0,\nab(\rr^{(b)}\nab u^{(c)}+\nab\phi^{(b)}\nab\phi^{(c)}))^{\top} \rV _{W^{N(a)-3,p}}\\
    &\ \ \ \ \ \ \ \ \ \ +\lV  e^{(t-s)\Delta}  \nab(\rr^{(b)}\nab u^{(c)}+\nab\phi^{(b)}\nab\phi^{(c)}) \rV_{W^{N(a)-3,p}}\Big)ds
    \end{aligned}
    \end{equation}
    From (\ref{e-tB_Lin}) and (\ref{Dec_heat}), we have
    \begin{align*}
    (\ref{F2IIIa-Dec})\lesssim &\int_0^{t/2}(t-s)^{-2+\delta}\lV \rr^{(b)}\nab u^{(c)}+\nab\phi^{(b)}\nab\phi^{(c)}\rV_{W^{N(a)-3,1}}ds\\
    &+\int_{t/2}^t \<t-s\>^{-5/4+\delta} \lV \rr^{(b)}\nab u^{(c)}+\nab\phi^{(b)}\nab\phi^{(c)}\rV_{H^{N(a)}} ds\\
    &+\int_0^te^{-c(t-s)}\lV \rr^{(b)}\nab u^{(c)}+\nab\phi^{(b)}\nab\phi^{(c)}\rV_{H^{N(a)}}ds.
    \end{align*}
    Then by (\ref{Hardy_2}) and (\ref{phi^2-Hna}) we get
    \begin{align*}
    (\ref{F2IIIa-Dec})\lesssim &\ep_1^2t^{-1+4\delta}+\int_{t/2}^t \<t-s\>^{-5/4+\delta}(\lV \rr^{(b)}\rV_{H^{N(a)}}\lV\nab u^{(c)}\rV_{L^{\infty}}+\lV\rr^{(b)}\rV_{L^{\infty}}\lV \nab u^{(c)}\rV_{H^{N(a)}}+\ep_1\<s\>^{-1+5\delta})ds\\
    &+\int_0^t e^{-c(t-s)} (\lV \rr^{(b)}\rV_{H^{N(a)}}\lV\nab u^{(c)}\rV_{L^{\infty}}+\lV\rr^{(b)}\rV_{L^{\infty}}\lV \nab u^{(c)}\rV_{H^{N(a)}}+\ep_1\<s\>^{-1+5\delta})ds\\
    \lesssim & \ep_1^2t^{-1+5\delta}+\ep_1\sum_{|b|\leq |a|}\sup_{s\in[t/2,t]}(\lV\nab u^{(b)}(s)\rV_{L^{\infty}}+\lV\rr^{(b)}(s)\rV_{L^{\infty}})\\
    \lesssim & \ep_1^2 t^{-3/4+(2+|a|)\delta}+\ep_1\sup_{s\in[t/2,t]}(\lV u^{(a)}(s)\rV_{W^{2,p}}+\lV\rr^{(a)}(s)\rV_{W^{2,p}}).
    \end{align*}
    Hence the bound (\ref{rra,v,w_inf_F2III}) follows. Using the similar argument to (\ref{rra,v,w_inf_F2III}), we obtain (\ref{rra,v,w_inf_Err2}). This completes the proof of (\ref{rra,v,w_inf}).

    {\bf Proof of (\ref{drra,v,w_L2dec}) for $t>1$.}

    By Duhamel's formula, it suffices to prove
    \begin{gather}       \label{drra,v,w_L2_Lin}
    \lV \nab e^{tB} (\rr^{(a)}_0,v^{(a)}_0)^{\top}\rV_{H^{N(a)-2}}+\lV \nab e^{t\Delta} w^{(a)}_0\rV_{H^{N(a)-2}}\lesssim \ep_0 t^{-1/2},  \\    \label{drra,v,w_L2_Lin2}
    \sum_{|b|\leq |a|-1}(\lV \int_0^t \nab e^{(t-s)B} (0,\Delta u^{(b)})^{\top} ds\rV _{H^{N(a)-2}}+\lV \int_0^t  \nab e^{(t-s)\Delta} \Delta u^{(b)} ds \rV_{H^{N(a)-2}})\lesssim \ep_1 t^{-1/2+(2+|b|)\delta},
    \end{gather}
    \begin{equation}
    \begin{aligned}
    \lV \int_0^t \nab e^{(t-s)B} (F_1^{(a)}, \tilde{F}_{2,II}^{(a)})^{\top} ds\rV _{H^{N(a)-2}}+\lV \int_0^t \nab e^{(t-s)\Delta} \P \tilde{F}_{2,II}^{(a)} ds \rV_{H^{N(a)-2}}
    \lesssim \ep_1^2 t^{-3/4+(3+|a|)\delta}, \label{drra,v,w_L2_NonLin}
    \end{aligned}
    \end{equation}
    \begin{equation}        \label{drra,v,w_L2_F2III}
    \begin{aligned}
    \int_0^t\lV \nab  e^{(t-s)B} (0, \tilde{F}_{2,III}^{(a)})^{\top} \rV _{H^{N(a)-2}}+\lV \nab  e^{(t-s)\Delta} \tilde{F}_{2,III}^{(a)} \rV_{H^{N(a)-2}} ds
    \lesssim  \ep_1\<t\>^{-3/4+5\delta},
    \end{aligned}
    \end{equation}
    and
    \begin{equation}     \label{drra,v,w_L2_Err2}
    \begin{aligned}
    \int_0^t\lV  \nab e^{(t-s)B} (0, Z^{a}Err2)^{\top} \rV _{H^{N(a)-2}}+\lV \nab  e^{(t-s)\Delta}  Z^{a}Err2  \rV_{H^{N(a)-2}}ds
    \lesssim \ep_1^2 \<t\>^{-1+(2+|a|)\delta}.
    \end{aligned}
    \end{equation}

    Now we begin to prove the bounds (\ref{drra,v,w_L2_Lin})-(\ref{drra,v,w_L2_Err2}), respectively. The bound (\ref{drra,v,w_L2_Lin}) follows from (\ref{etB_L2-L2}), (\ref{Dec_heat}) and (\ref{MainAss_ini}) directly.

    {\bf Step 1. Proof of the bound (\ref{drra,v,w_L2_Lin2}).} By (\ref{etB_L2-L2}), (\ref{Dec_heat}), (\ref{Main_Prop_Ass1}) and inductive assumptions, we have for any $|b|\leq |a|-1$
    \begin{align*}
    &\lV \int_0^t \nab e^{(t-s)B} (0,\Delta u^{(b)})^{\top} ds\rV _{H^{N(a)-2}}+\lV \int_0^t  \nab e^{(t-s)\Delta} \Delta u^{(b)} ds \rV_{H^{N(a)-2}}\\
    \lesssim &\int_0^{t/2} (t-s)^{-3/2}\lV u^{(b)}\rV_{H^{N(a)-2}}ds+\int_{t/2}^t \<t-s\>^{-1/2}\lV\nab^2 u^{(b)}\rV_{H^{N(a)-1}}ds\\
    &+\int_0^t e^{-c(t-s)}\lV \nab^3 u^{(b)}\rV_{H^{N(a)-2}}ds\\
    \lesssim & \ep_1t^{-1/2}+\ep_1t^{1/2-1+(2+|b|)\delta}+\ep_1t^{-2}\int_0^{t/2}(t-s)^2e^{-c(t-s)}ds+\ep_1t^{-1+(2+|b|)\delta}\\
    \lesssim & \ep_1t^{-1/2+(2+|b|)\delta}.
    \end{align*}

    {\bf Step 2. Proof of the bound (\ref{drra,v,w_L2_NonLin}).} By (\ref{etB_L2-L2}), (\ref{Dec_heat}) we have
    \begin{equation}     \label{drr,u-1/2_F1a}
    \begin{aligned}
    &\lV \int_0^t \nab e^{(t-s)B} (F_1^{(a)}, \tilde{F}_{2,II}^{(a)})^{\top} ds\rV _{H^{N(a)-2}}+\lV \int_0^t \nab e^{(t-s)\Delta} \P \tilde{F}_{2,II}^{(a)} ds \rV_{H^{N(a)-2}}\\
    \lesssim & \int_0^{t/2} (t-s)^{-5/4}\lV(F_1^{(a)},F_{2,II}^{(a)})\rV_{W^{N(a)-2,1}}ds+\int_{t/2}^t \<t-s\>^{-1/2}\lV(F_1^{(a)},F_{2,II}^{(a)})\rV_{H^{N(a)-1}}ds\\
    &+\int_0^t e^{-c(t-s)}\lV(F_1^{(a)},F_{2,II}^{(a)})\rV_{H^{N(a)-1}} ds.
    \end{aligned}
    \end{equation}
    From Lemma \ref{F1a,F2a_Lem}, (\ref{rra,v,w_inf}) and (\ref{Main_Prop_Ass1}), we may bound this by
    \begin{align*}
    {\rm LHS}(\ref{drr,u-1/2_F1a})\lesssim & \int_0^{t/2} (t-s)^{-5/4} \ep_1\sum_{|b|\leq |a|}\lV (\nab\rr^{(b)},\nab u^{(b)})\rV_{H^{N(a)-2}}ds\\
    &+\int_{t/2}^t\<t-s\>^{-1/2}\ep_1\<s\>^{-3/4+(2+|a|)\delta}\sum_{|b|\leq |a|}\lV \nab\rr^{(b)},\nab u^{(b)}\rV_{H^{N(a)-1}\times H^{N(a)}}ds\\
    &+\int_0^t e^{-c(t-s)} \ep_1\<s\>^{-3/4+(2+|a|)\delta}\sum_{|b|\leq |a|}\lV \nab\rr^{(b)},\nab u^{(b)}\rV_{H^{N(a)-1}\times H^{N(a)}}ds\\
    \lesssim & \ep_1^2t^{-3/4}+\ep_1^2 t^{-3/4+(3+|a|)\delta}\lesssim \ep_1^2t^{-3/4+(3+|a|)\delta}.
    \end{align*}

    {\bf Step 3. Proof of the bound (\ref{drra,v,w_L2_F2III}).} By the definition of $\tilde{F}^{(a)}_{2,III}$, we only prove the following term in detail
    \begin{equation}        \label{drra,-1/2,F2,III}
    \int_0^t\lV \nab  e^{(t-s)B} (0, \nab(\rr^{(b)}\nab u^{(c)}+\nab\phi^{(b)}\nab\phi^{(c)}))^{\top} \rV _{H^{N(a)-2}}+\lV \nab  e^{(t-s)\Delta} \nab(\rr^{(b)}\nab u^{(c)}+\nab\phi^{(b)}\nab\phi^{(c)}) \rV_{H^{N(a)-2}} ds,
    \end{equation}
    for $|b|+|c|\leq |a|$. Using (\ref{etB_L2-L2}) and (\ref{Dec_heat}), it follows that
    \begin{align*}
    (\ref{drra,-1/2,F2,III})\lesssim & \int_0^{t/2} (t-s)^{-7/4}\lV \rr^{(b)}\nab u^{(c)}+\nab\phi^{(b)}\nab\phi^{(c)}\rV_{W^{N(a)-2,1}} ds\\
    &+\int_{t/2}^t \<t-s\>^{-1}\lV \rr^{(b)}\nab u^{(c)}+\nab\phi^{(b)}\nab\phi^{(c)}\rV_{H^{N(a)}} ds\\
    &+\int_0^t e^{-c(t-s)}\lV \rr^{(b)}\nab u^{(c)}+\nab\phi^{(b)}\nab\phi^{(c)}\rV_{H^{N(a)}}ds.
    \end{align*}
    By (\ref{biL-Ffg-WN1}), (\ref{biL-Ffg-HN}) and (\ref{Main_Prop_Ass1}), one may obtain
    \begin{equation}        \label{rrua,phiphia-biL}
    \lV \rr^{(b)}\nab u^{(c)}+\nab\phi^{(b)}\nab\phi^{(c)}\rV_{W^{N(a)-2,1}}+\lV \rr^{(b)}\nab u^{(c)}+\nab\phi^{(b)}\nab\phi^{(c)}\rV_{H^{N(a)}}
    \lesssim  \ep_1\lV\nab u^{(c)}\rV_{H^{N(a)}}+\ep_1^2\<t\>^{3\delta}
    \end{equation}
    and from (\ref{Main_Prop_Ass1}), (\ref{rra,v,w_inf}) and (\ref{Dec_Phi}) we have
    \begin{align*}
    &\lV \rr^{(b)}\nab u^{(c)}+\nab\phi^{(b)}\nab\phi^{(c)}\rV_{H^{N(a)}}\\
    \lesssim & \lV\rr^{(b)}\rV_{H^{N(a)}}\lV\nab u^{(c)}\rV_{L^{\infty}}+\lV\rr^{(b)}\rV_{L^{\infty}}\lV\nab u^{(c)}\rV_{H^{N(a)}}+\ep_1^2\<t\>^{-1+4\delta}\\
    \lesssim & \ep_1^2\<t\>^{-3/4+(2+|c|)\delta}+\ep_1t^{-3/4+(2+|b|)\delta}\lV\nab u^{(c)}\rV_{H^{N(a)}}.
    \end{align*}
    Hence, by the above two bounds, we may bound (\ref{drra,-1/2,F2,III}) by
    \begin{align*}
    (\ref{drra,-1/2,F2,III})\lesssim & \int_0^{t/2}((t-s)^{-7/4}+e^{-c(t-s)})(\ep_1\sum_{|c|\leq |a|}\lV\nab u^{(c)}\rV_{H^{N(a)}}+\ep_1^2\<s\>^{3\delta})ds\\
    &+\int_{t/2}^t (\<t-s\>^{-1}+e^{-c(t-s)})(\ep_1^2\<t\>^{-3/4+(2+|a|)\delta}+\ep_1t^{-3/4+(2+|a|)\delta}\sum_{|c|\leq |a|}\lV\nab u^{(c)}\rV_{H^{N(a)}})ds\\
    \lesssim & \ep_1^2t^{-3/4+3\delta}+\ep_1^2t^{-3/4+(3+|a|)\delta}\lesssim \ep_1^2 t^{-3/4+(3+|a|)\delta}.
    \end{align*}
    This implies the bound (\ref{drra,v,w_L2_F2III}).

    Finally, the bound (\ref{drra,v,w_L2_Err2}) can also be obtained using the similar argument to that of the bounds (\ref{drra,v,w_L2_NonLin}) and (\ref{drra,v,w_L2_F2III}).    This completes the proof of (\ref{drra,v,w_L2dec}).

    {\bf Proof of (\ref{d2rra,v,w_L2dec}) for $t>1$.}

    It suffices to prove
    \begin{align}       \label{d2rra,v,w_L2_Lin}
    &\lV \nab^2 e^{tB} (\rr^{(a)}_0,v^{(a)}_0)^{\top}\rV_{H^{N(a)-5}}+\lV \nab^2 e^{t\Delta} w^{(a)}_0\rV_{H^{N(a)-5}}\lesssim \ep_0 t^{-1},  \\    \label{d2rra,v,w_L2_Lin2}
    &\sum_{|b|\leq |a|-1}(\lV \int_0^t \nab^2 e^{(t-s)B} (0,\Delta u^{(b)})^{\top} ds\rV _{H^{N(a)-5}}+\lV \int_0^t  \nab^2 e^{(t-s)\Delta} \Delta u^{(b)} ds \rV_{H^{N(a)-5}})\lesssim \ep_1 t^{-1+(2+|a|)\delta},
    \end{align}
    \begin{equation}
    \begin{aligned}
    \lV \int_0^t \nab^2 e^{(t-s)B} (F_1^{(a)}, \tilde{F}_{2,II}^{(a)})^{\top} ds\rV _{H^{N(a)-5}}+\lV \int_0^t \nab^2 e^{(t-s)\Delta} \P \tilde{F}_{2,II}^{(a)} ds \rV_{H^{N(a)-5}}
    \lesssim \ep_1^2 t^{-5/4+(4+|a|)\delta} \label{d2rra,v,w_L2_NonLin}
    \end{aligned}
    \end{equation}
    \begin{equation}        \label{d2rra,v,w_L2_F2III}
    \begin{aligned}
    \int_0^t\lV \nab^2  e^{(t-s)B} (0, \tilde{F}_{2,III}^{(a)})^{\top} \rV _{H^{N(a)-5}}+\lV \nab^2  e^{(t-s)\Delta} \tilde{F}_{2,III}^{(a)} \rV_{H^{N(a)-5}} ds
    \lesssim  \ep_1\<t\>^{-1+(2+|a|)\delta},
    \end{aligned}
    \end{equation}
    and
    \begin{equation}     \label{d2rra,v,w_L2_Err2}
    \begin{aligned}
    \int_0^t\lV  \nab e^{(t-s)B} (0, Z^{a}Err2)^{\top} \rV _{H^{N(a)-2}}+\lV \nab  e^{(t-s)\Delta}  Z^{a}Err2  \rV_{H^{N(a)-2}}ds
    \lesssim \ep_1^2 \<t\>^{-1+(2+|a|)\delta}.
    \end{aligned}
    \end{equation}
    then, by Duhamel's formula and (\ref{d2rra,v,w_L2_Lin})-(\ref{d2rra,v,w_L2_Err2}), the bound (\ref{d2rra,v,w_L2dec}) follows.

    Now we begin to prove the bounds (\ref{d2rra,v,w_L2_Lin})-(\ref{d2rra,v,w_L2_Err2}), respectively. The bound (\ref{d2rra,v,w_L2_Lin}) follows from (\ref{etB_L2-L2}), (\ref{Dec_heat}) and (\ref{MainAss_ini}) directly.

    {\bf Step 1. Proof of the bound (\ref{d2rra,v,w_L2_Lin2}).} By (\ref{etB_L2-L2}), (\ref{Dec_heat}), (\ref{Main_Prop_Ass1}) and inductive assumptions, we have for any $|b|\leq |a|-1$
    \begin{align*}
    &\lV \int_0^t \nab^2 e^{(t-s)B} (0,\Delta u^{(b)})^{\top} ds\rV _{H^{N(a)-5}}+\lV \int_0^t  \nab^2 e^{(t-s)\Delta} \Delta u^{(b)} ds \rV_{H^{N(a)-5}}\\
    \lesssim &\int_0^{t/2} (t-s)^{-2}\lV u^{(b)}\rV_{H^{N(a)-5}}ds+\int_{t/2}^t \<t-s\>^{-1}\lV\nab^2 u^{(b)}\rV_{H^{N(a)-3}}ds\\
    &+\int_0^t e^{-c(t-s)}\lV \nab^4 u^{(b)}\rV_{H^{N(a)-5}}ds\\
    \lesssim & \ep_1t^{-1}+\ep_1t^{\delta-1+(2+|b|)\delta}+\ep_1t^{-2}\int_0^{t/2}(t-s)^2e^{-c(t-s)}ds+\ep_1t^{-1+(2+|b|)\delta}\\
    \lesssim & \ep_1t^{-1+(2+|a|)\delta}.
    \end{align*}

    {\bf Step 2. Proof of the bound (\ref{d2rra,v,w_L2_NonLin}).} By (\ref{etB_L2-L2}), (\ref{Dec_heat}) we have
    \begin{equation}          \label{d2rra,-1,Nonl}
    \begin{aligned}
    &\lV \int_0^t \nab^2 e^{(t-s)B} (F_1^{(a)}, \tilde{F}_{2,II}^{(a)})^{\top} ds\rV _{H^{N(a)-5}}+\lV \int_0^t \nab^2 e^{(t-s)\Delta} \P \tilde{F}_{2,II}^{(a)} ds \rV_{H^{N(a)-5}}\\
    \lesssim & \int_0^{t/2} (t-s)^{-7/4}\lV(F_1^{(a)},F_{2,II}^{(a)})\rV_{W^{N(a)-5,1}}ds+\int_{t/2}^t \<t-s\>^{-1}\lV(F_1^{(a)},F_{2,II}^{(a)})\rV_{H^{N(a)-3}}ds\\
    &+\int_0^t e^{-c(t-s)}\lV(F_1^{(a)},F_{2,II}^{(a)})\rV_{H^{N(a)-3}} ds.
    \end{aligned}
    \end{equation}
    Then from Lemma \ref{F1a,F2a_Lem}, (\ref{rra,v,w_inf}), (\ref{drra,v,w_L2dec}) and (\ref{Main_Prop_Ass1}), we may bound this by
    \begin{align*}
    {\rm LHS}(\ref{d2rra,-1,Nonl})\lesssim & \int_0^{t/2}(t-s)^{-7/4} \ep_1\sum_{|b|\leq |a|}\lV (\nab\rr^{(b)},\nab u^{(b)})\rV_{H^{N(a)-1}}ds\\
    &+\int_{t/2}^t \<t-s\>^{-1} \ep_1^2 \<s\>^{-5/4+(4+|a|)\delta}ds+\int_0^t e^{-c(t-s)} \ep_1^2 \<s\>^{-5/4+(4+|a|)\delta}ds\\
    \lesssim & \ep_1^2t^{-5/4}+\ep_1^2 t^{-5/4+(5+|a|)\delta}\lesssim \ep_1^2t^{-5/4+(5+|a|)\delta}.
    \end{align*}

    {\bf Step 3. Proof of the bound (\ref{d2rra,v,w_L2_F2III}).} We only estimate the following term in detail
    \begin{equation}        \label{d2rra,-1,F2,III}
    \int_0^t\lV \nab^2  e^{(t-s)B} (0, \nab(\rr^{(b)}\nab u^{(c)}+\nab\phi^{(b)}\nab\phi^{(c)}))^{\top} \rV _{H^{N(a)-5}}+\lV \nab^2  e^{(t-s)\Delta} \nab(\rr^{(b)}\nab u^{(c)}+\nab\phi^{(b)}\nab\phi^{(c)}) \rV_{H^{N(a)-5}} ds,
    \end{equation}
    for $|b|+|c|\leq |a|$. In fact, it follows from (\ref{etB_L2-L2}) and (\ref{Dec_heat}) that
    \begin{align*}
    (\ref{d2rra,-1,F2,III})\lesssim & \int_0^{t/2} (t-s)^{-9/4}\lV \rr^{(b)}\nab u^{(c)}+\nab\phi^{(b)}\nab\phi^{(c)}\rV_{W^{N(a)-5,1}} ds\\
    &+\int_{t/2}^t \<t-s\>^{-3/2}\lV \rr^{(b)}\nab u^{(c)}+\nab\phi^{(b)}\nab\phi^{(c)}\rV_{H^{N(a)-2}} ds\\
    &+\int_0^t e^{-c(t-s)}\lV \rr^{(b)}\nab u^{(c)}+\nab\phi^{(b)}\nab\phi^{(c)}\rV_{H^{N(a)-2}}ds.
    \end{align*}
    Then by (\ref{rra,v,w_inf}), (\ref{drra,v,w_L2dec}) and (\ref{phi^2-Hna}), we obtain
    \begin{align*}
    &\lV \rr^{(b)}\nab u^{(c)}+\nab\phi^{(b)}\nab\phi^{(c)}\rV_{H^{N(a)-2}}\\
    \lesssim & \lV\rr^{(b)}\rV_{W^{N(a)/2,\infty}}\lV\nab u^{(c)}\rV_{H^{N(a)-2}}+\lV\nab\rr^{(b)}\rV_{H^{N(a)-3}}\lV\nab u^{(c)}\rV_{W^{N(a)/2,\infty}}+\ep_1^2\<t\>^{-1+(2+|a|)\delta}\\
    \lesssim & \ep_1^2\<t\>^{-5/4+(4+|a|)\delta}+\ep_1t^{-1+(2+|a|)\delta}\lesssim \ep_1^2\<t\>^{-1+(2+|a|)\delta},
    \end{align*}
  which, combined with (\ref{rrua,phiphia-biL}), yields
    \begin{align*}
    (\ref{d2rra,-1,F2,III})\lesssim & \int_0^{t/2}((t-s)^{-9/4}+e^{-c(t-s)})(\ep_1\sum_{|c|\leq |a|}\lV\nab u^{(c)}\rV_{H^{N(a)}}+\ep_1^2\<s\>^{3\delta})ds\\
    &+\int_{t/2}^t (\<t-s\>^{-3/2}+e^{-c(t-s)})\ep_1^2\<t\>^{-1+(2+|a|)\delta}ds\\
    \lesssim & \ep_1^2t^{-5/4+3\delta}+\ep_1^2t^{-1+(2+|a|)\delta}\lesssim \ep_1^2 t^{-1+(2+|a|)\delta}.
    \end{align*}
    This implies the bound (\ref{drra,v,w_L2_F2III}).

    Finally, the bound (\ref{d2rra,v,w_L2_Err2}) can also be obtained using the similar argument to that of (\ref{d2rra,v,w_L2_NonLin}) and (\ref{d2rra,v,w_L2_F2III}). This completes the proof of (\ref{d2rra,v,w_L2dec}).

\end{proof}

\subsection{Decay estimates of $\nab\rr^{(a)},\nab u^{(a)}$ for any $|a|\leq N_1$} Due to the presence of dissipation, by the linear estimate (\ref{e tB_Lp-Lp}), we have the following improved pointwise decay estimates.

\begin{lemma} With the hypothesis in Proposition \ref{Main_Prop}, let $l(|a|+1)=l(|a|)/40$ for $0\leq |a|\leq N_1$ and $l(0)=10^{-2}$. For any $t\in[0,T]$ we have
	\begin{equation}     \label{drr,v,w_inf}
	\lV |\nab|(\rr,v,w)\rV_{W^{N(2)-2,\infty}} \lesssim \ep_1 \<t\>^{-1-l(0)+2\delta},
	\end{equation}
	\begin{equation}  \label{drra,v,w_a_inf}
	\lV |\nab|(\rr^{(a)},v^{(a)},w^{(a)})\rV_{W^{N(a+2)-2,\infty}} \lesssim \ep_1 \<t\>^{-1-l(|a|)/3+(2+|a|)\delta},\ {\rm for}\ 1\leq |a|\leq N_1-2,
	\end{equation}
	\begin{equation}  \label{drra,v,w_a_inf_5}
	\lV|\nab|(\rr^{(a)},v^{(a)},w^{(a)})\rV_{W^{N(a+2)-2,\infty}} \lesssim \ep_1 \<t\>^{-1+(2+|a|)\delta},\ {\rm for}\ N_1-1\leq |a|\leq N_1.
	\end{equation}	
\end{lemma}
\begin{proof} {\bf Proof of the bound (\ref{drr,v,w_inf}).}
	
	By Duhamel's formula it suffices to show the bounds for $t>1$
	\begin{gather}       \label{drr,v-Lin}
	\lV e^{tB}|\nab|(\rr_0,v_0)\rV_{W^{N(2)-2,\infty}}+\lV e^{t\Delta}|\nab|w_0\rV_{W^{N(2)-2,\infty}}\lesssim \ep_0t^{-5/4},\\\label{drr,v-NonL}
	\int_0^t \lV  e^{(t-s)B}(|\nab|F_1,\nab\cdot F_2)^{\top}\rV_{W^{N(2)-2,\infty}}ds \lesssim \ep_1^2 t^{-1-l(0)+2\delta},\\\label{dw_inf_NonLin}
	\int_0^t \lV e^{(t-s)\Delta}|\nab|\P F_2\rV_{W^{N(2)-2,\infty}}ds \lesssim \ep_1^2 t^{-5/4}.
	\end{gather}
	The first bound (\ref{drr,v-Lin}) is obtained by (\ref{e-tB_Lin}) and (\ref{Dec_heat}). Now we prove (\ref{drr,v-NonL}) and (\ref{dw_inf_NonLin}).
	
	\textbf{Step 1. We prove the bound (\ref{drr,v-NonL}).} By (\ref{e-tB_Lin}) we have
	\begin{equation}           \label{LHS3.58}
	\begin{aligned}
	&\int_0^{t/2} \lV  e^{(t-s)B}(|\nab|F_1,\nab\cdot F_2)^{\top}\rV_{W^{N(2)-2,\infty}}ds \\
	\lesssim & \int_0^{t/2}(t-s)^{-2}\lV(F_1,F_{2,I})\rV_{W^{N(2)-2,1}\cap H^{N(2)+1}}+(t-s)^{-5/2}\lV \tfrac{1+\cos^2\phi_2}{1+\rr}|\nab\phi|^2\rV_{W^{N(2)-2,1}\cap H^{N(2)+1}},
	\end{aligned}
	\end{equation}
	which, by Lemma \ref{biL-lem} and (\ref{Main_Prop_Ass1}) , can be bounded by
	\begin{align*}
	& \int_0^{t/2}(t-s)^{-2}\lV(\nab u,\nab\rr)\rV_{H^{N(2)+2}}(\lV(u,\rr)\rV_{H^{N(2)+1}}+\lV\Phi\rV_{H^{N(2)+1}}^2)+(t-s)^{-5/2}\lV\Phi\rV_{H^{N(2)+2}}^2ds\\
	\lesssim & \int_0^{t/2}\ep_1 t^{-2}\lV(\nab u,\nab\rr)\rV_{H^{N(2)+2}}+\ep_1^2 t^{-5/2}ds\lesssim \ep_1^2 t^{-3/2}.
	\end{align*}
	Next, let $k(0)$ satisfy $2^{k(0)}\simeq \<t\>^{-l(0)}$, by frequency decomposition $\mathbf{1}=P_{\leq k(0)}+P_{>k(0)}$, we may rewrite the integral
	\begin{align*}
	\int_{t/2}^t  e^{(t-s)B}(|\nab|F_1,\nab\cdot F_2)^{\top}ds=&\int_{t/2}^t P_{\leq k(0)} e^{(t-s)B}(|\nab|F_1,\nab\cdot F_2)^{\top}ds+\int_{t/2}^t  P_{>k(0)} e^{(t-s)B}(|\nab|F_1,\nab\cdot F_2)^{\top}ds\\
	=& \mathbf{I}_{\leq k(0)}+\mathbf{I}_{>k(0)}.
	\end{align*}
	We estimate now the low frequency part and high frequency part, respectively.
		
	\emph{Low frequency part: $\mathbf{I}_{\leq k(0)}$.} From (\ref{e-tB_Lin}), (\ref{rr,v,w_inf}), (\ref{drr,v,w_L2dec}) and (\ref{Dec_Phi}), we get
	\begin{align*}
	\left|\mathbf{I}_{\leq k(0)}\right|\lesssim & \int_{t/2}^t \<t-s\>^{-5/4}\big(\lV (u,\rr)\rV_{L^{\infty}}\lV(\nab u,\nab\rr)\rV_{H^2}+\lV\nab\rr\rV_{L^2}\lV\Phi\rV_{L^2}\lV\Phi\rV_{L^{\infty}}+2^{k(0)}\lV\Phi\rV_{L^2}\lV\Phi\rV_{L^{\infty}}\big)ds\\
	\lesssim & \ep_1^2 t^{-5/4}+\ep_1^3t^{-3/2+2\delta}+\ep_1^2t^{-1-l(0)+2\delta}\lesssim \ep_1^2 t^{-1-l(0)+2\delta}.
	\end{align*}
	
	\emph{High frequency part: $\mathbf{I}_{>k(0)}$.}	By (\ref{e tB_Lp-Lp}), (\ref{rr,v,w_inf}) and (\ref{Dec_Phi}), we have
	\begin{align*}
	\left|\mathbf{I}_{>k(0)}\right|\lesssim &\int_{t/2}^t \<t-s\>^{-3/2}2^{-13k(0)} \lV(|\nab|F_1,\nab\cdot F_2)^{\top}\rV_{W^{N(2)-2,\infty}}  ds\\
	\lesssim & 2^{-13k(0)}\cdot\sup_{s\in[t/2,t]}\lV  (|\nab|F_1,\nab\cdot F_2)^{\top}(s)\rV_{W^{N(2)-2,\infty}}\\
	\lesssim & t^{13l(0)}\cdot\sup_{s\in[t/2,t]} \big(\lV(u,\rr)(s)\rV_{W^{N(2)+2,\infty}}^2+\sum_{k}2^{N(2)k}\lV P_k\nab\phi(s)\rV_{L^{\infty}}\lV\nab\phi(s)\rV_{L^{\infty}}\big)\\
	\lesssim & t^{13/100}(\ep_1^2 t^{-3/2+2\delta}+\ep_1^2t^{-2+4\delta})\lesssim \ep_1^2 t^{-5/4},
	\end{align*}
	where we have used (\ref{rr,v,w_inf}) and (\ref{Dec_Phi}). Hence, the bound (\ref{drr,v-NonL}) follows.
	
	\textbf{Step 2. We prove the bound (\ref{dw_inf_NonLin}).}
	By (\ref{Dec_heat}) we have
	\begin{equation*}       \label{dw,nonl}
	\begin{aligned}
	&\int_0^{t/2} \lV  e^{(t-s)\Delta}|\nab|\P F_2\rV_{W^{N(2)-2,\infty}}ds\\
	\lesssim &\int_0^{t/2} (t-s)^{-2}\lV F_{2,I}\rV_{W^{N(2)-2,1}}+(t-s)^{-5/2}\lV \frac{1+\cos^2\phi_2}{1+\rr}\nab\phi\nab\phi\rV_{W^{N(2)-2,1}} ds,
	\end{aligned}
	\end{equation*}
	which, by (\ref{F1,F2I-WN1}), (\ref{biL-Ffg-WN1}) and (\ref{Main_Prop_Ass1}), is further bounded by
	\begin{equation*}
	\int_0^{t/2} \ep_1t^{-2}(\lV\nab \rr(s)\rV_{H^{N(0)-1}}+\lV\nab u(s)\rV_{H^{N(0)}})+\ep_1^2t^{-5/2} ds\lesssim \ep_1^2t^{-3/2}.
	\end{equation*}
	
	Next, using (\ref{Dec_heat}), we get
	\begin{equation}     \label{dw,-1-delta,t/2}
	\begin{aligned}
	&\int_{t/2}^t\lV e^{t\Delta} |\nab|\P F_2\rV_{W^{N(2)-2,\infty}}ds\\
	\lesssim &\int_{t/2}^{t-1}(t-s)^{-5/4}\lV F_{2,I}\rV_{H^{N(2)-2}}+(t-s)^{-1}\lV \frac{1+\cos^2\phi_2}{1+\rr}\nab\phi\nab\phi\rV_{W^{N(2)-2,\infty}}ds\\
	&+\int_{t-1}^t \lV   |\nab| F_2\rV_{W^{N(2)-2,\infty}}ds.
	\end{aligned}
	\end{equation}
	From (\ref{rr,v,w_inf}), (\ref{drr,v,w_L2dec}) and (\ref{Dec_Phi}), we have
	\begin{equation*}
	\begin{aligned}     \label{dw_I1}
	\lV F_{2,I}\rV_{H^{N(2)-2}}\lesssim & \lV(\rr,u)\rV_{W^{N(2)-1,\infty}}\lV \nab(\rr,u)\rV_{H^{N(2)-1}}+\lV \nab\rr\rV_{L^{\infty}}\lV\nab\phi\rV_{L^{\infty}}\lV\nab\phi\rV_{H^{N(0)}}+\lV \nab\rr\rV_{H^{N(2)-2}}\lV \nab\phi\rV_{L^{\infty}}^2 \\
	\lesssim & \ep_1^2\<t\>^{-5/4+\delta}+\ep_1^3\<t\>^{-7/4+3\delta}+\ep_1^3\<t\>^{-5/2+4\delta}\lesssim \ep_1^2\<t\>^{-5/4+\delta},
	\end{aligned}
	\end{equation*}
	\begin{equation*}      \label{dw_I2}
	\begin{aligned}
	\lV \frac{1+\cos^2\phi_2}{1+\rr}\nab\phi\nab\phi\rV_{W^{N(2)-2,\infty}}\lesssim (1+\ep_1)\lV\nab\phi\rV_{W^{N(2)-2,\infty}}^2
	\lesssim \ep_1^2t^{-2+4\delta}.
	\end{aligned}
	\end{equation*}
	and
	\begin{equation*}
	\lV   |\nab| F_2\rV_{W^{N(2)-2,\infty}}\lesssim \ep_1^2\<t\>^{-5/4+\delta}.
	\end{equation*}
	Thus by these estimates, we get
	\begin{equation*}
	{\rm LHS}(\ref{dw,-1-delta,t/2})\lesssim \int_{t/2}^{t-1}(t-s)^{-5/4}\ep_1^2\<s\>^{-5/4+\delta}+(t-s)^{-1}\ep_1^2\<s\>^{-2+2\delta}ds+\int_{t-1}^t \ep_1^2\<s\>^{-5/4+\delta}\lesssim \ep_1^2 t^{-5/4+\delta},
	\end{equation*}
	which gives the bound (\ref{dw_inf_NonLin}). Thus this completes the proof of (\ref{drr,v,w_inf}).
	
	{\bf Proof of the bound (\ref{drra,v,w_a_inf}).}
	
	We prove the bound (\ref{drra,v,w_a_inf}) by induction. Precisely, by (\ref{drr,v,w_inf}) we assume that the bound
	\begin{equation}  \label{drra,v,w-IndAss}
	\lV |\nab|(\rr^{(b)},v^{(b)},w^{(b)})\rV_{W^{N(b+2)-2,\infty}} \lesssim \ep_1 \<t\>^{-1-l(|b|)/3+(2+|b|)\delta},\ {\rm for}\ |b|<|a|
	\end{equation}
	holds, then we prove the bound (\ref{drra,v,w_a_inf}).
	
	Since the bound (\ref{drra,v,w_a_inf}) for $w^{(a)}$ can be proved by a similar argument to the bound (\ref{drra,v,w_a_inf}) for $\rr^{(a)}$ and $v^{(a)}$ using (\ref{Dec_heat}), (\ref{Main_Prop_Ass1}), (\ref{rra,v,w_inf}), (\ref{drra,v,w_L2dec}), (\ref{Dec_Phi}) and inductive assumptions (\ref{drra,v,w-IndAss}), we only prove the bound (\ref{drra,v,w_a_inf} for $\rr^{(a)}$ and $v^{(a)}$ in detail here. By Duhamel's formula it suffices to show the following bounds
	\begin{gather}       \label{drra,v-Lin1}
	\lV e^{tB}|\nab|(\rr_0^{(a)},v_0^{(a)})\rV_{W^{N(a+2)-2,\infty}}\lesssim \ep_0 t^{-5/4},\\\label{drra,v-Lin2}
	\int_0^t \lV  e^{(t-s)B}(0,|\nab|\Delta u^{(b)})^{\top}\rV_{W^{N(a+2)-2,\infty}}ds \lesssim \ep_1^2 t^{-1-l(|a|)/3+(2+|b|)\delta},\ {\rm for}\ |b|<|a|,\\      \label{drra,v-NonL}
	\int_0^t \lV  e^{(t-s)B}(|\nab|\tilde{F}^{(a)}_1,\nab\cdot \tilde{F}^{(a)}_{2,II})^{\top}\rV_{W^{N(a+2)-2,\infty}}ds \lesssim \ep_1^2 t^{-1-l(|a|)+4\delta},\\ \label{drra,v-F2III+Err2}
	\int_0^t \lV  e^{(t-s)B}(0,\nab\cdot \tilde{F}^{(a)}_{2,III}+\nab\cdot(S+1)^{a_1}\Gamma^{a'}Err2)^{\top}\rV_{W^{N(a+2)-2,\infty}}ds \lesssim \ep_1^2 t^{-1-l(|a|)+4\delta}
	\end{gather}
	The first bound (\ref{drra,v-Lin1}) is obtained by (\ref{e-tB_Lin}). Next we show (\ref{drra,v-Lin2}) and (\ref{drra,v-NonL}).
	
	\textbf{Step 1. Proof of the bound (\ref{drra,v-Lin2}).} From (\ref{e-tB_Lin}) and (\ref{Main_Prop_Ass1}), we have
	\begin{align*}
	\int_0^{t/2} \lV  e^{(t-s)B}(0,|\nab|\Delta u^{(b)})^{\top}\rV_{W^{N(2)-2,\infty}}ds
	\lesssim  \int_0^{t/2} (t-s)^{-3/4-3/2}\lV u^{(b)}\rV_{H^{N(2)+3}} ds\lesssim \ep_1 t^{-5/4}.
	\end{align*}
	Next, by frequency decomposition we divide the following integral into low frequency part and high frequency part, i.e
	$$\int_{t/2}^t  e^{(t-s)B}(0,|\nab|\Delta u^{(b)})^{\top}ds=\int_{t/2}^t P_{\leq k(|a|)} e^{(t-s)B}(0,|\nab|\Delta u^{(b)})^{\top}ds+\int_{t/2}^t  P_{>k(|a|)} e^{(t-s)B}(0,|\nab|\Delta u^{(b)})^{\top}ds, $$
	where $k(|a|)$ satisfies $2^{k(|a|)}\simeq \<t\>^{-l(|a|)}$. Then we estimate these two contributions separately.

	\emph{Low frequency part.} By (\ref{e-tB_Lin}) and (\ref{d2rra,v,w_L2dec}),
	\begin{align*}
	\int_{t/2}^t \lV P_{\leq k(|a|)} e^{(t-s)B}(0,|\nab|\Delta u^{(b)})^{\top}\rV_{L^{\infty}}ds\lesssim & \int_{t/2}^t \<t-s\>^{-3/4-1/3}2^{k(|a|)/3}\lV\Delta u^{(b)}\rV_{L^2}ds\lesssim \ep_1 t^{-1+(2+|b|)\delta-l(|a|)/3}.
	\end{align*}
	
	\emph{High frequency part.} By (\ref{e tB_Lp-Lp}), (\ref{drra,v,w-IndAss}) and $l(|b|+1)=\frac{l(|b|)}{40}$ for $0\leq |b|\leq N_1$
	\begin{align*}
	&\int_{t/2}^t \lV P_{> k(|a|)} e^{(t-s)B}(0,|\nab|\Delta u^{(b)})^{\top}\rV_{W^{N(a+2)-2,\infty}}ds\\
	\lesssim &\int_{t/2}^t\<t-s\>^{-3/2}2^{-13k(|a|)} \lV\nab u^{(b)}\rV_{W^{N(a+2),\infty}}ds\\
	\lesssim &\ep_1 t^{13l(|a|)-1-l(|b|)/3+(2+|b|)\delta+\delta}\lesssim \ep_1 t^{-1-l(|a|)/3+(2+|a|)\delta}.
	\end{align*}
	This completes the proof of (\ref{drra,v-Lin2}).
	
	\textbf{Step 2. Proof of the bound (\ref{drra,v-NonL}).} On one hand, by (\ref{e-tB_Lin}), H\"{o}lder and (\ref{Main_Prop_Ass1}), we have
	\begin{align*}
	&\int_0^{t/2} \lV  e^{(t-s)B}(|\nab|F_1^{(a)},\nab\cdot \tilde{F}_{2,II}^{(a)})^{\top}\rV_{W^{N(a+2)-2,\infty}}ds\\
	\lesssim & \int_0^{t/2} (t-s)^{-2} \sum_{b+c=a}\lV (\nab\rr^{(b)},\nab u^{(b)})\rV_{H^{N(a+1)}}\Big(\lV(\rr^{(c)},u^{(c)})\rV_{H^{N(a+1)}}+\sum_{e+\tilde{e}=c}\lV \Phi^{(e)}\rV_{H^{N(a+1)}}\lV\Phi^{(\tilde{e})}\rV_{H^{N(a+1)}}\Big)ds\\
	\lesssim & \ep_1\<t\>^{-3/2}(\ep_1+\ep_1^2\<t\>^{2\delta})\lesssim \ep_1^2\<t\>^{-3/2+2\delta}.
	\end{align*}
	On the other hand, for the nonlinear terms $F_1^{(a)}$ and $\tilde{F}_{2,II}^{(a)}$, by (\ref{e-tB_Lin}), (\ref{rra,v,w_inf}) and (\ref{drra,v,w_L2dec}), we have
	\begin{align*}
	&\int_{t/2}^t \lV  e^{(t-s)B}(|\nab|F_1^{(a)},\nab\cdot \tilde{F}_{2,II}^{(a)})^{\top}\rV_{W^{N(a+2)-2,\infty}}ds\\
	\lesssim & \int_{t/2}^t \<t-s\>^{-5/4} \lV (F_1^{(a)},\tilde{F}_{2,II}^{(a)})\rV_{H^{N(a+2)+1}}ds\\
	\lesssim & \sup_{s\in[t/2,t]}\Big[\sum_{b+c=a}\lV(\rr^{(b)},u^{(b)})\rV_{W^{2,\infty}}\lV(\nab\rr^{(c)},\nab u^{(c)})\rV_{H^{N(a+2)+2}}\\
	&\ \ \ \ \ \ \ \ +\sum_{b+c+e=a}\lV \nab\rr^{(b)}\rV_{H^{N(a+2)+1}}\lV\Phi^{(c)}\rV_{L^{\infty}}\lV\Phi^{(e)}\rV_{H^{N(a+1)}}\Big]\\
	\lesssim & \ep_1^2 t^{-5/4+(4+|a|)\delta}+\ep_1^3 t^{-1/2+(2+|b|)\delta-1+4\delta}\lesssim \ep_1^2 t^{-5/4+(4+|a|)\delta}.
	\end{align*}
	
	{\bf Step 3. Proof of the bound (\ref{drra,v-F2III+Err2}).} It suffices to estimate
	\begin{equation*}
	\int_0^t \lV e^{(t-s)B}\nab^2(0,\rr^{(b)}\nab u^{(c)})^{\top}\rV_{W^{N(a+2)-2,\infty}}ds+\int_0^t \lV e^{(t-s)B}\nab^2(0,\Phi^{(b)}\Phi^{(c)})^{\top}\rV_{W^{N(a+2)-2,\infty}}ds,
	\end{equation*}
	for $b+c=a$. The first term can be estimated by (\ref{e-tB_Lin}), (\ref{Main_Prop_Ass1}), (\ref{rra,v,w_inf}) and (\ref{drra,v,w_L2dec}). For the second term, by (\ref{e-tB_Lin}) and (\ref{Main_Prop_Ass1}), we have
	\begin{equation*}
	\int_0^{t/2} \lV e^{(t-s)B}\nab^2(0,\Phi^{(b)}\Phi^{(c)})^{\top}\rV_{W^{N(a+2)-2,\infty}}ds\lesssim \int_0^{t/2} (t-s)^{-5/2}\lV\Phi^{(b)}\rV_{H^{N(a+1)}}\lV\Phi^{(c)}\rV_{H^{N(a+1)}}ds
	\lesssim \ep_1^2 \<t\>^{-3/2+2\delta}.
	\end{equation*}
	Next, by frequency decomposition, we divide the following integral into low frequency part and high frequency part,
	\begin{align*}
	     \int_{t/2}^t \lV e^{(t-s)B}\nab^2(0,  \big( \Phi^{(c)} \Phi^{(e)} \big))\rV_{W^{N(a+2)-2,\infty}}ds\lesssim & \int_{t/2}^t \lV P_{\leq k(|a|)} e^{(t-s)B}\nab^2(0,  \big( \Phi^{(c)} \Phi^{(e)} \big))\rV_{L^{\infty}}ds\\
	     &+\int_{t/2}^t \lV P_{> k(|a|)} e^{(t-s)B}\nab^2(0,  \big( \Phi^{(c)} \Phi^{(e)} \big))\rV_{W^{N(a+2)-2,\infty}}ds\\
	     =:& I_{\leq k(|a|)}+I_{> k(|a|)},
	\end{align*}
	then we control these two parts separately. By (\ref{e-tB_Lin}) and (\ref{Main_Prop_Ass1}), we have
	\begin{align*}
	I_{\leq k(|a|)}\lesssim & \int_{t/2}^t \<t-s\>^{-5/4} 2^{k(|a|)} \lV \Phi^{(c)}\rV_{L^2}\lV\Phi^{(e)}\rV_{L^{\infty}}ds\\
	\lesssim & 2^{k(|a|)} \ep_1^2 t^{-1+4\delta}\lesssim \ep_1^2 t^{-1-l(|a|)+4\delta}.
	\end{align*}
	By (\ref{e tB_Lp-Lp}), (\ref{Main_Prop_Ass1}) and (\ref{Dec_Phi}), we have
	\begin{align*}
	I_{> k(|a|)}\lesssim & \int_{t/2}^t \<t-s\>^{-3/2}2^{-13k_0}\lV \Phi^{(c)}\rV_{W^{N(a+2),\infty}}\lV \Phi^{(e)}\rV_{W^{N(a+2),\infty}}ds\\
	\lesssim & 2^{-13k_0} \ep_1^2t^{-2+6\delta}\lesssim \ep_1^2 t^{-3/2}.
	\end{align*}
	Thus the bound (\ref{drra,v-F2III+Err2}) follows and the proof of (\ref{drra,v,w_a_inf}) completes.

	{\bf Proof of the bound (\ref{drra,v,w_a_inf_5}).}
	
	By Duhamel's formula, (\ref{Dec_heat}) and (\ref{e-tB_Lin}) it suffices to show that
	\begin{gather}       \label{drra-Lin2}
	\lV\int_0^t  e^{(t-s)B} \nab(0,\Delta u^{(b)})^{\top}+  e^{(t-s)\Delta}|\nab|\P\Delta u^{(b)}ds \rV_{W^{N(a+2)-2,\infty}} \lesssim \ep_1 t^{-1+(2+|a|)\delta},\ {\rm for}\ |b|<|a|,\\ \label{drra-NonL}
	\lV \int_0^t e^{(t-s)B}\nab(F_1^{(a)},F_2^{(a)})^{\top}+ e^{(t-s)\Delta}\nab\P \tilde{F}_2^{(a)}ds\rV_{W^{N(a+2)-2,\infty}} \lesssim \ep_1^2 t^{-1+(2+|a|)\delta}.
	\end{gather}
	By (\ref{e-tB_Lin}), (\ref{Dec_heat}), (\ref{Main_Prop_Ass1}) and (\ref{d2rra,v,w_L2dec}), we have
	\begin{align*}
	&\lV\int_0^t  e^{(t-s)B} \nab(0,\Delta u^{(b)})^{\top}+  e^{(t-s)\Delta}|\nab|\P\Delta u^{(b)}ds \rV_{W^{N(a+2)-2,\infty}} \\
	\lesssim &\int_0^{t/2} (t-s)^{-9/4}\lV u^{(b)}\rV_{H^{N(a+2)+3}}ds+\int_{t/2}^t \<t-s\>^{-5/4}\lV\nab^2 u^{(b)}\rV_{H^{N(a+2)+1}}ds\\
	\lesssim & \ep_1t^{-5/4}+\ep_1t^{-1+(2+|b|)\delta}\lesssim \ep_1 t^{-1+(2+|b|)\delta}.
	\end{align*}
	Next, we prove the bound (\ref{drra-NonL}). From the proof of (\ref{drra,v,w_a_inf}), it suffices to prove that for $b+c=a$
	\begin{align}          \label{d2phi^2}
	\lV \int_0^t e^{(t-s)B}\nab(0,\nab(\Phi^{(b)}\Phi^{(c)}))^{\top}+ e^{(t-s)\Delta}\nab^2 (\Phi^{(b)}\Phi^{(c)}) ds\rV_{W^{N(a+2)-2,\infty}}\lesssim\ep_1^2 t^{-1+(2+|a|)\delta}.
	\end{align}
	We may assume that $|b|\leq |c|$, by (\ref{e-tB_Lin}), (\ref{Dec_heat}), (\ref{Main_Prop_Ass1}) and (\ref{Dec_Phi}) we may bound the left-hand side by
	\begin{align*}
	{\rm LHS}(\ref{d2phi^2})\lesssim &\int_0^{t/2} (t-s)^{-5/2}\lV\Phi^{(b)}\rV_{H^{N(a)}} \lV\Phi^{(c)}\rV_{H^{N(a)}}ds\\
	&+\int_{t/2}^t \<t-s\>^{-7/4}\sum_{k}2^{(N(a+2)+2)k^+}\lV P_k\Phi^{(b)}\rV_{L^{\infty}}\lV \Phi^{(c)}\rV_{H^{N(a)}}ds\\
	\lesssim & \ep_1^2 t^{-3/2+3\delta}+\ep_1^2 t^{-1+4\delta}\lesssim \ep_1^2 t^{-1+(2+|a|)\delta}.
	\end{align*}
	This completes the proof of (\ref{drra,v,w_a_inf_5}).
\end{proof}

    An immediate consequence of the above lemma is the decay estimates for $\d_t\rr^{(a)},\d_t u^{(a)}$ and $\d_t\Psi^{(a)}$.
\begin{lemma}         \label{dt_v_2norm_Lem}
	Under the assumptions (\ref{Main_Prop_Ass1}) and (\ref{Main_Prop_Ass2}), for any $t\in[0,T]$, we have
	
	(i)
	\begin{gather}\label{dtrr-L2}
	\lV\d_t\rr^{(a)}\rV_{H^{N(a)-1}}\lesssim \sum_{|b|\leq |a|} \lV(\nab\rr^{(b)},\nab u^{(b)})\rV_{H^{N(a)-1}},\ {\rm for}\ |a|\leq N_1,\\        \label{dtrr-inf}
	\lV\d_t\rr^{(a)}\rV_{W^{N(a+2)-3,\infty}}+\lV\d_t u^{(a)}\rV_{W^{N(a+2)-4,\infty}}\lesssim \ep_1\<t\>^{-1-l(a)/4},\ \ {\rm for}\ |a|\leq N_1-2,
	\end{gather}
	and
	\begin{gather}   \label{dtu-L2}
	\lV\d_t u^{(a)}\rV_{H^n}\lesssim \sum_{|l|\leq |a|}\lV(\nab\rr^{(l)}, \nab^2 u^{(l)})\rV_{H^{n}}+\ep_1^2\<t\>^{-1+5\delta},\ \ \mathrm{for}\ n\leq N(a)-1,\ |a|\leq N_1.
	\end{gather}
	(ii) For any $|a|\leq N_1-1$,
	\begin{equation}    \label{dtPsi_L2_dec}
	\lV\d_t \Psi^{(a)}\rV_{H^{N(a)-3}}\lesssim \ep_1^2 \<t\>^{-3/4+(2+|a|)\delta}.
	\end{equation}
	In particular, if $|a|\leq N_1-2$, we have
	\begin{equation}       \label{dtPsi_L2_dec-<4}
	\lV\d_t \Psi^{(a)}\rV_{H^{N(a)-1}}\lesssim \ep_1^2 \<t\>^{-3/4+(2+|a|)\delta}+\ep_1\<t\>^{-1+3\delta}\sum_{|b|\leq |a|}\lV\nab u^{(b)}\rV_{H^{N(a)}}.
	\end{equation}
\end{lemma}
\begin{proof}
	{\bf (i)} The first bound (\ref{dtrr-L2}) follows directly from $\rr^{(a)}$-equation in (\ref{Main_Sys_VecFie}). For the second bound (\ref{dtrr-inf}), by $\rr^{(a)}$-equation (\ref{Main_Sys_VecFie}), (\ref{drr,v,w_inf}) and (\ref{drra,v,w_a_inf}) we have
	\begin{align*}
	\lV\d_t\rr^{(a)}\rV_{W^{N(a+2)-3,\infty}}
	\lesssim &\lV\nab\cdot u^{(a)}\rV_{W^{N(a+2)-3,\infty}}+\sum_{b+c=a}\big(\lV u^{(b)}\cdot\nab \rr^{(c)}\rV_{W^{N(a+2)-3,\infty}}+\lV\rr^{(b)}\nab\cdot u^{(c)}\rV_{W^{N(a+2)-3,\infty}}\big)\\
	\lesssim & \lV\nab\cdot u^{(a)}\rV_{W^{N(a+2)-2,p}}+\sum_{|c|\leq |a|}\lV(\nab\cdot u^{(c)},\nab\rr^{(c)})\rV_{W^{N(a+2)-2,p}}\\
	\lesssim & \ep_1\<t\>^{-1-l(a)/4}.
	\end{align*}
	and similarly, by $u^{(a)}$-equation in (\ref{Main_Sys_VecFie}) we have
	\begin{align*}
	\lV\d_t u^{(a)}\rV_{W^{N(a+2)-4,\infty}}\lesssim &\sum_{|b|\leq |a|}\lV \nab^2 u^{(b)}\rV_{W^{N(a+2)-4,\infty}}+\lV \nab \rr^{(a)}\rV_{W^{N(a+2)-4,\infty}}+\sum_{b+c=a}\big[\lV u^{(b)}\nab u^{(c)}\rV_{W^{N(a+2)-4,\infty}}\\
	&+\lV \rr^{(b)}\nab \rr^{(c)}\rV_{W^{N(a+2)-4,\infty}}
	+\lV \rr^{(b)}\nab^2 u^{(c)}\rV_{W^{N(a+2)-4,\infty}}+\lV \nab(\nab\phi^{(b)}\nab \phi^{(c)})\rV_{W^{N(a+2)-4,\infty}}\big]
	\end{align*}
	Then by $|a|\leq N_1-2$, (\ref{drra,v,w_a_inf}) and (\ref{Dec_Phi}), we may bound this by
	\begin{align*}
	\lV\d_t u^{(a)}\rV_{W^{N(a+2)-4,\infty}}\lesssim & \ep_1\<t\>^{-1-l(a)/4}+\ep_1^2\<t\>^{-3/4+(2+|b|)\delta-1}+\ep_1^2\<t\>^{-2+6\delta}\lesssim \ep_1\<t\>^{-1-l(a)/4}.
	\end{align*}
	This concludes the proof of (\ref{dtrr-inf}).

	Next, we prove the bound (\ref{dtu-L2}). In fact, from $u^{(a)}$-equation in (\ref{Main_Sys_VecFie}) we have that 	\begin{align*}
	\lV\d_t u^{(a)}\rV_{H^n}\lesssim &\sum_{|b|\leq |a|}\lV \nab u^{(b)}\rV_{H^{n+1}}+\lV\nab\rr^{(a)}\rV_{H^{n}}\\
	&+\sum_{|b|+|c|\leq |a|}\big[\lV u^{(b)}\cdot\nab u^{(c)}\rV_{H^n}
	+\lV \rr^{(b)}\nab\rr^{(c)}\rV_{H^n}+\lV \rr^{(b)}\nab^2 u^{(c)}\rV_{H^n}\big]\\
	&+\sum_{b+c+e=a}\lV Z^b(\frac{1}{1+\rr})\nab(\nab\phi^{(c)}\nab\phi^{(e)})\rV_{H^n}+\sum_{|c|\leq |a|}\lV  Z^c Err2\rV_{H^n}.
	\end{align*}
	Moreover, from (\ref{Main_Prop_Ass1}) it follows that
	\begin{align*}
	&\lV (u^{(b)}\cdot\nab u^{(c)})\rV_{H^n}+\lV \rr^{(b)}\nab\rr^{(c)}\rV_{H^n}+\lV (\rr^{(b)}\nab^2 u^{(c)})\rV_{H^n}\\
	\lesssim &\lV u^{(b)}\rV_{H^n}\lV(\nab\rr^{(c)},\nab u^{(c)})\rV_{H^n\times H^{n+1}}\lesssim \ep_1 \lV(\nab\rr^{(c)},\nab u^{(c)})\rV_{H^n\times H^{n+1}},
	\end{align*}
	and from (\ref{phi^2-Hna}) and (\ref{Main_Prop_Ass1}) we obtain
	\begin{align*}
	\lV Z^b(\frac{1}{1+\rr})\nab(\nab\phi^{(c)}\nab\phi^{(e)})\rV_{H^n}
	\lesssim  (1+\lV\rr^{(b)}\rV_{H^n})\lV \nab\phi^{(c)}\nab \phi^{(e)}\rV_{H^{n+1}}\lesssim\ep_1^2\<t\>^{-1+4\delta}.
	\end{align*}
	The term $\lV Z^a Err2\rV_{H^n}$ is estimated using similar argument. Hence, the bound (\ref{dtu-L2}) follows.
	
	{\bf (ii)} We prove the bound (\ref{dtPsi_L2_dec}) first. In view of $\phi^{(a)}$-equation in (\ref{Main_Sys_VecFie}), it suffices to prove that
	\begin{equation}       \label{dtPsi_Dec}
	\begin{aligned}
	&\sum_{b+c=a}\big(\lV\rr^{(b)}\nab\phi^{(c)}\rV_{H^{N(a)-3}}+\lV \d_t u^{(b)}\cdot\nab \phi^{(c)}\rV_{H^{N(a)-3}}+\lV  u^{(b)}\cdot\nab \d_t\phi^{(c)}\rV_{H^{N(a)-3}}\big)\\
	&+\sum_{b+c+e=a}\lV u^{(b)}\cdot\nab(u^{(c)}\cdot\nab \phi^{(e)})\rV_{H^{N(a)-3}}+\lV Z^a Err2\rV_{H^{N(a)-3}}\lesssim \ep_1^2\<t\>^{-3/4+(|a|+1)\delta}.
	\end{aligned}
	\end{equation}

	Here we only estimate the first term $\d_t u^{(b)}\cdot\nab \phi^{(c)}$ in detail, the other terms in (\ref{dtPsi_Dec}) are similar.
	
	\emph{Case 1: $|a|\geq N_1-1$.}
	
	When $|b|\leq |c|$, using (\ref{dtrr-inf}) and (\ref{Main_Prop_Ass1}), it follows that
	\begin{equation}\label{dtPsi_N1-1_b<c}
	\begin{aligned}
	\lV \d_t u^{(b)}\cdot\nab \phi^{(c)}\rV_{H^{N(a)-3}}\lesssim &\big[\sum_{k\geq 0}2^{(N(a)-3)k^+}\lV \d_t P_k u^{(b)}\rV_{\infty}+\lV \d_t u^{(b)}\rV_{\infty}\big]\lV\Phi^{(c)}\rV_{H^{N(a)-3}}\\
	\lesssim &\ep_1\<t\>^{H(c)\delta}\lV \d_t u^{(b)}\rV_{W^{N(|b|+2)-4,\infty}}\\
	\lesssim & \ep_1\<t\>^{2\delta}\lV \d_t u^{(b)}\rV_{W^{N(|b|+2)-2,\infty}}
	\lesssim \ep_1^2 \<t\>^{-1}.
	\end{aligned}
	\end{equation}
	If $|b|>|c|$, it follows from (\ref{dtu-L2}) and (\ref{Dec_Phi}) that
	\begin{equation}     \label{dtPsi_N1-1_b>c}
	\lV \d_t u^{(b)}\cdot\nab \phi^{(c)}\rV_{H^{N(a)-3}}\lesssim \lV \d_t u^{(b)}\rV_{H^{N(a)-3}}\sum_k 2^{(N(a)-3)k^+}\lV\Phi^{(c)}\rV_{\infty}\lesssim \ep_1^2\<t\>^{-1+3\delta}.
	\end{equation}
	
	\emph{Case 2: $|a|\leq N_1-2$.}
	
	Using (\ref{dtrr-inf}), (\ref{dtu-L2}), (\ref{Dec_Phi}) and (\ref{Main_Prop_Ass1}), we get
	\begin{equation}      \label{dtPsi_N1-2}
	\lV \d_t u^{(b)}\cdot\nab \phi^{(c)}\rV_{H^{N(a)-3}}\lesssim \lV \d_t u^{(b)}\rV_{H^{N(a)-3}}\lV \Phi^{(c)}\rV_{\infty}+\lV \d_t u^{(b)}\rV_{\infty}\lV\Phi^{(c)}\rV_{H^{N(a)-3}}\lesssim \ep_1^2\<t\>^{-1+5\delta}.
	\end{equation}
	From (\ref{dtPsi_N1-1_b<c}), (\ref{dtPsi_N1-1_b>c}) and (\ref{dtPsi_N1-2}), the bound (\ref{dtPsi_Dec}) for the first term is obtained.
	
	Finally, we prove the bound (\ref{dtPsi_L2_dec-<4}). By $\phi^{(a)}$-equation in (\ref{Main_Sys_VecFie}), it suffices to show
	\begin{equation*}
	\begin{aligned}
	&\sum_{b+c=a}\big(\lV\rr^{(b)}\nab\phi^{(c)}\rV_{H^{N(a)-1}}+\lV \d_t u^{(b)}\cdot\nab \phi^{(c)}\rV_{H^{N(a)-1}}+\lV  u^{(b)}\cdot\nab \d_t\phi^{(c)}\rV_{H^{N(a)-1}}\big)\\
	&+\sum_{b+c+e=a}\lV u^{(b)}\cdot\nab(u^{(c)}\cdot\nab \phi^{(e)})\rV_{H^{N(a)-1}}+\lV Z^a Err2\rV_{H^{N(a)-1}}\\
	\lesssim &\ep_1^2\<t\>^{-3/4+(2+|a|)\delta}+\ep_1\<t\>^{-1+3\delta}\sum_{|b|\leq |a|}\lV \nab u^{(b)}\rV_{H^{N(a)}}.
	\end{aligned}
	\end{equation*}
	Here we only estimate the term
	\begin{equation*}
	\lV \d_t u^{(b)}\cdot\nab \phi^{(c)}\rV_{H^{N(a)-1}}
	\end{equation*}
	in detail. By $|a|\leq N_1-2$, (\ref{dtu-L2}), (\ref{Dec_Phi}) and (\ref{dtrr-inf}), we may bound this by
	\begin{align*}
	\lV \d_t u^{(b)}\cdot\nab \phi^{(c)}\rV_{H^{N(a)-1}}\lesssim &\lV\d_t u^{(b)}\rV_{H^{N(a)-1}}\lV\nab\phi^{(c)}\rV_{L^{\infty}}+\lV\d_tu^{(b)}\rV_{L^{\infty}}\lV\nab\phi^{(c)}\rV_{H^{N(a)-1}}\\
	\lesssim & ( \sum_{|l|\leq |b|}\lV \nab u^{(l)}\rV_{H^{n+1}}+\ep_1^2\<t\>^{-1+5\delta})\ep_1\<t\>^{-1+3\delta}+\ep_1^2\<t\>^{-1-l(b)/4+H(c)\delta}\\
	\lesssim & \ep_1^2\<t\>^{-1}+\ep_1\<t\>^{-1+3\delta}\sum_{|l|\leq |b|}\lV \nab u^{(l)}\rV_{H^{n+1}}.
	\end{align*}
	This completes the proof of the Lemma.

\end{proof}

\section{Energy estimates, I: Sobolev bounds}
In this section we prove the energy bounds (\ref{Main_Prop_result1}) for $a=0$.
\subsection{The bound on $\rr$ and $u$} We start with the Sobolev bound in (\ref{Main_Prop_result1}).

\begin{proposition}            \label{Prop_Ene_Sob}
	With the notation and hypothesis in Proposition $\ref{Main_Prop}$, for any $t\in[0,T]$, we have
	\begin{equation}         \label{E-rr&u}
	\lV \rr\rV_{H^{N(0)}}^2+\lV u\rV_{H^{N(0)}}^2+\int_0^t\big( \lV\nab \rr\rV_{H^{N(0)-1}}^2 + \lV\nab u\rV_{H^{N(0)}}^2+ \lV\div u\rV_{H^{N(0)}}^2\big) ds\lesssim \epsilon_0^2.
	\end{equation}
\end{proposition}
\begin{proof}
{\bf Step 1. Higher order estimate for $\rr$ and $u$.}	

We define the energy functional
\begin{align*}
E_0(t)=&\sum\limits_{|n|\leq N(0)}\big[\frac{1}{2}\int_{\R^3}\frac{P'(1+\rr)}{1+\rr}|\d^{n} \rr|^2 +(1+\rr)|\d^n u|^2dx
+\int_0^t \int_{\R^3} \frac{\mu_1}{2} |\d^{n}\nab u|^2+(\frac{\mu_1}{2}+\mu_2)|\d^n \div u|^2 dx ds\big].
\end{align*}

From (\ref{Main_Sys}), we have
\begin{equation}       \label{dtE0}
\begin{aligned}
	\frac{d}{dt}E_0(t)=&\sum_{|n|\leq N(0)}\int_{\R^3}\frac{1}{2}\d_t(\frac{P'(1+\rr)}{1+\rr})|\d^n \rr|^2-\frac{P'(1+\rr)}{1+\rr}\d^n\rr\cdot\d^n(u\cdot\nab\rr+\rr\nab\cdot u)dx\\
	&-\sum_{|n|\leq N(0)}\int_{\R^3} \frac{P'(1+\rr)}{1+\rr} \d^n\rr \nab\cdot\d^n u+(1+\rr)\d^n u\cdot\d^n\big(\frac{P'(1+\rr)}{1+\rr}\nab\rr\big) dx \\
	&+\sum_{|n|\leq N(0)} \int_{\R^3} \frac{1}{2}\d_t \rr|\d^n u|^2+(1+\rr)\d^n u\cdot \d^n \mathcal{N} dx\\
	&+\sum_{|n|\leq N(0)}\int_{\R^3} (1+\rr)\d^n u\cdot\sum_{|n_1|\geq 1,n_1+n_2=n}\d^{n_1}( \frac{1}{1+\rr})\d^{n_2}\big(\frac{\mu_1}{2}\Delta u+(\frac{\mu_1}{2}+\mu_2)\nab\div u\big)dx\\
	=&R_1+R_2+R_3+R_4,
\end{aligned}
\end{equation}
where $\mathcal{N}$ denoted as
\begin{equation*}
\mathcal{N}:=-u\cdot\nab u+\frac{1}{1+\rr}\big[\frac{1}{2}\nab(\cos^2 \phi_2|\nab\phi_1|^2+|\nab\phi_2|^2)-\sum_{j=1}^3\d_j(\cos^2\phi_2 \nab\phi_1 \d_j \phi_1+\nab\phi_2\d_j\phi_2) \big].
\end{equation*}

Then by using H\"{o}lder's inequality, Sobolev embedding, (\ref{Main_Sys}) and the assumptions (\ref{Main_Prop_Ass1}), we have the estimates of $R_i(i=1,\cdots,4)$ in the following.

For the estimate of $R_1$, note that $P(1+\rr)=a(1+\rr)^{\gamma}$, by (\ref{dtrr-L2}) we get
\begin{align*}
R_1\lesssim&\lV(1+\rr)^{\gamma-3}\rV_{L^{\infty}}\big(\lV\d_t \rr\rV_{L^2}\lV \rr\rV_{L^2} \lV \rr\rV_{L^{\infty}}+\lV\d_t\rr\rV_{L^{\infty}} \lV\nab^{N(0)}\rr\rV_{L^2}^2\big)+\lV\frac{P'(1+\rr)}{1+\rr}\rV_{L^{\infty}}\lV\rr\rV_{H^{N(0)-1}} \lV u\cdot\nab\rr\rV_{H^{N(0)-1}}\\
&-\sum_{|n|=N(0)}\int_{\R^3} \frac{P'(1+\rr)}{1+\rr}\d^n\rr\cdot(u\cdot\nab\d^n\rr+\sum_{|n_1|\geq 1,n_1+n_2=n}\d^{n_1}u\cdot\nab\d^{n_2}\rr)dx\\
&+\lV\frac{P'(1+\rr)}{1+\rr}\rV_{L^{\infty}}\lV\rr\rV_{H^{N(0)}} \lV \rr\nab\cdot u\rV_{H^{N(0)}}\\
\lesssim & (\lV\nab u\rV_{H^{N(0)}}+\lV\nab \rr\rV_{H^{N(0)-1}})\lV\rr\rV_{H^{N(0)}}\lV\nab\rr\rV_{H^{N(0)-1}}-\frac{1}{2}\sum_{|n|=N(0)}\int_{\R^3} \frac{P'(1+\rr)}{1+\rr}u\cdot\nab|\d^n\rr|^2 dx,
\end{align*}
By integration by parts, Sobolev embedding and (\ref{Main_Prop_Ass1}), we may bound the second term by
\begin{align*}
-\frac{1}{2}\sum_{|n|=N(0)}\int_{\R^3} \frac{P'(1+\rr)}{1+\rr}u\cdot\nab|\d^n\rr|^2 dx=&\frac{1}{2}\sum_{|n|=N(0)}\int_{\R^3} (\nab\frac{P'(1+\rr)}{1+\rr}u+\frac{P'(1+\rr)}{1+\rr}\nab\cdot u)|\d^n\rr|^2 dx\\
\lesssim & (\lV\nab\rr\rV_{H^2}+1)\lV \nab u\rV_{H^2}\lV\nab\rr\rV_{H^{N(0)-1}}^2\lesssim \lV \nab u\rV_{H^2}\lV\nab\rr\rV_{H^{N(0)-1}}^2.
\end{align*}
Thus we have
\begin{align*}
R_1\lesssim (\lV\nab u\rV_{H^{N(0)}}+\lV\nab\rr\rV_{H^{N(0)-1}})^2\lV\rr\rV_{H^{N(0)}},
\end{align*}
which implies
\begin{equation}
\begin{aligned}
\int_0^t R_1(s)ds\lesssim (\lV \nab u\rV_{L^2([0,T]; H^{N(0)})}^2+\lV \nab\rr\rV_{L^2([0,T]; H^{N(0)-1})}^2)\lV \rr\rV_{L^{\infty}([0,T]; H^{N(0)})}
\lesssim  \ep_1^3.
\end{aligned}
\end{equation}

For $R_2$, let $e\in\{(1,0,0),(0,1,0),(0,0,1)\}$, integration by parts gives
\begin{align*}
R_2=&-\sum_{|n|\leq N(0)}\int_{\R^3} \frac{P'(1+\rr)}{1+\rr} \d^n\rr \nab\cdot\d^n u+\d^n u\cdot\frac{P'(1+\rr)}{1+\rr} \nab\d^n\rr\\
&+\d^n u\cdot \sum_{|n_1|\geq 1,n_1+n_2=n}\d^{n_1}(\frac{P'(1+\rr)}{1+\rr})\nab\d^{n_2}\rr +\rr\d^n u\cdot\d^n\big(\frac{P'(1+\rr)}{1+\rr}\nab\rr\big) dx\\
=&-\sum_{|n|\leq N(0)}\int_{\R^3}-\nab(\frac{P'(1+\rr)}{1+\rr})\cdot\d^n u\d^n\rr+\d^n u\cdot \sum_{|n_1|\geq 1,n_1+n_2=n}\d^{n_1}(\frac{P'(1+\rr)}{1+\rr})\nab\d^{n_2}\rr\\
&-\d^{e}(\rr\d^n u)\cdot \d^{n-e}(\frac{P'(1+\rr)}{1+\rr}\nab\rr) dx\\
\lesssim & \lV  \nab\rr\rV_{H^{N(0)-1}}\lV\nab u\rV_{H^{N(0)}} \lV\rr\rV_{H^{N(0)}}.
\end{align*}
Then from (\ref{Main_Prop_Ass1}), we get
\begin{equation}
\int_0^t R_2(s)ds\lesssim \ep_1^3.
\end{equation}

Finally, we estimate the terms $R_3$ and $R_4$. By $\rr$-equation in (\ref{Main_Sys}), Sobolev embedding and (\ref{Main_Prop_Ass1}), one has
\begin{align*}
&\sum_{|n|\leq N(0)}\int_{\R^3} \frac{1}{2}\d_t \rr|\d^n u|^2+(1+\rr)\d^n u\cdot \d^n (u\cdot\nab u) dx+R_4\\
\lesssim & \lV\d_t\rr\rV_{L^2} \lV u\rV_{L^2}\lV u\rV_{L^{\infty}}+\lV\d_t\rr\rV_{L^{\infty}} \lV \nab u\rV_{H^{N(0)}}^2+\lV 1+\rr\rV_{L^{\infty}}\lV u\rV_{H^{N(0)}}(\lV\nab u\rV_{H^{N(0)}}^2+\lV\nab\rr\rV_{H^{N(0)-1}}\lV\nab u\rV_{H^{N(0)}})\\
\lesssim & \lV u\rV_{H^{N(0)}}\lV\nab u\rV_{H^{N(0)}}^2+\lV u\rV_{H^{N(0)}}(\lV\nab u\rV_{H^{N(0)}}^2+\lV\nab\rr\rV_{H^{N(0)-1}}\lV\nab u\rV_{H^{N(0)}})\\
\lesssim &\lV u\rV_{H^{N(0)}}(\lV\nab u\rV_{H^{N(0)}}+\lV\nab\rr\rV_{H^{N(0)-1}})^2,
\end{align*}
For the other terms in $R_3$, by integration by parts, Sobolev embedding and (\ref{Main_Prop_Ass1}), we have
\begin{align*}
&\sum_{|n|\leq N(0)}\int_{\R^3} (1+\rr)\d^n u\cdot\d^n (\mathcal{N}+u\cdot\nab u)dx\\
=& \int_{\R^3} -\frac{1}{2}\nab\cdot u (\cos^2 \phi_2|\nab\phi_1|^2+|\nab\phi_2|^2)+\sum_{j=1}^3 \d_j u (\cos^2\phi_2 \nab\phi_1 \d_j \phi_1+\nab\phi_2\d_j\phi_2) dx\\
&-\sum_{1\leq|n|\leq N(0)}\int_{\R^3} \d^e ((1+\rr)\d^n u)\cdot\d^{n-e}(\mathcal{N}+u\cdot\nab u)dx\\
\lesssim & \lV \nab u\rV_{H^{N(0)}}\lV\nab\phi\rV_{H^{N(0)}}\lV\nab\phi\rV_{L^{\infty}}.
\end{align*}
Hence, by H\"{o}lder, (\ref{Main_Prop_Ass1}) and (\ref{Dec_Phi}) we get
\begin{align*}
\int_0^t R_3(s)+R_4(s)ds\lesssim &\int_0^t\lV u\rV_{H^{N(0)}}(\lV\nab u\rV_{ H^{N(0)}}+\lV\nab\rr\rV_{H^{N(0)-1}})^2
+\lV \nab u\rV_{H^{N(0)}}\lV\nab\phi\rV_{ H^{N(0)}}\lV\nab\phi(s)\rV_{L^{\infty}}ds\\
\lesssim & \ep_1^3+\ep_1^2 \int_0^t \lV\nab u\rV_{ H^{N(0)}}\<s\>^{-1+2\delta} ds\lesssim \ep_1^3.
\end{align*}
This concludes the bound
\begin{equation}       \label{E0<ep02}
E_0(t)\lesssim E_0(0)+\int_0^t \d_s E_0(s)ds\lesssim \ep_0^2.
\end{equation}
Using (\ref{Main_Prop_Ass1}), it is easy to obtain
\begin{equation}             \label{1-(1+rr)gamma}
\begin{aligned}
&\lV \frac{P'(1+\rr)}{1+\rr}-a\gamma\rV_{L^{\infty}}=a\gamma\lV (1+\rr)^{\gamma-2}-1\rV_{L^{\infty}}\\
=&a\gamma\lV\int_0^{\rr}\frac{d}{d s}(1+s)^{\gamma-2}ds\rV_{L^{\infty}}\lesssim \lV (1+\rr)^{\gamma-3}\rV_{L^{\infty}}\lV\rr\rV_{L^{\infty}}\lesssim  \ep_1.
\end{aligned}
\end{equation}
Therefore, combining (\ref{E0<ep02}) and (\ref{1-(1+rr)gamma}) we further obtain
\begin{equation}          \label{E_rr&u-inequ}
\frac{a\gamma}{4}\lV\rr\rV_{H^{N(0)}}^2 +\frac{1}{4}\lV u\rV_{H^{N(0)}}^2 +\int_0^t \frac{\mu_1}{2}\lV\nab u\rV_{H^{N(0)}}^2+(\frac{\mu_1}{2}+\mu_2)\lV\div u\rV_{H^{N(0)}}^2 ds\lesssim E_0(t)\lesssim  \ep_0^2.
\end{equation}

{\bf Step 2. The estimate for dissipation of density.}

Taking $\d^n$ on the $u$-equation of (\ref{Main_Sys}) and multiplying $\d^n\nab \rr$, by integration by parts we have
\begin{equation}            \label{Diss_rr}
\begin{aligned}
&\sum_{|n|\leq N(0)-1}\int_{\R^3}\d^n\nab\rr\cdot\d^n\d_t u dx+\sum_{|n|\leq N(0)-1}\int_{\R^3}\d^n\nab\rr\cdot\d^n(\frac{P'(1+\rr)}{1+\rr}\nab\rr)dx\\
=& \sum_{|n|\leq N(0)-1}\int_{\R^3}\d^n\nab\rr\cdot\d^n\big(\frac{1}{1+\rr}(\frac{\mu_1}{2}\Delta u+(\frac{\mu_1}{2}+\mu_2)\nab\div u)+\mathcal{N}\big)dx.
\end{aligned}
\end{equation}

We use integration by parts and $\rr$-equation in (\ref{Main_Sys}) to rewrite the first term on the left hand side of \eqref{Diss_rr} by
\begin{equation}          \label{Diss_rr_L1}
\begin{aligned}
&\sum_{|n|\leq N(0)-1}\int_{\R^3}\d^n\nab\rr\cdot\d^n\d_t udx\\
=& \sum_{|n|\leq N(0)-1}\big(\frac{d}{dt}\int_{\R^3}\d^n\nab\rr\cdot\d^n udx-\int_{\R^3}\d^n\d_t\nab\rr\cdot\d^n udx\big)\\
=&\sum_{|n|\leq N(0)-1}\big(\frac{d}{dt}\int_{\R^3}\d^n\nab\rr\cdot\d^n udx+\int_{\R^3}\d^n\nab(\nab\cdot u+u\cdot\nab\rr+\rr\nab\cdot u)\cdot\d^n udx\big)\\
=&\sum_{|n|\leq N(0)-1}\big(\frac{d}{dt}\int_{\R^3}\d^n\nab\rr\cdot\d^n udx-\lV\d^n\nab u\rV_{L^2}^2-\int_{\R^3} \d^n(u\cdot\nab\rr+\rr\nab\cdot u)\d^n\nab\cdot u dx\big).
\end{aligned}
\end{equation}
For the second term on the left hand side of (\ref{Diss_rr}), one has
\begin{equation}              \label{Diss_rr_L2}
\begin{aligned}
&\sum_{|n|\leq N(0)-1}\int_{\R^3}\d^n\nab\rr\cdot\d^n(\frac{P'(1+\rr)}{1+\rr}\nab\rr)dx\\
=&\sum_{|n|\leq N(0)-1}\big(\int_{\R^3}\frac{P'(1+\rr)}{1+\rr}|\d^n\nab\rr|^2dx+\int_{\R^3} \d^n\nab\rr \cdot\sum_{|n_1|\geq 1,n_1+n_2=n}\d^{n_1}(\frac{P'(1+\rr)}{1+\rr})\d^{n_2}\nab\rr dx\big).
\end{aligned}
\end{equation}
Then from (\ref{Diss_rr})-(\ref{Diss_rr_L2}), we obtain
\begin{equation}         \label{LHS4.11}
\begin{aligned}
&\sum_{|n|\leq N(0)-1}(\frac{d}{dt}\int_{\R^3}\d^n\nab\rr\cdot\d^n udx-\lV\d^n\nab u\rV_{L^2}^2+\int_{\R^3}\frac{P'(1+\rr)}{1+\rr}|\d^n\nab\rr|^2dx)\\
=&\sum_{|n|\leq N(0)-1}\int_{\R^3} \d^n(u\cdot\nab\rr+\rr\nab\cdot u)\d^n\nab\cdot u dx-\d^n\nab\rr \cdot\sum_{|n_1|\geq 1,n_1+n_2=n}\d^{n_1}(\frac{P'(1+\rr)}{1+\rr})\d^{n_2}\nab\rr\\
&+\d^n\nab\rr\cdot\d^n\big(\frac{1}{1+\rr}(\frac{\mu_1}{2}\Delta u+(\frac{\mu_1}{2}+\mu_2)\nab\div u)+\mathcal{N}\big)dx.
\end{aligned}
\end{equation}
By Sobolev embedding, (\ref{Main_Prop_Ass1}) and the following observation
\begin{equation*}
\sum_{|n|\leq N(0)-1}\int_{\R^3}\d^n \nab\rr\cdot\d^n\Delta udx\leq \frac{1}{4}\sum_{|n|\leq N(0)-1}\int_{\R^3}\frac{P'(1+\rr)}{1+\rr}|\d^n\nab\rr|^2dx+C(\mu_1,\mu_2)\lV \nab u\rV_{H^{N(0)}}^2,
\end{equation*}
we get
\begin{align*}
{\rm LHS}(\ref{LHS4.11})\leq & C\lV\nab u\rV_{H^{N(0)}}^2\lV \rr\rV_{H^{N(0)}}+C\lV\nab\rr\rV_{H^{N(0)-1}}^2\lV\rr\rV_{H^{N(0)}}+\frac{1}{4}\sum_{|n|\leq N(0)-1}\int_{\R^3}\frac{P'(1+\rr)}{1+\rr}|\d^n\nab\rr|^2dx\\
&+C(\mu_1,\mu_2)\lV \nab u\rV_{H^{N(0)}}^2+\lV\nab\rr\rV_{H^{N(0)-1}}(\lV\nab u\rV_{H^{N(0)}}^2+C\lV \nab\phi\rV_{H^{N(0)}}\lV\nab\phi\rV_{L^{\infty}}).
\end{align*}
Noticing that
\begin{equation*}
\d^n\nab\rr\cdot\d^n u=\frac{1}{2}(|\d^n(u+\nab\rr)|^2-|\d^n\nab\rr|^2-|\d^n u|^2),
\end{equation*}
it follows that
\begin{align*}
&\sum_{|n|\leq N(0)-1}\frac{1}{2}\frac{d}{dt}(\lV\d^n(u+\nab\rr)\rV_{L^2}^2-\lV\d^n\nab\rr\rV_{L^2}^2-\lV\d^n u\rV_{L^2}^2)\\
&+\sum_{|n|\leq N(0)-1}\frac{3}{4}\int_{\R^3}\frac{P'(1+\rr)}{1+\rr}|\d^n\nab\rr|^2dx-C(\mu_1,\mu_2)\lV\nab u\rV_{H^{N(0)}}^2\\
\lesssim & \lV\nab u\rV_{H^{N(0)}}^2\lV \rr\rV_{H^{N(0)}}+C\lV\nab\rr\rV_{H^{N(0)-1}}^2\lV\rr\rV_{H^{N(0)}}+\lV\nab\rr\rV_{H^{N(0)-1}}(\lV\nab u\rV_{H^{N(0)}}^2+C\lV \nab\phi\rV_{H^{N(0)}}\lV\nab\phi\rV_{L^{\infty}}).
\end{align*}
Hence, by bootstrap assumption (\ref{Main_Prop_Ass1}), H\"{o}lder and (\ref{Dec_Phi}), we have
\begin{equation}              \label{Diss_rr_ineq}
\begin{aligned}
&\frac{1}{2}(\lV u+\nab\rr\rV_{H^{N(0)-1}}^2-\lV\nab\rr\rV_{H^{N(0)-1}}^2-\lV u\rV_{H^{N(0)-1}}^2)\\
&+\frac{3}{4}\sum_{l=1}^{N(0)}\int_0^t\int_{\R^3} \frac{P'(1+\rr)}{1+\rr}|\nab^l\rr|^2dxds-C(\mu_1,\mu_2)\int_0^t \lV\nab u\rV_{H^{N(0)}}^2 ds\lesssim \ep_0^2.
\end{aligned}
\end{equation}

{\bf Step 3. We prove the energy estimate (\ref{E-rr&u})}

In order to prove the bound (\ref{E-rr&u}), choosing
\begin{equation*}
c_1:=\min\{\frac{1}{4},\frac{\mu_1}{4C(\mu_1,\mu_2)}\}
\end{equation*}
and multiplying (\ref{Diss_rr_ineq}) by $c_1$, then adding one to the inequality (\ref{E_rr&u-inequ}), we have
\begin{align*}
&\frac{c_1}{2}\lV u+\nab\rr\rV_{H^{N(0)-1}}^2+\frac{a\gamma}{8}\lV\rr\rV_{H^{N(0)}}^2 +\frac{1}{8}\lV u\rV_{H^{N(0)}}^2 +\frac{3c_1}{4}\sum_{l=1}^{N(0)}\int_0^t\int_{\R^3} \frac{P'(1+\rr)}{1+\rr}|\nab^l\rr|^2dxds\\
&+\int_0^t \frac{\mu_1}{4}\lV\nab u\rV_{H^{N(0)}}^2+(\frac{\mu_1}{2}+\mu_2)\lV\div u\rV_{H^{N(0)}}^2 ds\lesssim  \ep_0^2,
\end{align*}
which implies the bound (\ref{E-rr&u}).

\end{proof}

\subsection{The bound on $\phi$}
\begin{proposition}            \label{Prop_Ene_phi}
	With the notation and hypothesis in Proposition $\ref{Main_Prop}$, for any $t\in[0,T]$, we have
	\begin{equation*}
	\lV \d_t \phi\rV_{H^{N(0)}}^2+\lV \nab\phi\rV_{H^{N(0)}}^2\lesssim \epsilon_0^2.
	\end{equation*}
\end{proposition}
\begin{proof}
    By (\ref{Main_Sys}), integration by parts, Sobolev embedding, (\ref{dtrr-L2}), (\ref{dtu-L2}) and (\ref{Dec_Phi}), it is easy to obtain the basic energy estimate
    \begin{align*}
    &\frac{1}{2}\frac{d}{dt}\int_{\R^3}(1+\rr)|\d_t\phi|^2+|\nab\phi|^2 dx\\
    =&\int_{\R^3} \frac{1}{2}\d_t\rr|\d_t\phi|^2+(1+\rr)\d_t\phi(-\d_t u\cdot\nab\phi-u\cdot\nab\d_t\phi-u\cdot\nab(u\cdot\nab\phi)+Err)dx\\
    \lesssim & \lV      \d_t\rr\rV_{L^2}\lV\d_t\phi\rV_{L^2}\lV\d_t\phi\rV_{L^{\infty}}+\lV\d_t\phi\rV_{L^2}\lV\d_t u\rV_{L^2}\lV\nab\phi\rV_{L^{\infty}}+\frac{1}{2}\int_{\R^3} ((1+\rr)\nab\cdot u+u\cdot\nab\rr)|\d_t\phi|^2dx\\
    &+\lV\d_t\phi\rV_{L^2}\lV\nab u\rV_{H^2}^2\lV\nab\phi\rV_{H^1}+\lV\d_t\phi\rV_{L^2}\lV\phi\rV_{L^{\infty}}\lV\nab\phi\rV_{L^2}\lV\nab\phi\rV_{L^{\infty}}\\
    \lesssim & \ep_1^2 \lV(\nab\rr,\nab u)\rV_{H^{N(0)-1}}\<t\>^{-1+2\delta}+\ep_1^4\<t\>^{-2+7\delta}+\ep_1^2(\lV\nab u\rV_{L^2}+\lV\nab\rr\rV_{L^2})\<t\>^{-1+2\delta}
    +\ep_1^2\lV\nab u\rV_{H^2}^2+\ep_1^4\<t\>^{-5/4},
    \end{align*}
    then from integration in time, we obtain
    \begin{equation}      \label{Ephi_n=0}
    \frac{1}{2}\int_0^t\int_{\R^3}(1+\rr)|\d_t\phi|^2+|\nab\phi|^2 dxds\lesssim \ep_0^2+\ep_1^3\lesssim \ep_0^2.
    \end{equation}

    Now it suffices to consider the following higher order energy
    \begin{equation*}
    E_1(t):=\frac{1}{2}\sum_{|n|=N(0)}\int_{\R^3}(1+\rr)|\d^n\d_t\phi|^2+|\d^n\nab\phi|^2 dx.
    \end{equation*}
    From the system (\ref{Main_Sys}), we have
    \begin{align*}
    \frac{d}{dt}E_1(t)=&\sum_{|n|=N(0)}\int_{\R^3}\frac{1}{2} \d_t\rr|\d^n\d_t\phi|^2 +(1+\rr)\d^n \d_t\phi \sum_{|n_1|\geq 1,n_1+n_2=n}\d^{n_1}(\frac{1}{1+\rr})\d^{n_2}\Delta\phi dx\\
    &+\sum_{|n|=N(0)}\int_{\R^3} (1+\rr)\d^n \d_t\phi \cdot\d^n (-\d_t u\cdot\nab\phi-2u\cdot\nab\d_t\phi-u\cdot\nab(u\cdot\nab\phi)+Err3)dx.
    \end{align*}
    By (\ref{dtrr-inf}) (\ref{drr,v,w_inf}), (\ref{Dec_Phi}) and (\ref{Main_Prop_Ass1}),  the first term can be estimated as
    \begin{equation}    \label{E1-two}
    \begin{aligned}
    &\int_0^t\sum_{|n|=N(0)}\int_{\R^3}\frac{1}{2} \d_s\rr|\d^n\d_s\phi|^2 +(1+\rr)\d^n \d_s\phi \sum_{|n_1|\geq 1,n_1+n_2=n}\d^{n_1}(\frac{1}{1+\rr})\d^{n_2}\Delta\phi dxds\\
    \lesssim & \int_0^t \lV\d_s\rr\rV_{L^{\infty}}\lV\Phi\rV_{H^{N(0)}}^2+\lV\Phi\rV_{H^{N(0)}}(\lV\<\nab\>^{N(0)/2}|\nab|\rr\rV_{L^{\infty}}\lV\Phi\rV_{H^{N(0)}}\\
    &+\lV|\nab|\rr\rV_{H^{N(0)-1}}\lV\<\nab\>^{N(0)/2}\nab\phi\rV_{L^{\infty}})ds\\
    \lesssim & \int_0^t \ep_1^3 \<s\>^{-1-l(0)+3\delta}+\ep_1(\ep_1^2\<s\>^{-1-l(0)+3\delta}+\ep_1\<s\>^{-1+2\delta}\lV|\nab|\rr\rV_{H^{N(0)-1}})ds\\
    \lesssim & \ep_1^3.
    \end{aligned}
    \end{equation}
    Next we consider the second term.

    {\bf Step 1. We prove the bound}
    \begin{equation}    \label{E_I1}
    I_1:=\int_0^t\int_{\R^3}(1+\rr) \d^n \d_s\phi\cdot \d^n(\d_s u\cdot\nab \phi)dxds\lesssim \ep_1^3,\ \ {\rm for}\ |n|=N(0).
    \end{equation}
    Integration by parts in time gives
    \begin{align*}
    I_1= &\int_{\R^3}(1+\rr)\d^n \d_s\phi\cdot \d^n(u\cdot\nab\phi)dx\big|_0^t-\int_0^t\int_{\R^3}\d_s\rr\ \d^n \d_s\phi\cdot \d^n( u\cdot\nab \phi)dxds\\
    &-\int_0^t\int_{\R^3} (1+\rr)\d^n \d_s\phi\cdot \d^n( u\cdot\nab \d_s\phi)dxds-\int_0^t\int_{\R^3}(1+\rr) \d^n \d_s^2\phi\cdot \d^n( u\cdot\nab \phi)dxds\\
    =&I_{11}-I_{12}-I_{13}-I_{14}.
    \end{align*}
    By (\ref{Main_Prop_Ass1}) and (\ref{dtrr-L2}), the first two terms $I_{11}$ and $I_{12}$ can be estimated as
    \begin{align*}
    |I_{11}-I_{12}|\lesssim \lV\Phi\rV_{H^{N(0)}}^2\lV u\rV_{H^{N(0)}}+\int_0^t \lV\d_s\rr\rV_{L^2}\lV\Phi\rV_{H^{N(0)}}^2\lV\nab u\rV_{H^{N(0)}}ds
    \lesssim \ep_1^3+\int_0^t \lV\Phi\rV_{H^{N(0)}}^2\lV\nab u\rV_{H^{N(0)}}^2ds\lesssim \ep_1^3.
    \end{align*}
    For the third term $I_{13}$, from integration by parts and H\"{o}lder, we have
    \begin{align*}
    I_{13}=&\int_0^t\int_{\R^3}\frac{1}{2}(1+\rr) u\cdot\nab|\d^n\d_s\phi|^2+(1+\rr)\d^n \d_s\phi\cdot \sum_{n_1+n_2=n,n_2<n}(\d^{n_1} u\cdot\nab \d^{n_2}\d_s\phi)dxds\\
    \lesssim & \int_0^t\int_{\R^3} -\frac{1}{2}(u\cdot\nab\rr+(1+\rr)\nab\cdot u)|\d^n\d_s\phi|^2 dxds
    \\
    &+\int_0^t\lV\Phi\rV_{H^{N(0)}}(\lV\<\nab\>^{|n|/2}\nab u\rV_{L^{\infty}}\lV\Phi\rV_{H^{N(0)}}+\lV\nab u\rV_{H^{N(0)}}\lV\<\nab\>^{|n|/2}\Phi\rV_{L^{\infty}})ds,
    \end{align*}
    then by (\ref{drr,v,w_inf}), (\ref{Dec_Phi}) and (\ref{Main_Prop_Ass1}) we may bound this by
    \begin{align*}
    I_{13}\lesssim &\int_0^t \lV(\nab\rr,\nab u)\rV_{H^2}^2\lV\Phi\rV_{H^{N(0)}}^2ds+\ep_1\int_0^t \big(\ep_1^2\<s\>^{-1-l(0)+3\delta}+\lV\nab u\rV_{H^{N(0)}}\<s\>^{-1+2\delta}\big)ds\\
    \lesssim & \ep_1^4+\ep_1^3
    \lesssim \ep_1^3 .
    \end{align*}
    For the last term $I_{14}$, it follows from (\ref{Main_Sys}) that
    \begin{align*}
    I_{14}= & \int_0^t\int_{\R^3} (1+\rr)\d^n \Big(\frac{\Delta\phi}{1+\rr}\Big)\cdot \d^n(u\cdot\nab\phi) dxds-\int_0^t\int_{\R^3}  (1+\rr)\d^n \d_s(u\cdot\nab\phi)\cdot \d^n(u\cdot\nab\phi) dxds\\
    &-\int_0^t\int_{\R^3}(1+\rr) \d^n (u\cdot\nab\d_s\phi)\cdot \d^n(u\cdot\nab\phi) dxds+\int_0^t\int_{\R^3} (1+\rr)\d^n Err3\cdot \d^n(u\cdot\nab\phi) dxds\\
    =&I_{141}+I_{142}+I_{143}+I_{144}.
    \end{align*}
    By integration by parts, we may estimate $I_{141}$ by
    \begin{align*}
    I_{141}=&\int_0^t\int_{\R^3} -\d^n \d_i\phi\cdot \d^n(u\cdot\nab\d_i\phi+\d_i u\cdot\nab\phi)
    +(1+\rr)\sum_{|n_1|\geq 1,n_1+n_2=n}\d^{n_1}\Big(\frac{1}{1+\rr}\Big)\d^{n_2}\Delta\phi\cdot \d^n(u\cdot\nab\phi) dxds\\
    \lesssim  & \int_0^t\int_{\R^3}\frac{1}{2}\nab\cdot u|\d^n\nab\phi|^2  -\sum_{n_1+n_2=n,n_2<n}\d^n\d_i\phi\cdot(\d^{n_1}u\cdot\nab\d^{n_2}\d_i\phi)-\d^n\d_i\phi\cdot\d^n(\d_i u\cdot\nab\phi)dxds\\
    &+\int_0^t \lV \nab\rr\rV_{H^{N(0)-1}}\lV \nab u\rV_{H^{N(0)}}\lV \Phi\rV_{H^{N(0)}}^2 ds,
    \end{align*}
  which,  from (\ref{Main_Prop_Ass1}), (\ref{drr,v,w_inf}) and (\ref{Dec_Phi}), can be estimated as
    \begin{align*}
    I_{141}\lesssim & \int_0^t \lV\nab u\rV_{L^{\infty}}\lV \Phi\rV_{H^{N(0)}}^2+ \lV\d^n\d_i\phi\rV_2(\lV\<\nab\>^{|n|/2}\nab u\rV_{L^{\infty}}\lV\Phi\rV_{H^{N(0)}}+\lV\nab u\rV_{H^{N(0)}}\lV\<\nab\>^{|n|/2}\Phi\rV_{L^{\infty}})ds+\ep_1^4\\
    \lesssim &\int_0^t \ep_1^3\<s\>^{-1-l(0)+3\delta}+\ep_1^2\<t\>^{-1+2\delta}\lV \nab u\rV_{H^{N(0)}}ds+ \ep_1^4\lesssim \ep_1^3.
    \end{align*}

    Similarly, using (\ref{Main_Prop_Ass1}), (\ref{drr,v,w_inf}) and (\ref{Dec_Phi}) we obtain
    \begin{align*}
    I_{142}=&-\frac{1}{2} \int_{\R^3}(1+\rr)|\d^n(u\cdot\nab\phi)|^2dx\Big|_0^t+\int_0^t\int_{\R^3}\d_s\rr|\d^n(u\cdot\nab\phi)|^2 dxds\\
    \lesssim & \ep_1^4+\int_0^t \lV\d_s\rr\rV_{L^2}\lV \nab u\rV_{H^{N(0)}}^2\lV\Phi\rV_{H^{N(0)}}^2ds\lesssim \ep_1^4.
    \end{align*}
    and
    \begin{align*}
    I_{144}\lesssim &\int_0^t\int_{\R^3}\frac{1}{2}((1+\rr)\nab\cdot u+u\cdot\nab\rr)|\d^n(u\cdot\nab\phi)|^2dxds\\
    &-\int_0^t\int_{\R^3}(1+\rr) \sum_{n_2<n}\d^{n_1}u\cdot\nab\d^{n_2}(u\cdot\nab\phi) \cdot\d^n(u\cdot\nab\phi)dxds\\
    &+\int_0^t \lV\Phi\rV_{H^{N(0)}}^3(1+\lV u\rV_{H^{N(0)}}^2)(\lV\<\nab\>^{|n|/2}\nab u\rV_{L^{\infty}}\lV\Phi\rV_{H^{N(0)}}+\lV\nab u\rV_{H^{N(0)}}\lV\<\nab\>^{|n|/2}\Phi\rV_{L^{\infty}})ds\\
    \lesssim & \int_0^t  (\lV  u\rV_{H^{N(0)}}+\lV \rr\rV_{H^{N(0)}})\lV \nab u\rV_{H^{N(0)}}^2\lV\Phi\rV_{H^{N(0)}}^2ds+\int_0^t \ep_1^3(\ep_1^2\<s\>^{-1-l(0)+3\delta}+\ep_1\<s\>^{-1+2\delta}\lV\nab u\rV_{H^{N(0)}})ds\\
    \lesssim & \ep_1^5.
    \end{align*}
    Finally, let $e\in\{(1,0,0),(0,1,0),(0,0,1)\}$, by integration by parts, we rewrite $I_{143}$ as
    \begin{align*}
    I_{143}=& -\int_0^t\int_{\R^3}(1+\rr)[ \d^n(u\cdot\nab \d_s\phi)\cdot (u\cdot\nab\d^n\phi)+\sum_{|n_2|<|n|}\d^n(u\cdot\nab \d_s\phi)\cdot (\d^{n_1}u\cdot\nab\d^{n_2}\phi)]dxds\\
    =&-\int_0^t\int_{\R^3} \Big[\frac{1}{2}\d_s((1+\rr)|u\cdot\nab\d^n\phi|^2)-\frac{1}{2}\d_s\rr|u\cdot\nab\d^n\phi|^2\\
    &-(1+\rr)(\d_s u\cdot\nab\d^n\phi) \cdot (u\cdot\nab\d^n\phi)+(1+\rr)\sum_{n_1+n_2=n,|n_2|<|n|}(\d^{n_1}u\cdot\nab\d^{n_2}\d_s\phi)\cdot(u\cdot\nab\d^n\phi)\\
    &-(1+\rr)\sum_{n_1+n_2=n,|n_2|<|n|}\d^{n-e}(u\cdot\nab\d_s\phi)\cdot\d^e (\d^{n_1}u\cdot\nab\d^{n_2}\phi)\\
    &-\d^e \rr\sum_{n_1+n_2=n,|n_2|<|n|}\d^{n-e}(u\cdot\nab\d_s\phi)\cdot (\d^{n_1}u\cdot\nab\d^{n_2}\phi)\Big]dxds,
    \end{align*}
    By (\ref{Main_Prop_Ass1}), (\ref{dtrr-L2}) and (\ref{dtu-L2}), we have
    \begin{align*}
    I_{143}\lesssim & \sup_{s\in[0,t]}\lV u(s)\rV_{H^{N(0)}}^2\lV \Phi(s)\rV_{H^{N(0)}}^2+\int_0^t (\lV\d_s\rr\rV_{L^2}+\lV\d\rr\rV_{L^2})\lV\nab u\rV_{H^{N(0)}}^2\lV\Phi\rV_{H^{N(0)}} ds\\
    &+\int_0^t\big(\lV\d_s u\rV_{L^{\infty}}\lV \nab u\rV_{H^{N(0)}}\lV\Phi\rV_{H^{N(0)}}^2+\lV \nab u\rV_{H^{N(0)}}^2\lV \Phi\rV_{H^{N(0)}}^2\big)ds\\
    \lesssim &\ep_1^4.
    \end{align*}
    Hence, the desired bound (\ref{E_I1}) follows.

    {\bf Step 2. We prove the bound}
    \begin{equation} \label{E0_I3}
    I_2:=\int_0^t \int_{\R^3} (1+\rr)\d^n\d_s \phi\cdot \d^n (u\cdot\nab\d_s\phi)dxds\lesssim \ep_1^3,\ \ {\rm for }\ |n|=N(0).
    \end{equation}

    Since $I_2$ can be rewritten as
    \begin{equation*}
    I_2=\int_0^t \int_{\R^3} -\frac{1}{2}\big(u\cdot\nab\rr+(1+\rr)\nab\cdot u\big)(\d^n\d_s\phi)^2+(1+\rr)\sum_{n_1+n_2=n,|n_2|<|n|}\d^n\d_s\phi\cdot(\d^{n_1}u\cdot\nab\d^{n_2}\d_s\phi)dxds,
    \end{equation*}
    it follow from (\ref{drr,v,w_inf}) and (\ref{Dec_Phi}) that
    \begin{align*}
    I_2\lesssim  &\int_0^t \lV \nab(\rr,u)\rV_{L^{\infty}}\lV\Phi\rV_{H^{N(0)}}^2+\lV\d^n\Phi\rV_{L^2} (\lV\<\nab\>^{|n|/2}\nab u\rV_{L^{\infty}}\lV\Phi\rV_{H^{N(0)}}+\lV\nab u\rV_{H^{N(0)}}\lV\<\nab\>^{|n|/2}\Phi\rV_{L^{\infty}})ds\\
    \lesssim & \int_0^t \ep_1^3\<s\>^{-1-l(0)+3\delta}+\ep_1^2\<s\>^{-1+2\delta}\lV \nab u\rV_{H^{N(0)}}ds
    \lesssim \ep_1^3.
    \end{align*}

    {\bf Step 3. We prove the bound}
    \begin{equation}     \label{E0_I4}
    I_3:=\int_0^t\int_{\R^3}(1+\rr)\d^n\d_s\phi\cdot\d^n(u\cdot\nab(u\cdot\nab\phi))dxds\lesssim \ep_1^4,\ \ {\rm for}\ |n|=N(0).
    \end{equation}
    By integration by parts, we rewrite this by
    \begin{align*}
    I_3=&\int_0^t\int_{\R^3}(1+\rr) \d^n\d_s\phi\cdot(u\cdot\nab(u\cdot\nab\d^n\phi))dxds\\
    &+\int_0^t\int_{\R^3} \sum_{n_1+n_2+n_3=n,|n_3|<|n|}(1+\rr)\d^n\d_s\phi\cdot(\d^{n_1}u\cdot\nab(\d^{n_2}u\cdot\nab\d^{n_3}\phi))dxds\\
    =:&I_{31}+I_{32}.
    \end{align*}
    The second term $I_{32}$ can be estimated directly using H\"{o}lder and (\ref{Main_Prop_Ass1}). By integration by parts in times, the first term $I_{31}$ can be further rewritten as
    \begin{align*}
    I_{31}=&-\int_0^t\int_{\R^3} (1+\rr)(u\cdot\nab\d^n\d_s\phi) \cdot (u\cdot\nab\d^n\phi)+\big(u\cdot\nab\rr +(1+\rr)\nab\cdot u\big)\d^n\d_s\phi \cdot (u\cdot\nab\d^n\phi) dxds\\
    =& -\int_0^t \int_{\R^3} \d_s\big( \frac{1}{2}(1+\rr) (u\cdot\nab\d^n\phi)^2\big)-\frac{1}{2}\d_s\rr|u\cdot\nab\d^n\phi|^2 dx -(1+\rr)(\d_s u\cdot\nab\d^n\phi)\cdot (u\cdot\nab\d^n\phi)\\
    &+\big(u\cdot\nab\rr +(1+\rr)\nab\cdot u\big)\d^n\d_s\phi \cdot (u\cdot\nab\d^n\phi)dxds.
    \end{align*}
    From (\ref{dtrr-L2}), (\ref{dtu-L2}) and (\ref{drr,v,w_inf}) we may bound this by
    \begin{align*}
    I_{31}\lesssim &\sup_{s\in[0,t]} \lV u(s)\rV_{H^{N(0)}}^2\lV\Phi(s)\rV_{H^{N(0)}}^2+\int_0^t \lV\d_s\rr\rV_{L^2}\lV\nab u\rV_{H^{N(0)}}^2\lV\Phi\rV_{H^{N(0)}}^2+ \lV\d_s u\rV_{\infty}\lV\nab u\rV_{H^{N(0)}}\lV\Phi\rV_{H^{N(0)}}^2ds\\
    &+ \int_0^t \lV\nab(\rr,u)\rV_{L^{\infty}}\lV u\rV_{L^{\infty}}\lV \Phi\rV_{H^{N(0)}}^2 ds\\
    \lesssim &\ep_1^4+\int_0^t \ep_1^3 \lV\nab u\rV_{H^{N(0)}}^2+\ep_1^3 \<s\>^{-3/4}\lV\nab u\rV_{H^{N(0)}}+\ep_1^2\lV\nab(\rr,u)\rV_{H^{N(0)-1}}^2 ds\lesssim \ep_1^4 .
    \end{align*}
    Thus the bound (\ref{E0_I4}) is obtained.

    {\bf Step 4. We prove the bound}
    \begin{equation}      \label{E0_I5}
    I_4:=\int_0^t\int_{\R^3}(1+\rr)\d^n\d_s\phi \cdot \d^n Err3(s)dxds\lesssim \ep_1^4,\ \ {\rm for}\ |n|=N(0).
    \end{equation}
    In fact, this is a consequence of
    \begin{align*}
    \lV\d^n Err3(s)\rV_{L^2}\lesssim \lV\Phi\rV_{H^{N(0)}} \lV\Phi\rV_{L^{\infty}}\lV |\nab|^{-1}\Phi\rV_{L^{\infty}}\lesssim \ep_1^3\<s\>^{-5/4},
    \end{align*}
    which is given by  (\ref{Dec_Phi/nab}).

    Hence, from (\ref{E1-two})-(\ref{E0_I5}) and the assumption (\ref{MainAss_ini}), we have
    \begin{equation*}
    E_1(t)=E_1(0)+\int_0^t \d_s E_1(s)ds \lesssim \ep_0^2+\ep_1^3\lesssim \ep_0^2.
    \end{equation*}
    Combining this and (\ref{Ephi_n=0}), we obtain
    \begin{equation*}
    \frac{1}{2}\sum_{|n|\in\{0,N(0)\}}\int_{\R^3} (1+\rr)|\d^n\d_t\phi|^2+|\d^n\nab\phi|^2 dx\lesssim \ep_0^2,
    \end{equation*}
    which completes the proof of the Proposition.

\end{proof}

\section{Energy estimates, II: Weighted norms}
In this section we prove the energy bounds (\ref{Main_Prop_result1}) for $1\leq |a|\leq N_1$.
\subsection{The bound on $Z^a\rr$ and $Z^au$ for $1\leq|a|\leq N_1$.} We start with the weighted Sobolev bounds in (\ref{Main_Prop_result1}).

\begin{proposition}            \label{Prop_Ene_VecFie}
	With the notation and hypothesis in Proposition $\ref{Main_Prop}$, for any $t\in[0,T]$, we have
	\begin{equation}         \label{E-rr&u_VecFie}
	\lV \rr^{(a)}\rV_{H^{N(a)}}^2+\lV u^{(a)}\rV_{H^{N(a)}}^2+\int_0^t\big( \lV\nab \rr^{(a)}\rV_{H^{N(a)-1}}^2 + \lV\nab u^{(a)}\rV_{H^{N(a)}}^2+ \lV\div u^{(a)}\rV_{H^{N(a)}}^2\big) ds\lesssim \epsilon_0^2.
	\end{equation}
\end{proposition}
\begin{proof}
	We prove the bound (\ref{E-rr&u_VecFie}) by induction. From (\ref{E-rr&u}), we may assume that (\ref{E-rr&u_VecFie}) holds for all $|a|\leq l-1\leq N_1-1$, it then suffices to prove (\ref{E-rr&u_VecFie}) for $|a|=l$.
	
	{\bf Step 1. Higher order estimates for $\rr$ and $u$.}
	
	We define the energy functional
	\begin{align*}
	\mathcal{E}_0(t)=&\sum\limits_{|n|\leq N(a)}\big[\frac{1}{2}\int_{\R^3}\frac{P'(1+\rr)}{1+\rr}|\d^{n} \rr^{(a)}|^2 +(1+\rr)|\d^n u^{(a)}|^2dx \\
	&+\int_0^t \int_{\R^3} \frac{\mu_1}{2} |\d^{n}\nab u^{(a)}|^2+(\frac{\mu_1}{2}+\mu_2)|\d^n \div u^{(a)}|^2 dx ds\big].
	\end{align*}
	
	By (\ref{Main_Sys_VecFie}), we have
	\begin{equation}       \label{dt(E0)}
	\begin{aligned}
	\frac{d}{dt}\mathcal{E}_0(t)=&\sum_{|n|\leq N(a)}\int_{\R^3}\frac{1}{2}\d_t(\frac{P'(1+\rr)}{1+\rr})|\d^n \rr^{(a)}|^2dx\\
	&-\sum_{|n|\leq N(a)}\int_{\R^3}\frac{P'(1+\rr)}{1+\rr}\d^n\rr^{(a)}\cdot\d^n\sum_{b+c=a}C_a^b(u^{(b)}\cdot\nab\rr^{(c)}+\rr^{(b)}\nab\cdot u^{(c)})dx\\
	&-\sum_{|n|\leq N(a)}\int_{\R^3} \frac{P'(1+\rr)}{1+\rr} \d^n\rr^{(a)} \nab\cdot\d^n u^{(a)}+(1+\rr)\d^n u^{(a)}\cdot\d^n\big(\frac{P'(1+\rr)}{1+\rr}\nab\rr^{(a)}\big) dx \\
	&+\sum_{|n|\leq N(a)} \int_{\R^3} \frac{1}{2}\d_t \rr|\d^n u^{(a)}|^2+(1+\rr)\d^n u^{(a)}\cdot \d^n F_2^{(a)} dx\\
	&+\sum_{|n|\leq N(a)}\int_{\R^3} (1+\rr)\d^n u^{(a)}\cdot\sum_{1\leq l\leq a_1}C_{a_1}^l(-1)^l\d^{n}\big(\frac{\mu_1}{2}\Delta S^{a_1-l}\Gamma^{a'} u\\
	&\ \ \ \ \ \ \ +(\frac{\mu_1}{2}+\mu_2)\nab\div S^{a_1-l}\Gamma^{a'}u\big)+\rr\d^n u^{(a)}\cdot\d^{n}\big(\frac{\mu_1}{2}\Delta  u^{(a)} +(\frac{\mu_1}{2}+\mu_2)\nab\div u^{(a)}\big)dx\\
	&=\mathcal{R}_1+\mathcal{R}_2+\mathcal{R}_3+\mathcal{R}_4+\mathcal{R}_5,
	\end{aligned}
	\end{equation}
We will estimate the five terms, respectively. First, we prove the bound
	\begin{equation} \label{R1,R2}
	\int_0^t \mathcal{R}_1(s)+\mathcal{R}_2(s)ds\lesssim \ep_1^3.
	\end{equation}
In fact, by H\"{o}lder, we may bound left hand side of (\ref{R1,R2}) by
	\begin{align*}
	{\rm LHS}(\ref{R1,R2})\lesssim &\int_0^t\Big[ \lV\d_s\rr\rV_{L^{\infty}}\lV(1+\rr)^{\gamma-3}\rV_{L^{\infty}}\lV\rr^{(a)}\rV_{H^{N(a)}}^2-\sum_{|n|\leq N(a)}\int_{\R^3}\frac{P'(1+\rr)}{1+\rr}\d^n\rr^{(a)}\cdot(u\cdot\nab\d^n\rr^{(a)})dx\\
	&+\sum_{b+c=a}\lV\rr^{(a)}\rV_{H^{N(a)}}\lV\nab u^{(b)}\rV_{H^{N(a)}}\lV\nab\rr^{(c)}\rV_{H^{N(c)-1}}\Big]ds,
	\end{align*}
	which by (\ref{Main_Prop_Ass1}), (\ref{dtrr-inf}), and (\ref{drr,v,w_inf}) is bounded by
	\begin{align*}
	{\rm LHS}(\ref{R1,R2})\lesssim & \int_0^t \ep_1^3\<s\>^{-1-l(0)/4}+\sum_{|n|\leq N(a)}\frac{1}{2}\int_{\R^3}(u\cdot\nab\frac{P'(1+\rr)}{1+\rr}+\frac{P'(1+\rr)}{1+\rr}\nab\cdot u)|\d^n\rr^{(a)}|^2dxds
	+\ep_1^3\\
	\lesssim & \ep_1^3+\int_0^t (\lV u\rV_{H^2}\lV\nab\rr\rV_{L^{\infty}}+\lV \nab u\rV_{L^{\infty}})\lV\rr^{(a)}\rV_{H^{N(a)}}^2 ds\\
	\lesssim &\ep_1^3+\int_0^t \ep_1^3\<s\>^{-1-l(0)+3\delta}ds\lesssim \ep_1^3.
	\end{align*}
	Next, we prove the bound
	\begin{equation}      \label{R3}
	\int_0^t \mathcal{R}_3(s)ds\lesssim \ep_1^3.
	\end{equation}
	Integration by parts yields
	\begin{align*}
	\mathcal{R}_3=&-\sum_{|n|\leq N(a)}\int_{\R^3} -\d^n\rr^{(a)}\d^n u^{(a)}\cdot\nab\Big(  \frac{P'(1+\rr)}{1+\rr}\Big)+\sum_{n_1+n_2=n,|n_1|\geq 1}\d^n u^{(a)}\cdot\d^{n_1}\Big(\frac{P'(1+\rr)}{1+\rr}\Big)\nab\d^{n_2}\rr^{(a)}\\
	&\ \ \ \ \ \ \ \ \ -\d^{e}\rr\ \d^n u^{(a)}\cdot\d^{n-e}\big(\frac{P'(1+\rr)}{1+\rr}\nab\rr^{(a)}\big)-\rr\ \d^{n+e} u^{(a)}\cdot\d^{n-e}\big(\frac{P'(1+\rr)}{1+\rr}\nab\rr^{(a)}\big) dx,
	\end{align*}
	where $e\in\{(1,0,0),(0,1,0),(0,0,1)\}$.
	By H\"{o}lder and (\ref{Main_Prop_Ass1}), we may bound this by
	\begin{align*}
	\int_0^t \mathcal{R}_3(s)ds\lesssim &\int_0^t \lV\nab\rr^{(a)}\rV_{H^{N(a)-1}}\lV u^{(a)}\rV_{H^{N(a)}}\lV\nab\rr\rV_{H^{N(a)}}+\lV \nab\rr\rV_{H^2}\lV u^{(a)}\rV_{H^{N(a)}}\lV\nab\rr^{(a)}\rV_{H^{N(a)-1}}\\
	&+\lV\rr\rV_{H^3}\lV\nab u^{(a)}\rV_{H^{N(a)}}\lV\nab\rr^{(a)}\rV_{H^{N(a)-1}}ds\\
	\lesssim & \ep_1^3.
	\end{align*}
	
	Nest we prove the bound
	\begin{equation}        \label{R4}
	\int_0^t \mathcal{R}_4(s)ds\lesssim \ep_1^3.
	\end{equation}
	
	By (\ref{dtrr-inf}), it is easy to estimate the first term in the integrand of $\mathcal{R}_4$ by
	\begin{equation*}
	\sum_{|n|\leq N(a)} \int_0^t\int_{\R^3} \frac{1}{2}\d_s \rr|\d^n u^{(a)}|^2 dxds\lesssim \int_0^t \lV\d_s\rr\rV_{L^{\infty}}\lV u^{(a)}\rV_{H^{N(a)}}^2 ds\lesssim \int_0^t \ep_1^3\<s\>^{-1-l(0)/4}ds\lesssim \ep_1^3.
	\end{equation*}
	Then we consider the second term in the integrand of $\mathcal{R}_4$. From the definition of $F_2^{(a)}$ (\ref{F2a}), we have
	\begin{align*}
	\sum_{|n|\leq N(a)} \int_{\R^3}(1+\rr)\d^n u^{(a)}\cdot \d^n F_2^{(a)} dx=&\sum_{|n|\leq N(a)} \int_{\R^3}(1+\rr)\d^n u^{(a)}\cdot \d^n F_{2,I}^{(a)} dx\\
	&+\sum_{|n|\leq N(a)} \int_{\R^3}(1+\rr)\d^n u^{(a)}\cdot \d^n F_{2,II}^{(a)} dx\\
	&+\sum_{|n|\leq N(a)} \int_{\R^3}(1+\rr)\d^n u^{(a)}\cdot \d^n(S+1)^{a_1}\Gamma^{a'}Err2dx\\
	=:&\mathcal{R}_{41}+\mathcal{R}_{42}+\mathcal{R}_{43},
	\end{align*}
	where
	\begin{align*}
	F_{2,I}^{(a)}:=&-\sum_{b+c=a}C^b_a u^{(b)}\cdot\nab u^{(c)}-\sum_{b+c=a,|b|\geq 1}C^b_a Z^b(\frac{P'(1+\rr)}{1+\rr})\nab\rr^{(c)}\\
	&-\sum_{b+c=a}C^b_a Z^b \big(\frac{\rr}{1+\rr}\big)\big( \frac{\mu_1}{2}\Delta (S-1)^{(c_1)}\Gamma^{(c')} u+(\frac{\mu_1}{2}+\mu_2)\nab\div(S-1)^{c_1}\Gamma^{c'}u \big), \\
	F_{2,II}^{(a)}:=&\sum_{b+c+e=a}C^{b,c}_aZ^b \Big(\frac{1}{1+\rr}\Big)\big[\frac{1}{2} \nab\big( \nab(S-1)^{c_1}\Gamma^{c'}\phi\cdot\nab(S-1)^{e_1}\Gamma^{e'}\phi \big)\\
	&-\d_j\big(  \nab(S-1)^{c_1}\Gamma^{c'}\phi\cdot\d_j(S-1)^{e_1}\Gamma^{e'}\phi \big) \big].
	\end{align*}
	
	On one hand, we consider the contribution of $\mathcal{R}_{41}$. By integration by parts, H\"{o}lder, and Sobolev embedding, we may bound $\mathcal{R}_{41}$ by
	\begin{align*}
	\mathcal{R}_{41}=&\int_{\R^3} (1+\rr)u^{(a)} F^{(a)}_{2,I}-\sum_{1\leq|n|\leq N(a)}\big( \d^e\rr\d^n u^{(a)}+(1+\rr)\d^{n+e}u^{(a)} \big) \d^{n-e}F_{2,I}^{(a)} dx\\
	\lesssim & \Big[\lV(1+\rr)u^{(a)}\rV_{L^2}+\sum_{1\leq|n|\leq N(a)}\lV\big( \d^e\rr\d^n u^{(a)}+(1+\rr)\d^{n+e}u^{(a)} \big)\rV_{L^2}\Big] \lV F^{(a)}_{2,I}\rV_{H^{N(a)-1}}\\
	\lesssim & (1+\lV\rr\rV_{H^3})(\lV u^{(a)}\rV_{H^2}+\lV\nab u^{(a)}\rV_{H^{N(a)}})\Big( \sum_{b+c=a}\lV \nab u^{(b)}\rV_{H^{N(a)-1}}\lV\nab u^{(c)}\rV_{H^{N(a)}}\\
	& +\sum_{b+c=a}\sum_{|\tilde{b}|\leq |b|}\lV\nab \rr^{(\tilde{b})}\rV_{H^{N(a)-1}}(1+\lV\nab\rr^{(\tilde{b})}\rV_{H^{N(a)-1}})^{|b|}(\lV\nab\rr^{(c)}\rV_{H^{N(a)-1}}+\lV\nab u^{(c)}\rV_{H^{N(a)}})\Big),
	\end{align*}
	which is acceptable by (\ref{Main_Prop_Ass1}).
	
	On the other hand, we consider the contribution of $\mathcal{R}_{42}$. By the symmetry, we may assume that $|c|\leq |e|$. It suffices to consider the term
	\begin{equation}   \label{R42'}
	\sum_{|n|\leq N(a),|c+e|\leq |a|}\int_{\R^3}(1+\rr)\d^n u^{(a)}\d^n \nab(\nab\phi^{(c)}\nab\phi^{(e)})dx,
	\end{equation}
	By integration by parts, we may estimate this by
	\begin{align*}
	(\ref{R42'})=\sum_{|n|\leq N(a),|c+e|\leq |a|}\int_{\R^3}  [\nab\rr\d^n u^{(a)}+(1+\rr)\d^n\nab u^{(a)}]\d^n (\nab\phi^{(c)}\nab\phi^{(e)})dx,
	\end{align*}
	we then use H\"{o}lder and Sobolev embedding to bound previous by
	\begin{align*}
	(\ref{R42'})\lesssim & \sum_{|n|\leq N(a),|c+e|\leq |a|} (1+\lV\rr\rV_{H^3})\lV\nab u^{(a)}\rV_{H^{N(a)}}\sum_{k}2^{N(a)k^+}\lV P_{k}\nab\phi^{(c)}\rV_{L^{\infty}}\lV\nab\phi^{(e)}\rV_{H^{N(a)}}\\
	\lesssim & \ep_1^2 \<s\>^{-1+4\delta}\lV\nab u^{(a)}\rV_{H^{N(a)}},
	\end{align*}
	which is acceptable by (\ref{Main_Prop_Ass1}). A similar argument could be made for $\mathcal{R}_{43}$.
	
	Thus, this concludes the bound (\ref{R4}).
	
	Finally, we prove the bound
	\begin{equation}        \label{R5}
	\int_0^t \mathcal{R}_5(s)ds\leq \frac{1}{2}\mathcal{E}_0(t)+C\ep_0^2.
	\end{equation}
	Let $e\in\{(1,0,0),(0,1,0),(0,0,1)\}$, by integration by parts, $\mathcal{R}_5$ can be rewritten as
	\begin{align*}
	\mathcal{R}_5=&-\sum_{|n|\leq N(a)}\int_{\R^3}\sum_{1\leq l\leq a_1}C_{a_1}^l(-1)^l (1+\rr)\big(\frac{\mu_1}{2}\d^n\nab u^{(a)}\d^n\nab S^{a_1-l}\Gamma^{a'}u\\
	&\ \ \ \ \ \ \ \ \ \ \ \ \ \ \  +(\frac{\mu_1}{2}+\mu_2)\d^n\div u^{(a)}\d^n\div S^{a_1-l}\Gamma^{a'}u\big)dx\\
	&-\sum_{|n|\leq N(a)}\int_{\R^3}\sum_{1\leq l\leq a_1}C_{a_1}^l(-1)^l \d^n u^{(a)}\big(\frac{\mu_1}{2}\nab\rr\cdot\d^n\nab S^{a_1-l}\Gamma^{a'}u\\
	&\ \ \ \ \ \ \ \ \ \ \ \ \ \ \ \ \ \ +(\frac{\mu_1}{2}+\mu_2)\div \rr\cdot\d^n\div S^{a_1-l}\Gamma^{a'}u\big)dx\\
	&+\int_{\R^3} \rr u^{(a)}\big(\frac{\mu_1}{2}\Delta  u^{(a)} +(\frac{\mu_1}{2}+\mu_2)\nab\div u^{(a)}\big)dx\\
	&-\sum_{|n|\leq N(a)}\int_{\R^3}(\d^{e}\rr\d^n u^{(a)}+\rr\d^{n+e} u^{(a)})\cdot\d^{n-e}\big(\frac{\mu_1}{2}\Delta  u^{(a)} +(\frac{\mu_1}{2}+\mu_2)\nab\div u^{(a)}\big)dx.
	\end{align*}
	We then use H\"{o}lder, inductive assumption, and (\ref{Main_Prop_Ass1}) to bound this by
	\begin{align*}
	\int_0^t \mathcal{R}_5(s)ds\leq &\frac{1}{2}\sum_{|n|\leq N(a)}\int_0^t \frac{\mu_1}{2}\lV\d^n\nab u^{(a)}\rV_{L^2}^2+(\frac{\mu_1}{2}+\mu_2)\lV\d^n\div u^{(a)}\rV_{L^2}^2 ds
	+C\sum_{|b|<|a|}\int_0^t \lV\nab u^{(b)}\rV_{H^{N(a)}}^2 ds\\
	&+C\int_0^t \sum_{|b|\leq |a|}\lV u^{(a)}\rV_{H^{N(a)}}\lV\nab\rr\rV_{H^2}\lV\nab u^{(b)}\rV_{H^{N(a)}}+  \lV\rr\rV_{H^3}\lV\nab u^{(a)}\rV_{H^{N(a)}}^2 ds\\
	\leq & \frac{1}{2}\sum_{|n|\leq N(a)}\int_0^t \frac{\mu_1}{2}\lV\d^n\nab u^{(a)}\rV_{L^2}^2+(\frac{\mu_1}{2}+\mu_2)\lV\d^n\div u^{(a)}\rV_{L^2}^2 ds+C\ep_0^2+C\ep_1^3.
	\end{align*}
	This implies the bound (\ref{R5}) by definition of $\mathcal{E}_0$.
	
	Hence, from (\ref{R1,R2})-(\ref{R5}) and (\ref{MainAss_ini}), one obtains
	\begin{align*}
	\mathcal{E}_0(t)=\mathcal{E}_0(0)+\int_0^t \frac{d}{ds}\mathcal{E}_0(s)ds\leq \frac{1}{2}\mathcal{E}_0(t)+C\ep_0^2,
	\end{align*}
	which, combining with (\ref{1-(1+rr)gamma}), implies
	\begin{equation}          \label{E-rra&ua-VF}
	\frac{a\gamma}{2}\lV \rr^{(a)}\rV_{H^{N(a)}}^2 +\lV u^{(a)}\rV_{H^{N(a)}}^2+\int_0^t \frac{\mu_1}{2} \lV\nab u^{(a)}\rV_{H^{N(a)}}^2+(\frac{\mu_1}{2}+\mu_2)\lV \div u^{(a)}\rV_{H^{N(a)}}^2 ds\lesssim \ep_0^2.
	\end{equation}
	
	{\bf Step 2. The estimate for dissipation of density $\rr^{(a)}$.}
	
	Taking $\d^n$ on the $Z^a u$-equation of (\ref{Main_Sys_VecFie}) and multiplying $\d^n\nab Z^a\rr$, from integration by parts we have
	\begin{equation}            \label{Diss_rra}
	\begin{aligned}
	&\sum_{|n|\leq N(a)-1}\int_{\R^3}\d^n\nab\rr^{(a)}\cdot\d^n\d_t u^{(a)} dx+\sum_{|n|\leq N(a)-1}\int_{\R^3}\d^n\nab\rr^{(a)}\cdot\d^n(\frac{P'(1+\rr)}{1+\rr}\nab\rr^{(a)})dx\\
	=& \sum_{|n|\leq N(a)-1}\int_{\R^3}\d^n\nab\rr^{(a)}\cdot\d^n\big(\frac{\mu_1}{2}\Delta(S-1)^{a_1}\Gamma^{a'} u+(\frac{\mu_1}{2}+\mu_2)\nab\div (S-1)^{a_1}\Gamma^{a'} u+F^{(a)}_2\big)dx.
	\end{aligned}
	\end{equation}
	By $\rr$-equation in (\ref{Main_Sys_VecFie}), the first term in the left-hand side of (\ref{Diss_rra}) can be rewritten as
	\begin{equation}          \label{Diss_rra_L1}
	\begin{aligned}
	&\sum_{|n|\leq N(a)-1}\int_{\R^3}\d^n\nab\rr^{(a)}\cdot\d^n\d_t u^{(a)}dx\\
	=& \sum_{|n|\leq N(a)-1}\big(\frac{d}{dt}\int_{\R^3}\d^n\nab\rr^{(a)}\cdot\d^n u^{(a)}dx-\int_{\R^3}\d^n\d_t\nab\rr^{(a)}\cdot\d^n u^{(a)}dx\big)\\
	=&\sum_{|n|\leq N(a)-1}\big(\frac{d}{dt}\int_{\R^3}\d^n\nab\rr^{(a)}\cdot\d^n u^{(a)}dx\\
	&\ \ \ -\int_{\R^3}\d^n[\nab\cdot u^{(a)}+\sum_{b+c=a}C_a^b(u^{(b)}\cdot\nab\rr^{(c)}+\rr^{(b)}\nab\cdot u^{(c)})]\d^n\nab\cdot u^{(a)}dx\big)\\
	=&\sum_{|n|\leq N(a)-1}\big(\frac{d}{dt}\int_{\R^3}\d^n\nab\rr^{(a)}\cdot\d^n u^{(a)}dx-\lV\d^n\nab\cdot u^{(a)}\rV_{L^2}^2\\
	&\ \ \ -\sum_{b+c=a}C_a^b\int_{\R^3} \d^n(u^{(b)}\cdot\nab\rr^{(c)}+\rr^{(b)}\nab\cdot u^{(c)})\d^n\nab\cdot u^{(a)} dx\big).
	\end{aligned}
	\end{equation}
	We use integration by parts to rewrite the second term in the left-hand side of (\ref{Diss_rra}) as
	\begin{equation}              \label{Diss_rra_L2}
	\begin{aligned}
	&\sum_{|n|\leq N(a)-1}\int_{\R^3}\d^n\nab\rr^{(a)}\cdot\d^n(\frac{P'(1+\rr)}{1+\rr}\nab\rr^{(a)})dx\\
	=&\sum_{|n|\leq N(a)-1}\big(\int_{\R^3}\frac{P'(1+\rr)}{1+\rr}|\d^n\nab\rr^{(a)}|^2dx+\int_{\R^3} \d^n\nab\rr^{(a)} \cdot\sum_{|n_1|\geq 1,n_1+n_2=n}\d^{n_1}(\frac{P'(1+\rr)}{1+\rr})\d^{n_2}\nab\rr^{(a)} dx\big).
	\end{aligned}
	\end{equation}
	Then from (\ref{Diss_rra})-(\ref{Diss_rra_L2}), H\"{o}lder, and (\ref{1-(1+rr)gamma}), we obtain
	\begin{equation}           \label{Diss_rra2}
	\begin{aligned}
	&\sum_{|n|\leq N(a)-1}(\frac{d}{dt}\int_{\R^3}\d^n\nab\rr^{(a)}\cdot\d^n u^{(a)}dx-\lV\d^n\nab u^{(a)}\rV_{L^2}^2+\int_{\R^3}\frac{P'(1+\rr)}{1+\rr}|\d^n\nab\rr^{(a)}|^2dx)\\
	=&\sum_{|n|\leq N(a)-1}\int_{\R^3}\Big[\sum_{b+c=a}C_a^b \d^n(u^{(b)}\cdot\nab\rr^{(c)}+\rr^{(b)}\nab\cdot u^{(c)})\d^n\nab\cdot u^{(a)} \\
	&-\d^n\nab\rr^{(a)} \cdot\sum_{|n_1|\geq 1,n_1+n_2=n}\d^{n_1}(\frac{P'(1+\rr)}{1+\rr})\d^{n_2}\nab\rr^{(a)}\\
	&+\d^n\nab\rr^{(a)}\cdot\d^n\big(\frac{\mu_1}{2}\Delta (S-1)^{a_1}\Gamma^{a'} u+(\frac{\mu_1}{2}+\mu_2)\nab\div (S-1)^{a_1}\Gamma^{a'} u+F^{(a)}_2\big)\Big]dx\\
	\leq & C\sum_{b+c=a}\lV\nab u^{(a)}\rV_{H^{N(a)}}\lV\nab u^{(b)}\rV_{H^{N(a)}}\lV \rr^{(c)}\rV_{H^{N(a)}}+C\lV\nab\rr^{(a)}\rV_{H^{N(a)-1}}^2\lV\rr\rV_{H^{N(a)}}\\
	&+\frac{1}{4}\sum_{|n|\leq N(a)-1}\int_{\R^3}\frac{P'(1+\rr)}{1+\rr}|\d^n\nab\rr^{(a)}|^2dx+C(\mu_1,\mu_2)\sum_{|b|\leq |a|}\lV \nab u^{(b)}\rV_{H^{N(a)}}^2\\
	&+\lV\nab\rr^{(a)}\rV_{H^{N(a)-1}}\lV F_2^{(a)}\rV_{H^{N(a)-1}}.
	\end{aligned}
	\end{equation}
	Noticing that
	\begin{equation*}
	\d^n\nab\rr\cdot\d^n u=\frac{1}{2}(|\d^n(u+\nab\rr)|^2-|\d^n\nab\rr|^2-|\d^n u|^2),
	\end{equation*}
	by Sobolev embedding and (\ref{phi^2-Hna}),  the nonlinear terms $F_2^{(a)}$ can be bounded by
	\begin{equation}          \label{F2a-HN(a)-1}
	\begin{aligned}
	\lV F_2^{(a)}\rV_{H^{N(a)-1}}\leq & C\sum_{b+c=a}(\lV u^{(b)}\rV_{H^{N(a)}}+\lV \rr^{(b)}\rV_{H^{N(a)}})(\lV \nab u^{(c)}\rV_{H^{N(a)}}+\lV \nab\rr^{(c)}\rV_{H^{N(a)-1}})\\
	&+C\sum_{b+c=a}\lV\nab\phi^{(b)}\nab\phi^{(c)}\rV_{H^{N(a)}}\\
	\leq &C\sum_{b+c=a}(\lV u^{(b)}\rV_{H^{N(a)}}+\lV \rr^{(b)}\rV_{H^{N(a)}})(\lV \nab u^{(c)}\rV_{H^{N(a)}}+\lV \nab\rr^{(c)}\rV_{H^{N(a)-1}})+C\ep_1^2\<t\>^{-1+4\delta}.
	\end{aligned}
	\end{equation}	
	Therefore,  (\ref{Diss_rra2}) can be further bounded by
	\begin{align*}
	&\sum_{|n|\leq N(a)-1}\big(\frac{1}{2}\frac{d}{dt}(\lV\d^n(u^{(a)}+\nab\rr^{(a)})\rV_{L^2}^2-\lV\d^n\nab\rr^{(a)}\rV_{L^2}^2-\lV\d^n u^{(a)}\rV_{L^2}^2)\\
	&+\frac{3}{4}\int_{\R^3}\frac{P'(1+\rr)}{1+\rr}|\d^n\nab\rr^{(a)}|^2dx\big)-C(\mu_1,\mu_2)\sum_{|b|\leq |a|}\lV\nab u^{(b)}\rV_{H^{N(a)}}^2\\
	\lesssim & \sum_{b+c=a}\lV\nab u^{(a)}\rV_{H^{N(a)}}\lV\nab u^{(b)}\rV_{H^{N(a)}}\lV \rr^{(c)}\rV_{H^{N(a)}}+\lV\nab\rr^{(a)}\rV_{H^{N(a)-1}}^2\lV\rr\rV_{H^{N(a)}}\\
	&+\lV\nab\rr^{(a)}\rV_{H^{N(a)-1}}(\sum_{b+c=a}(\lV u^{(b)}\rV_{H^{N(a)}}+\lV \rr^{(b)}\rV_{H^{N(a)}})(\lV \nab u^{(c)}\rV_{H^{N(a)}}+\lV \nab\rr^{(c)}\rV_{H^{N(a)-1}})+\ep_1^2\<t\>^{-1+4\delta}).
	\end{align*}
	Then by integration in time and (\ref{Main_Prop_Ass1}) we have
	\begin{equation}              \label{Diss_rra_ineq}
	\begin{aligned}
	&\frac{1}{2}(\lV u^{(a)}+\nab\rr^{(a)}\rV_{H^{N(a)-1}}^2-\lV\nab\rr^{(a)}\rV_{H^{N(a)-1}}^2-\lV u^{(a)}\rV_{H^{N(a)-1}}^2)\\
	&+\frac{3}{4}\sum_{l=1}^{N(a)}\int_0^t\int_{\R^3} \frac{P'(1+\rr)}{1+\rr}|\nab^l\rr^{(a)}|^2dxds-C(\mu_1,\mu_2)\int_0^t \lV\nab u^{(a)}\rV_{H^{N(a)}}^2 ds\lesssim \ep_0^2.
	\end{aligned}
	\end{equation}
	
	{\bf Step 3. The energy estimate (\ref{E-rr&u_VecFie})}
	
	In order to prove the bound (\ref{E-rr&u_VecFie}), choosing
	\begin{equation*}
	c_1:=\min\{\frac{1}{4},\frac{a\ga}{4},\frac{\mu_1}{4C(\mu_1,\mu_2)}\}
	\end{equation*}
	multiplying (\ref{Diss_rra_ineq}) by $c_1$ and then adding one to the inequality (\ref{E-rra&ua-VF}), we have
	\begin{align*}
	&\frac{c_1}{2}\lV u^{(a)}+\nab\rr^{(a)}\rV_{H^{N(a)-1}}^2+\frac{a\gamma}{8}\lV\rr^{(a)}\rV_{H^{N(a)}}^2 +\frac{1}{8}\lV u^{(a)}\rV_{H^{N(a)}}^2 +\frac{3c_1}{4}\sum_{l=1}^{N(a)}\int_0^t\int_{\R^3} \frac{P'(1+\rr)}{1+\rr}|\nab^l\rr^{(a)}|^2dxds\\
	&+\int_0^t \frac{\mu_1}{4}\lV\nab u^{(a)}\rV_{H^{N(a)}}^2+(\frac{\mu_1}{2}+\mu_2)\lV\div u^{(a)}\rV_{H^{N(a)}}^2 ds\lesssim  \ep_0^2.
	\end{align*}
	This implies the bound (\ref{E-rr&u_VecFie}). We complete the proof of the lemma.

\end{proof}

\subsection{The bound on $Z^a\phi$ for $1\leq|a|\leq N_1$.}
\begin{proposition}            \label{Prop_Ene_Weiphi}
	With the notation and hypothesis in Proposition $\ref{Main_Prop}$, for any $t\in[0,T]$, we have
	\begin{equation*}
	\lV \d_t \phi^{(a)}\rV_{H^{N(a)}}^2+\lV \nab\phi^{(a)}\rV_{H^{N(a)}}^2\lesssim \epsilon_0^2\<t\>^{2H(a)\delta}.
	\end{equation*}
\end{proposition}
\begin{proof}
    By $\lV \rr\rV_{H^2}\lesssim \ep_1$ it suffices to show that
    \begin{equation*}
    \EE_a(t):=\frac{1}{2}\sum_{|n|\in\{0,N(a)\}}\int_{\R^3}(1+\rr)|\d^n\d_t\phi^{(a)}|^2|+|\d^n\nab\phi^{(a)}|^2 dx\lesssim \ep_0^2\<t\>^{2H(a)\delta}.
    \end{equation*}
    From $\phi^{(a)}$-equation in (\ref{Main_Sys_VecFie}) and integration by parts we have
    \begin{align*}
    \frac{d}{dt}\EE_a(t)=&\int_{\R^3}\sum_{|n|\in\{0,N(a)\}} \frac{1}{2}\d_t\rr|\d^n\d_t\phi^{(a)}|^2dx+\sum_{\substack{|n|=N(a),|n_1|\geq 1,\\n_1+n_2=n}}(1+\rr)\d^n\d_t\phi^{(a)}\d^{n_1}\big( \frac{1}{1+\rr} \big)\d^{n_2}\Delta\phi^{(a)}dx\\
    &-\int_{\R^3} \sum_{\substack{|n|\in\{0,N(a)\},\\b+c=a,|b|\geq 1}}C_a^b (1+\rr)\d^n\d_t\phi^{(a)}\d^n\big( Z^b(\frac{1}{1+\rr})\Delta\phi^{(c)} \big)dx\\
    &-2\int_{\R^3} \sum_{\substack{|n|\in\{0,N(a)\},\\b+c=a}}C_a^b (1+\rr)\d^n\d_t\phi^{(a)}\d^n\big( u^{(b)}\cdot\nab\d_t\phi^{(c)} \big)dx\\
    &-\int_{\R^3} \sum_{\substack{|n|\in\{0,N(a)\},\\b+c=a}}C_a^b (1+\rr)\d^n\d_t\phi^{(a)}\d^n\big( \d_tu^{(b)}\cdot\nab\phi^{(c)} \big)dx\\
    &-\int_{\R^3} \sum_{\substack{|n|\in\{0,N(a)\},\\b+c+e=a}}C_a^{b,c} (1+\rr)\d^n\d_t\phi^{(a)}\d^n\big( u^{(b)}\cdot\nab(u^{(c)}\cdot\nab\phi^{(e)}) \big)dx\\
    &+\int_{\R^3} \sum_{|n|\in\{0,N(a)\}} (1+\rr)\d^n\d_t\phi^{(a)}\d^n (S+2)^{a_1}\Gamma^{a'}Err3 dx\\
    =:&\bfR_1+\bfR_2+\bfR_3+\bfR_4+\bfR_5+\bfR_6.
    \end{align*}

    \textbf{Step 1. We prove the bound
    \begin{equation*}
    \int_0^t \bfR_1(s)ds\lesssim \ep_1^3.
    \end{equation*}}

    By (\ref{dtrr-inf}) we may bound the first term in $\bfR_1$ by
    \begin{align*}
    \int_{\R^3}\sum_{|n|\in\{0,N(a)\}} \frac{1}{2}\d_s\rr|\d^n\d_s\phi^{(a)}|^2dx\lesssim \lV\d_s\rr\rV_{L^{\infty}}\lV\d_s\phi^{(a)}\rV_{H^{N(a)}}^2\lesssim \ep_1^3\<s\>^{-1-l(0)/4+2H(a)\delta}.
    \end{align*}
    For the second term in $\bfR_1$, when $|a|\geq N_1-2$, by (\ref{drr,v,w_inf}) we have
    \begin{align*}
    &\int_{\R^3} \sum_{\substack{|n|=N(a),|n_1|\geq 1,\\n_1+n_2=n}}(1+\rr)\d^n\d_s\phi^{(a)}\d^{n_1}\big( \frac{1}{1+\rr} \big)\d^{n_2}\Delta\phi^{(a)}dx\\
    \lesssim & \lV\d_s\phi^{(a)}\rV_{H^{N(a)}}\lV\nab\rr\rV_{W^{N(a),\infty}}(1+\lV\rr\rV_{H^{N(a)}})^{|n_1|}\lV\nab\phi^{(a)}\rV_{H^{N(a)}}\lesssim \ep_1^3\<s\>^{-1-l(0)+2\delta+2H(a)\delta};
    \end{align*}
    when $|a|\leq N_1-3$, by (\ref{drr,v,w_inf}) and (\ref{Dec_Phi}) we have
    \begin{align*}
    &\int_{\R^3} \sum_{\substack{|n|=N(a),|n_1|\geq 1,\\n_1+n_2=n}}(1+\rr)\d^n\d_s\phi^{(a)}\d^{n_1}\big( \frac{1}{1+\rr} \big)\d^{n_2}\Delta\phi^{(a)}dx\\
    \lesssim & \lV\d_s\phi^{(a)}\rV_{H^{N(a)}}\big(\lV\nab\rr\rV_{W^{N(a)/2,\infty}}\lV\nab\phi^{(a)}\rV_{H^{N(a)}}+\lV\nab\rr\rV_{H^{N(a)-1}}\lV\nab\phi^{(a)}\rV_{W^{N(a)/2,\infty}}\big)(1+\lV\rr\rV_{H^{N(a)}})^{|n_1|}\\
    \lesssim &\ep_1^3\<s\>^{-1-l(0)+2\delta+2H(a)\delta}+\ep_1^2\<s\>^{H(a)\delta-1+\delta+H(a+2)\delta}\lV\nab\rr\rV_{H^{N(a)-1}}.
    \end{align*}
    Thus combining these bounds and (\ref{Main_Prop_Ass1}), we get
    \begin{equation*}
    \int_0^t \bfR_1(s)ds\lesssim \int_0^t \ep_1^3\<s\>^{-1-l(0)+3\delta+2H(a)\delta}+\ep_1^2\<s\>^{H(a)\delta-1+\delta+H(a+2)\delta}\lV\nab\rr\rV_{H^{N(a)-1}}ds\lesssim \ep_1^3.
    \end{equation*}

    \textbf{Step 2. We prove the bound
    	\begin{equation*}
    	\int_0^t \bfR_2(s)ds\lesssim \ep_1^3\<t\>^{2H(a)\delta}.
    	\end{equation*}}
    Noticing that for any $|b|\geq 1$
    \begin{equation*}
    Z^b(\frac{1}{1+\rr})=\sum_{b_1+\cdots+b_s=b}C_{b_1\cdots b_s}(1+\rr)^{-s-1}\rr^{(b_1)}\cdots\rr^{(b_s)}.
    \end{equation*}
    By H\"{o}lder and (\ref{Main_Prop_Ass1}), it suffices to prove
    \begin{equation}           \label{bfR2-rep}
    \int_0^t\int_{\R^3} \sum_{|n|\in\{0,N(a)\}} \d^n\d_s\phi^{(a)}\d^n\big( \rr^{(b)}\Delta\phi^{(c)} \big)dxds\lesssim \ep_1^3\<t\>^{2H(a)\delta},
    \end{equation}
    for $|b|+|c|\leq |a|,|b|\geq 1$.

    For the contribution when $|a|\geq N_1-1,|c|\leq 2$, by H\"{o}lder, (\ref{Main_Prop_Ass1}) and (\ref{Dec_Phi}) we may bound the left-hand side of (\ref{bfR2-rep}) by
    \begin{align*}
    {\rm LHS}(\ref{bfR2-rep})\lesssim & \int_0^t \lV \d_s\phi^{(a)}\rV_{H^{N(a)}}\lV\rr^{(b)}\rV_{H^{N(a)}}\sum_k 2^{N(a)k^++k}\lV P_k\nab\phi^{(c)}\rV_{L^{\infty}}ds\\
    \lesssim & \int_0^t \ep_1^3\<s\>^{H(a)\delta-1+\delta+H(c+2)\delta}ds\lesssim \ep_1^3\<t\>^{2H(a)\delta}.
    \end{align*}

    We now consider the contribution when $|a|\geq N_1-1, |b|\geq 1, |c|\geq 3$ or $|a|\leq N_1-2,|b|\geq 1$. We decompose the left hand side of (\ref{bfR2-rep}) dyadically in frequency, and use (\ref{profile_phia}) to rewrite the functions $\d_s\phi^{(a)}$ and $|\nab|\phi^{(c)}$ in terms of the profiles $\Psi^{(a)}$ and $\Psi^{(c)}$, thus it suffices to estimate
    \begin{equation*}
    \mathfrak{R}_{k,k_1,k_2}:=\int_1^t\int_{\R^6} e^{-is(|\xi|-|\eta|)}(1+\xi^{2n})|\eta|\overline{\widehat{P_k\Psi^{(a)}}}(\xi) \widehat{P_{k_1}\rr^{(b)}}(\xi-\eta) \widehat{P_{k_2}\Psi^{(c)}}(\eta) d\xi d\eta ds.
    \end{equation*}

    {\bf Low-high interaction: the contribution when $k_1\leq k+6$, $k-5\leq k_2\leq k+5$.}

    \noindent By integration by parts in $\eta$, we have
    \begin{align*}
    \mathfrak{R}_{k,k_1,k_2}=&i\int_1^t s^{-1}\int_{\R^6} \nab_{\eta}\big( |\eta|(1+\xi^{2n}) \big) \overline{ \widehat{P_k\Phi^{(a)}}}(\xi) \widehat{P_{k_1}\rr^{(b)}}(\xi-\eta) \widehat{P_{k_2}\Phi^{(c)}}(\eta)\\
    &+|\eta|(1+\xi^{2n})\overline{ \widehat{P_k\Phi^{(a)}}}(\xi) \nab_{\eta}\widehat{P_{k_1}\rr^{(b)}}(\xi-\eta) \widehat{P_{k_2}\Phi^{(c)}}(\eta)\\
    &+|\eta|(1+\xi^{2n})\overline{ \widehat{P_k\Phi^{(a)}}}(\xi) \widehat{P_{k_1}\rr^{(b)}}(\xi-\eta) e^{is|\eta|}\nab_{\eta}\widehat{P_{k_2}\Psi^{(c)}}(\eta) d\xi d\eta ds.
    \end{align*}
    We then use (\ref{Symbol_Pre}) and H\"{o}lder to bound this by
    \begin{align*}
    \sum_{k_1\leq k+6,k_2=k+O(1)}\mathfrak{R}_{k,k_1,k_2}\lesssim &\int_1^t s^{-1} \sum_{k_1\leq k+6,k_2=k+O(1)}\Big[ 2^{2N(a)k^+}\lV P_k\Phi^{(a)}\rV_{L^2}\lV P_{k_1}\rr^{(b)}\rV_{L^{\infty}}\lV P_{k_2}\Phi^{(c)}\rV_{L^2} \\
    &+2^{2N(a)k^++k_2} \lV P_k\Phi^{(a)}\rV_{L^2}\lV \mathcal{F}^{-1}(\nab_{\eta}\widehat{P_{k_1}\rr^{(b)}(\eta)})\rV_{L^{\infty}}\lV P_{k_2}\Phi^{(c)}\rV_{L^2}\\
    & 2^{2N(a)k^++k_2}\lV P_k\Phi^{(a)}\rV_{L^2}\lV P_{k_1}\rr^{(b)}\rV_{L^{\infty}}\lV \nab_{\eta}\widehat{P_{k_2}\Phi^{(c)}(\eta)}\rV_{L^2} \Big]ds\\
    \lesssim &\int_1^t s^{-1} \lV\Phi^{(a)}\rV_{H^{N(a)}}\Big[\lV\nab\rr^{(b)}\rV_{H^2}\lV\Phi^{(c)}\rV_{H^{N(a)}} +\sum_{k_1}\lV\<r\>P_{k_1}\rr^{(b)}\rV_{L^{\infty}}\lV\Phi^{(c)}\rV_{H^{N(a)+1}}\\
    &+\lV\nab\rr^{(b)}\rV_{H^2}\lV\FF^{-1}(|\xi|\nab_{\xi}\widehat{\Psi^{(c)}})\rV_{H^{N(a)}} \Big] ds.
    \end{align*}
    Since from (\ref{rv}), (\ref{Main_Prop_Ass1}), and $|b|\leq N_1-2$, we have
    \begin{align*}
    \sum_{k_1}\lV\<r\>P_{k_1}\rr^{(b)}\rV_{L^{\infty}}\lesssim & \sum_{k_1}\sum_{|\alpha|\leq 1}\lV\d_r\tilde{\Om}^{\alpha}P_{k_1}\rr^{(b)}\rV_{L^2}^{1/2}\sum_{|\alpha|\leq 2}\lV\tilde{\Om}^{\alpha}P_{k_1}\rr^{(b)}\rV_{L^2}^{1/2}\\
    \lesssim &\sum_{k_1}2^{k_1^-/2-k_1^+/2}\sum_{|\alpha|\leq 1}\lV\tilde{\Om}^{\alpha}P_{k_1}\rr^{(b)}\rV_{H^2}^{1/2}\sum_{|\alpha|\leq 2}\lV\tilde{\Om}^{\alpha}P_{k_1}\rr^{(b)}\rV_{L^2}^{1/2}\\
    \lesssim & \sum_{|\alpha|\leq 2}\lV\tilde{\Om}^{\alpha}\rr^{(b)}\rV_{H^2}\lesssim \ep_1.
    \end{align*}
    Combining this estimate, (\ref{Main_Prop_Ass1}) and (\ref{Main_Prop_Ass2}) we may continue to bound $\mathfrak{R}_{k,k_1,k_2}$ by
    \begin{align*}
    \sum_{k_1\leq k-5,k_2=k+O(1)}\mathfrak{R}_{k,k_1,k_2}\lesssim & \int_1^t s^{-1}\ep_1\<s\>^{H(a)\delta}(\ep_1^2\<s\>^{H(c)\delta}+\ep_1^2\<s\>^{H(c+1)\delta})ds\\
    \lesssim &\int_1^t \ep_1^3\<s\>^{-1+2H(a)\delta}ds\lesssim \ep_1^3\<t\>^{2H(a)\delta}.
    \end{align*}
    This concludes the treatment of low-high interaction.

    {\bf High-low interaction: The contribution when $k-3<k_1<k+3,k_2<k-5$.}

    \noindent Integration by parts in time, we divide $\mathfrak{R}_{k,k_1,k_2}$ into four terms
    \begin{equation}              \label{rrd2v-HL}
    \begin{aligned}
    \mathfrak{R}_{k,k_1,k_2}=& i\int_{\R^6}\frac{(1+\xi^{2n})|\eta|}{|\xi|-|\eta|} \overline{\widehat{P_k\Phi^{(a)}}}(\xi) \widehat{P_{k_1}\rr^{(b)}}(\xi-\eta) \widehat{P_{k_2}\Phi^{(c)}}(\eta) d\xi d\eta\Big|_1^t\\
    &-i\int_1^t\int_{\R^6}\frac{(1+\xi^{2n})|\eta|}{|\xi|-|\eta|} \overline{e^{is|\xi|}\widehat{P_k\d_s\Psi^{(a)}}}(\xi) \widehat{P_{k_1}\rr^{(b)}}(\xi-\eta) \widehat{P_{k_2}\Phi^{(c)}}(\eta) d\xi d\eta ds\\
    &-i\int_1^t\int_{\R^6}\frac{(1+\xi^{2n})|\eta|}{|\xi|-|\eta|} \overline{\widehat{P_k\Phi^{(a)}}}(\xi) \widehat{P_{k_1}\d_s\rr^{(b)}}(\xi-\eta) \widehat{P_{k_2}\Phi^{(c)}}(\eta) d\xi d\eta ds\\
    &-i\int_1^t\int_{\R^6}\frac{(1+\xi^{2n})|\eta|}{|\xi|-|\eta|} \overline{\widehat{P_k\Phi^{(a)}}}(\xi) \widehat{P_{k_1}\rr^{(b)}}(\xi-\eta) e^{is|\eta|}\widehat{P_{k_2}\d_s\Psi^{(c)}}(\eta) d\xi d\eta ds\\
    =:& \mathfrak{R}^0_{k,k_1,k_2}+\mathfrak{R}^1_{k,k_1,k_2}+\mathfrak{R}^2_{k,k_1,k_2}+\mathfrak{R}^3_{k,k_1,k_2}.
    \end{aligned}
    \end{equation}

    \emph{Case 1. The contribution of $\mathfrak{R}^0_{k,k_1,k_2}$.}

    \noindent By (\ref{BernIneq}) and (\ref{Main_Prop_Ass1}), we may control this by
    \begin{align*}
    \sum_{k_1=k+O(1),k_2\leq k-5}\mathfrak{R}_{k,k_1,k_2}^0\lesssim & \sum_{k_1=k+O(1),k_2\leq k-5} 2^{k_2-k+2N(a)k^+}\lV P_k\Phi^{(a)}\rV_{L^2}\lV P_{k_1}\rr^{(b)}\rV_{L^2}\lV P_{k_2}\Phi^{(c)}\rV_{L^{\infty}}\\
    \lesssim & \lV \Phi^{(a)}\rV_{H^{N(a)}}\lV \rr^{(b)}\rV_{H^{N(a)}}\lV \Phi^{(c)}\rV_{H^2}\lesssim \ep_1^3\<t\>^{H(a)\delta+H(c)\delta}\lesssim \ep_1^2\<t\>^{2H(a)\delta}.
    \end{align*}

    \emph{Case 2: The contribution of $\mathfrak{R}^1_{k,k_1,k_2}$.}

    \noindent By (\ref{BernIneq}), (\ref{dtPsi_L2_dec}) and (\ref{Main_Prop_Ass1}), we may bound this contribution when $|a|\geq N_1-1,|b|\geq 1, |c|\geq 3$ by
    \begin{align*}
    \sum_{k_1=k+O(1),k_2\leq k-5}\mathfrak{R}^1_{k,k_1,k_2}\lesssim & \int_1^t \sum_{k_1=k+O(1),k_2\leq k-5} 2^{k_2-k+2N(a)k^+}\lV P_k\d_s\Psi^{(a)}\rV_{L^2}\lV P_{k_1}\rr^{(b)}\rV_{L^6}\lV P_{k_2}\Phi^{(c)}\rV_{L^3}ds\\
    \lesssim &\int_1^t \lV\d_s\Psi^{(a)}\rV_{H^{N(a)-5}} \lV\nab\rr^{(b)}\rV_{H^{N(a)+4}}\lV\Phi^{(c)}\rV_{H^2}ds\\
    \lesssim & \int_1^t \ep_1^2\<s\>^{-3/4+(2+|a|)\delta+H(c)\delta}\lV\nab\rr^{(b)}\rV_{H^{N(b)}}ds\lesssim \ep_1^3.
    \end{align*}
    When $|a|\leq N_1-2$, we utilize (\ref{dtPsi_L2_dec-<4}), (\ref{Dec_Phi}) and (\ref{Main_Prop_Ass1}) to control the contribution by
    \begin{align*}
    \sum_{k_1=k+O(1),k_2\leq k-5}\mathfrak{R}^1_{k,k_1,k_2}\lesssim & \int_1^t \sum_{k_1=k+O(1),k_2\leq k-5} 2^{k_2-k+2N(a)k^+}\lV P_k\d_s\Psi^{(a)}\rV_{L^2}\lV P_{k_1}\rr^{(b)}\rV_{L^2}\lV P_{k_2}\Phi^{(c)}\rV_{L^{\infty}}ds\\
    \lesssim &\int_1^t \lV\d_s\Psi^{(a)}\rV_{H^{N(a)-1}} \lV\rr^{(b)}\rV_{H^{N(a)}}\sum_{k_2}2^{k_2^+}\lV\Phi^{(c)}\rV_{L^{\infty}}ds\\
    \lesssim & \int_1^t(\ep_1^2\<s\>^{-3/4+(2+|a|)\delta}+\ep_1\<t\>^{-1+3\delta}\sum_{|\tilde{b}|\leq |a|}\lV\nab u^{(\tilde{b})}\rV_{H^{N(a)}})\ep_1^2\<s\>^{-1+\delta+H(c+2)\delta} ds\lesssim \ep_1^3.
    \end{align*}

    \emph{Case 3: The contribution of $\mathfrak{R}^2_{k,k_1,k_2}$.}

    \noindent By (\ref{BernIneq}), (\ref{Main_Prop_Ass1}) and (\ref{dtrr-inf}) we may bound this contribution when $|a|\geq N_1-1,|b|\geq 1,|c|\geq 3$ by
    \begin{align*}
    \sum_{k_1=k+O(1),k_2\leq k-5}\mathfrak{R}^2_{k,k_1,k_2}\lesssim & \int_1^t \sum_{k_1=k+O(1),k_2\leq k-5} 2^{k_2-k+2N(a)k^+}\lV P_k\Phi^{(a)}\rV_{L^2}\lV P_{k_1}\d_s\rr^{(b)}\rV_{L^{\infty}}\lV P_{k_2}\Phi^{(c)}\rV_{L^2}ds\\
    \lesssim & \int_1^t \lV\Phi^{(a)}\rV_{H^{N(a)}}\sum_{k_1}2^{N(a)k_1^+}\lV P_{k_1}\d_s\rr^{(b)}\rV_{L^{\infty}}\lV\Phi^{(c)}\rV_{L^2}ds\\
    \lesssim & \int_1^t \ep_1^3 \<s\>^{H(a)\delta-1-l(b)/4+H(c)\delta}ds \lesssim \ep_1^3.
    \end{align*}
    By (\ref{Main_Prop_Ass1}) and (\ref{dtrr-L2}) we may bound this contribution when $|a|\leq N_1-2,|b|\geq 1$ by
    \begin{align*}
    \sum_{k_1=k+O(1),k_2\leq k-5}\mathfrak{R}^2_{k,k_1,k_2}\lesssim & \int_1^t \sum_{k_1=k+O(1),k_2\leq k-5} 2^{k_2-k+2N(a)k^+}\lV P_k\Phi^{(a)}\rV_{L^2}\lV P_{k_1}\d_s\rr^{(b)}\rV_{L^2}\lV P_{k_2}\Phi^{(c)}\rV_{L^{\infty}}ds\\
    \lesssim & \int_1^t \lV\Phi^{(a)}\rV_{H^{N(a)}}\lV \d_s\rr^{(b)}\rV_{H^{N(a)-1}}\lV\Phi^{(c)}\rV_{L^{\infty}}ds\\
    \lesssim & \int_1^t \ep_1^2 \<s\>^{H(a)\delta-1+\delta+H(c+2)\delta} \sum_{|b|\leq |a|}\lV(\nab\rr^{(b)},\nab u^{(b)})\rV_{H^{N(a)-1}} ds \lesssim \ep_1^3.
    \end{align*}

    \emph{Case 4: The contribution of $\mathfrak{R}^3_{k,k_1,k_2}$.}

    \noindent By (\ref{dtPsi_L2_dec}) and (\ref{Main_Prop_Ass1}) we bound this contribution by
    \begin{align*}
    \sum_{k_1=k+O(1),k_2\leq k-5}\mathfrak{R}^3_{k,k_1,k_2}\lesssim & \int_1^t\sum_{\substack{k_1=k+O(1),k_2\leq k-5,\\k_1\geq 0}} 2^{k_2-k+2N(a)k^+}\lV P_k\Phi^{(a)}\rV_{L^2}\lV P_{k_1}\rr^{(b)}\rV_{L^2}\lV P_{k_2}\d_s\Psi^{(c)}\rV_{L^{\infty}}\\
    &+\sum_{\substack{k_1=k+O(1),k_2\leq k-5,\\k_1< 0}} 2^{k_2-k}\lV P_k\Phi^{(a)}\rV_{L^2}\lV P_{k_1}\rr^{(b)}\rV_{L^{\infty}}\lV P_{k_2}\d_s\Psi^{(c)}\rV_{L^2}ds\\
    \lesssim & \int_1^t\lV\Phi^{(a)}\rV_{H^{N(a)}}\lV \nab\rr^{(b)}\rV_{H^{N(a)-1}}\lV \d_s\Psi^{(c)}\rV_{H^2}ds\\
    \lesssim & \int_1^t \ep_1^4 \<s\>^{H(a)\delta-3/4+(2+|c|)\delta}\lV \nab\rr^{(b)}\rV_{H^{N(a)-1}}ds\lesssim \ep_1^4.
    \end{align*}
    This concludes the treatment of high-low interaction.

    {\bf High-high interaction: the contribution when $k_1=k_2+O(1)$, $k_2>k+5$.}

    \noindent This case will follows as easily as the high-low interaction, by (\ref{rrd2v-HL}) it suffices to show
    \begin{equation*}
    \sum_{k_2>k+5,k_1=k_2+O(1)}(\mathfrak{R}^0_{k,k_1,k_2}+\mathfrak{R}^1_{k,k_1,k_2}+\mathfrak{R}^2_{k,k_1,k_2}+\mathfrak{R}^3_{k,k_1,k_2})\lesssim \ep_1^3\<t\>^{2H(a)\delta}.
    \end{equation*}

    \emph{Case 1: The contribution of $\mathfrak{R}^0_{k,k_1,k_2}$.}

    \noindent By (\ref{Main_Prop_Ass1}) we have
    \begin{align*}
    \sum_{k_2>k+5,k_1=k_2+O(1)}\mathfrak{R}^0_{k,k_1,k_2}\lesssim &\sum_{k_2>k+5,k_1=k_2+O(1)} 2^{2N(a)k^+}\lV P_k\Phi^{(a)}\rV_{L^{\infty}}\lV P_{k_1}\rr^{(b)}\rV_{L^2}\lV P_{k_2}\Phi^{(c)}\rV_{L^2}\\
    \lesssim &\lV \Phi^{(a)}\rV_{H^2} \lV \rr^{(b)}\rV_{H^{N(a)}}\lV \Phi^{(c)}\rV_{H^{N(a)}}\lesssim \ep_1^3\<t\>^{2H(a)\delta}.
    \end{align*}

    \emph{Case 2: The contribution of $\mathfrak{R}^1_{k,k_1,k_2}$.}

    \noindent By (\ref{dtPsi_L2_dec}) and (\ref{rra,v,w_inf}) we may bound this contribution by
    \begin{align*}
    \sum_{k_2>k+5,k_1=k_2+O(1)}\mathfrak{R}^1_{k,k_1,k_2}\lesssim &\int_1^t \sum_{k_2>k+5,k_1=k_2+O(1)} 2^{2N(a)k^+}\lV P_k\d_s\Psi^{(a)}\rV_{L^{2+\delta}}\lV P_{k_1}\rr^{(b)}\rV_{L^{2(2+\delta)/\delta}}\lV P_{k_2}\Phi^{(c)}\rV_{L^2}ds\\
    \lesssim & \int_1^t \lV \d_s\Psi^{(a)}\rV_{H^{N(a)/2}}\lV \rr^{(b)}\rV_{W^{N(a)/2,2(2+\delta)/\delta}}\lV\Phi^{(c)}\rV_{H^{N(a)}}ds\\
    \lesssim & \int_1^t \ep_1^4\<s\>^{-3/4+(2+|a|)\delta-3/4+(3+|b|)\delta+H(c)\delta}ds \lesssim \ep_1^4.
    \end{align*}

    \emph{Case 3: The contribution of $\mathfrak{R}^2_{k,k_1,k_2}$.}

    \noindent By (\ref{BernIneq}), (\ref{Main_Prop_Ass1}) and (\ref{dtrr-inf}) we may bound this contribution when $|a|\geq N_1-1,|b|\geq 1,|c|\geq 3$ by
    \begin{align*}
    \sum_{k_2>k+5,k_1=k_2+O(1)}\mathfrak{R}^2_{k,k_1,k_2}\lesssim & \int_1^t \sum_{\substack{k_2>k+5,k_1=k_2+O(1),\\k<0}} \lV P_k\Phi^{(a)}\rV_{L^{2+\delta}}\lV P_{k_1}\d_s\rr^{(b)}\rV_{L^{2(2+\delta)/\delta}}\lV P_{k_2}\Phi^{(c)}\rV_{L^2}\\
    	&+\sum_{\substack{k_2>k+5,k_1=k_2+O(1),\\k\geq 0}}2^{2N(a)k^+}\lV P_k\Phi^{(a)}\rV_{L^2}\lV P_{k_1}\d_s\rr^{(b)}\rV_{L^{\infty}}\lV P_{k_2}\Phi^{(c)}\rV_{L^2}ds\\
    \lesssim & \int_1^t \lV\Phi^{(a)}\rV_{H^{N(a)}}\sum_{k_1}2^{k_1^+}\lV P_{k_1}\d_s\rr^{(b)}\rV_{L^{2(2+\delta)/\delta}}\lV\Phi^{(c)}\rV_{H^{N(a)}}ds\\
    \lesssim & \int_1^t \ep_1^3 \<s\>^{H(a)\delta-1-l(b)/5+H(c)\delta}ds \lesssim \ep_1^3.
    \end{align*}
    By (\ref{Main_Prop_Ass1}), (\ref{dtrr-L2}) and (\ref{Dec_Phi}) we may bound this contribution when $|a|\leq N_1-2,|b|\geq 1$ by
    \begin{align*}
    \sum_{k_2>k+5,k_1=k_2+O(1)}\mathfrak{R}^2_{k,k_1,k_2}\lesssim & \int_1^t \sum_{\substack{k_2>k+5,k_1=k_2+O(1),\\k<0}} \lV P_k\Phi^{(a)}\rV_{L^{2+\delta}}\lV P_{k_1}\d_s\rr^{(b)}\rV_{L^2}\lV P_{k_2}\Phi^{(c)}\rV_{L^{2(2+\delta)/\delta}}\\
    &+\sum_{\substack{k_2>k+5,k_1=k_2+O(1),\\k\geq 0}}2^{2N(a)k^+}\lV P_k\Phi^{(a)}\rV_{L^2}\lV P_{k_1}\d_s\rr^{(b)}\rV_{L^2}\lV P_{k_2}\Phi^{(c)}\rV_{L^{\infty}}ds\\
    \lesssim & \int_1^t \lV\Phi^{(a)}\rV_{H^{N(a)}}\lV \d_s\rr^{(b)}\rV_{H^{N(a)-1}}\sum_{k_2}2^{2k_2^+}\lV P_{k_2}\Phi^{(c)}\rV_{L^{2(2+\delta)/\delta}}ds\\
    \lesssim & \int_1^t \ep_1^2 \<s\>^{H(a)\delta-1+\delta+H(c+2)\delta}\sum_{|b|\leq |a|}\lV(\nab\rr^{(b)},\nab u^{(b)})\rV_{H^{N(a)-1}} ds \lesssim \ep_1^3.
    \end{align*}

    \emph{Case 4: The contribution of $\mathfrak{R}^3_{k,k_1,k_2}$.}

    \noindent By (\ref{dtPsi_L2_dec}) and (\ref{Main_Prop_Ass1}) we bound this contribution by
    \begin{align*}
    \sum_{k_2>k+5,k_1=k_2+O(1)}\mathfrak{R}^3_{k,k_1,k_2}\lesssim & \int_1^t\sum_{\substack{k_2>k+5,k_1=k_2+O(1),\\k\geq 0}} 2^{2N(a)k^+}\lV P_k\Phi^{(a)}\rV_{L^2}\lV P_{k_1}\rr^{(b)}\rV_{L^2}\lV P_{k_2}\d_s\Psi^{(c)}\rV_{L^{\infty}}\\
    &+\sum_{\substack{k_2>k+5,k_1=k_2+O(1),\\k< 0}} \lV P_k\Phi^{(a)}\rV_{L^{2+\delta}}\lV P_{k_1}\rr^{(b)}\rV_{L^{2(2+\delta)/\delta}}\lV P_{k_2}\d_s\Psi^{(c)}\rV_{L^2}ds\\
    \lesssim & \int_1^t\lV\Phi^{(a)}\rV_{H^{N(a)}}\lV \nab\rr^{(b)}\rV_{H^{N(a)-1}}\lV \d_s\Psi^{(c)}\rV_{H^2}ds\\
    \lesssim & \int_1^t \ep_1^4 \<s\>^{H(a)\delta-3/4+(2+|c|)\delta}\lV\nab\rr^{(b)}\rV_{H^{N(a)-1}}ds\lesssim \ep_1^4.
    \end{align*}
    This concludes the treatment of the high-high interaction, and thus of Step 2.

     \textbf{Step 3. We prove the bound
    	\begin{equation*}
    	\int_0^t \bfR_3(s)ds\lesssim \ep_1^3\<t\>^{2H(a)\delta}.
    	\end{equation*}}

    \emph{Case 1: Prove the bound
    \begin{equation*}
    \int_0^t \int_{\R^3} \sum_{|n|\in\{0,N(a)\},b+c=a} \rr\d^n\d_s\phi^{(a)}\d^n\big( u^{(b)}\cdot\nab\d_s\phi^{(c)} \big)dxds\lesssim \ep_1^4.
    \end{equation*}}

    When $c=a$, integration by parts in space yields
    \begin{equation}        \label{case3(a)-c=a}
    \begin{aligned}
    &\int_0^t \int_{\R^3} \sum_{|n|\in\{0,N(a)\}} \rr\d^n\d_s\phi^{(a)}\d^n\big( u\cdot\nab\d_s\phi^{(a)} \big)dxds\\
    =&\int_0^t \int_{\R^3} \sum_{|n|\in\{0,N(a)\}} \rr\d^n\d_s\phi^{(a)}\big( u\cdot\nab\d^n\d_s\phi^{(a)} +\sum_{n_1+n_2=n}\d^{n_1}u\cdot\nab\d^{n_2}\d_s\phi^{(a)}\big)dxds\\
    =&\int_0^t \int_{\R^3} \sum_{|n|\in\{0,N(a)\}}\Big[ -\frac{1}{2}(u\cdot\nab\rr+\rr\nab\cdot u)|\d^n\d_s\phi^{(a)}|^2 \\
    &\ \ \ \ \ +\sum_{\substack{n_1+n_2=n,\\|n_1|\geq 1}}\rr\d^n\d_s\phi^{(a)}\big(\d^{n_1}u\cdot\nab\d^{n_2}\d_s\phi^{(a)}\big)\Big]dxds,
    \end{aligned}
    \end{equation}
    By (\ref{rr,v,w_inf}) and (\ref{Main_Prop_Ass1}) we may control this term by
    \begin{align*}
    {\rm LHS}(\ref{case3(a)-c=a})\lesssim &\int_0^t \lV\rr\rV_{W^{1,\infty}}\lV\nab u\rV_{H^2}\lV \d_s\phi^{(a)}\rV_{H^{N(a)}}^2+\lV\rr\rV_{\infty}\lV\d_s\phi^{(a)}\rV_{H^{N(a)}}^2\lV\nab u\rV_{H^{N(a)}}ds\\
    \lesssim &\int_0^t \ep_1^3\<s\>^{-3/4+\delta+2H(a)\delta}\lV \nab u\rV_{H^{N(a)}}ds\lesssim \ep_1^4.
    \end{align*}

    When $|c|\leq |a|-1$, by H\"{o}lder we have
    \begin{align*}
    \int_0^t\int_{\R^3} \sum_{\substack{|n|\in\{0,N(a)\},\\b+c=a,|c|\leq |a|-1}} \rr\d^n\d_s\phi^{(a)}\d^n\big( u^{(b)}\cdot\nab\d_s\phi^{(c)} \big)dx\lesssim& \int_0^t \lV\rr\rV_{\infty}\lV\d_s\phi^{(a)}\rV_{H^{N(a)}}\lV u^{(b)}\nab\d_s\phi^{(c)}\rV_{H^{N(a)}}ds\\
    \lesssim & \int_0^t \ep_1^2\<s\>^{-3/4+\delta+H(a)\delta} \lV u^{(b)}\nab\d_s\phi^{(c)}\rV_{H^{N(a)}}ds.
    \end{align*}
    Noticing that by $|c|\leq |a|-1$ and (\ref{Main_Prop_Ass1}) we have
    \begin{equation*}
    \lV u^{(b)}\nab\d_s\phi^{(c)}\rV_{H^{N(a)}}\lesssim \lV\nab u^{(b)}\rV_{H^{N(a)}}\lV\d_s\phi^{(c)}\rV_{H^{N(a)+1}}\lesssim \ep_1\<s\>^{H(c)\delta}\lV \nab u^{(b)}\rV_{H^{N(a)}}.
    \end{equation*}
    Combining these estimates we see that
    \begin{equation*}
    \int_0^t \ep_1^2\<s\>^{-3/4+H(a)\delta} \lV u^{(b)}\nab\d_s\phi^{(c)}\rV_{H^{N(a)}}ds\lesssim \int_0^t \ep_1^3 \<s\>^{-3/4+\delta+H(a)\delta+H(c)\delta}\lV\nab u^{(b)}\rV_{H^{N(a)}} ds\lesssim \ep_1^4.
    \end{equation*}
    This concludes the bound of Case 1.

    \emph{Case 2: Prove the following bound for $c=a$ or $|a|\geq N_1-1,|c|\leq 2$}
    \begin{equation*}
    \int_0^t\int_{\R^3} \sum_{\substack{|n|\in\{0,N(a)\},\\b+c=a}} \d^n\d_s\phi^{(a)}\d^n\big( u^{(b)}\cdot\nab\d_s\phi^{(c)} \big)dx\lesssim \ep_1^3\<t\>^{2H(a)\delta}.
    \end{equation*}
    When $c=a$, by integration by parts, (\ref{drr,v,w_inf}) and (\ref{Dec_Phi}) we may bound the left-hand side by
    \begin{align*}
    &\int_0^t\int_{\R^3} \sum_{|n|\in\{0,N(a)\}} \d^n\d_s\phi^{(a)}\d^n\big( u\cdot\nab\d_s\phi^{(a)} \big)dx\\
    =&\int_0^t \int_{\R^3} \sum_{|n|\in\{0,N(a)\}}\Big[-\frac{1}{2}\nab\cdot u|\d^n\d_s\phi^{(a)}|^2 +\d^n\d_s\phi^{(a)}\big(\sum_{n_1+n_2=n,|n_1|\geq 1} \d^{n_1}u\cdot\nab\d^{n_2}\d_s\phi^{(a)}\big)\Big]dxds\\
    \lesssim &\int_0^t\lV\nab u\rV_{L^{\infty}}\lV\d_s\phi^{(a)}\rV_{H^{N(a)}}^2+\sum_{|a|\geq N_1-1}\lV\d_s\phi^{(a)}\rV_{H^{N(a)}}^2\lV \nab u\rV_{W^{N(a)-1,\infty}}\\
    &+\sum_{|a|\leq N_1-2}\lV\d_s\phi^{(a)}\rV_{H^{N(a)}}\lV\nab u\rV_{H^{N(a)}}\lV\d_s\phi^{(a)}\rV_{W^{N(a),\infty}}ds\\
    \lesssim & \int_0^t  \ep_1^3\<s\>^{-1-l(0)+3\delta+2H(a)\delta}+\ep_1^2\<s\>^{-1+\delta+H(a+2)\delta}\lV\nab u\rV_{H^{N(a)}}ds \lesssim \ep_1^3.
    \end{align*}
    When $|a|\geq N_1-1,|c|\leq 2$, by (\ref{Dec_Phi}) we have
    \begin{align*}
    \int_0^t\int_{\R^3} \sum_{|n|\in\{0,N(a)\}} \d^n\d_s\phi^{(a)}\d^n\big( u^{(b)}\cdot\nab\d_s\phi^{(c)} \big)dx\lesssim &\int_0^t \lV \d_s\phi^{(a)}\rV_{H^{N(a)}}\lV u^{(b)}\rV_{H^{N(a)}}\lV\d_s\phi^{(c)}\rV_{W^{N(a)+1,\infty}}ds\\
    \lesssim & \int_0^t \ep_1^3 \<s\>^{H(a)\delta-1+\delta+H(c+2)\delta}ds\lesssim \ep_1^3\<t\>^{2H(a)\delta}.
    \end{align*}

    \emph{Case 3: Prove the following bound for $|a|\geq N_1-1,3\leq|c|\leq |a|-1$ or $|a|\leq N_1-2,|c|\leq |a|-1$,}
    \begin{equation}         \label{ud2phi}
    \int_0^t\int_{\R^3} \sum_{|n|\in\{0,N(a)\}} \d^n\d_s\phi^{(a)}\d^n\big( u^{(b)}\cdot\nab\d_s\phi^{(c)} \big)dx\lesssim \ep_1^3\<t\>^{2H(a)\delta}.
    \end{equation}
    We decompose dyadically in frequency and rewrite the functions $\d_s\phi^{(a)}$ and $\d_s\phi^{(c)}$ in terms of the profiles $\Psi^{(a)}$ and $\Psi^{(c)}$, thus it suffices to estimate
    \begin{equation*}
    \mathfrak{Q}_{k,k_1,k_2}:=\int_1^t\int_{\R^6} e^{-is(|\xi|-|\eta|)}(1+\xi^{2n})\eta\overline{\widehat{P_k\Psi^{(a)}}}(\xi) \widehat{P_{k_1}u^{(b)}}(\xi-\eta) \widehat{P_{k_2}\Psi^{(c)}}(\eta) d\xi d\eta ds.
    \end{equation*}
    This bound has similarity with the analysis in \textbf{Step 2} in this Lemma. Hence, we may use the same argument and the corresponding properties of $u^{(b)}$ to obtain the bound (\ref{ud2phi}). This completes the treatment of \textbf{Step 3}.

    \textbf{Step 4. We prove the bound
    	\begin{equation*}
    	\int_0^t \bfR_4(s)ds\lesssim \ep_1^3\<t\>^{2H(a)\delta}.
    	\end{equation*}}

    We first consider the contribution when $|n|=0$. Notice that when $|b|\leq |c|$, by (\ref{dtrr-inf}) and (\ref{Main_Prop_Ass1}) we have
    \begin{equation*}
    \lV\d_su^{(b)}\cdot\nab\phi^{(c)}\rV_{L^2}\lesssim \lV\d_su^{(b)}\rV_{L^{\infty}}\lV \nab\phi^{(c)}\rV_{L^2}\lesssim \ep_1^2\<s\>^{-1-l(b)/4+H(c)\delta},
    \end{equation*}
    when $|b|>|c|$, by (\ref{dtu-L2}) and (\ref{Dec_Phi}) we have
    \begin{equation*}
    \lV\d_su^{(b)}\cdot\nab\phi^{(c)}\rV_{L^2}\lesssim \lV\d_su^{(b)}\rV_{L^2}\lV \nab\phi^{(c)}\rV_{L^{\infty}}\lesssim (\lV(\nab\rr^{(b)},\nab u^{(b)})\rV_{H^2}+\ep_1^2\<s\>^{-1/2}) \ep_1\<s\>^{-1+\delta+H(c+2)\delta}.
    \end{equation*}
    From these estimates and (\ref{Main_Prop_Ass1}) we get
    \begin{align*}
    &\int_0^t\int_{\R^3} \sum_{b+c=a}C_a^b (1+\rr)\d_s\phi^{(a)}\big( \d_su^{(b)}\cdot\nab\phi^{(c)} \big)dxds\\
    \lesssim &\int_0^t \lV 1+\rr\rV_{L^{\infty}}\lV\d_s\phi^{(a)}\rV_{L^2}\lV \d_su^{(b)}\cdot\nab\phi^{(c)}\rV_{L^2}ds\\
    \lesssim &\int_0^t (1+\ep_1)\ep_1\<s\>^{H(a)\delta}\big( \ep_1\<s\>^{-1+\delta+H(c+2)\delta}\lV(\nab\rr^{(b)},\nab u^{(b)})\rV_{H^2}+\ep_1^2\<s\>^{-1-l(b)/4+H(c)\delta} \big)ds\\
    \lesssim &  \ep_1^3.
    \end{align*}

    Next we prove the following bound
    	\begin{align}         \label{case4(b)}
    	\int_0^t\int_{\R^3} \sum_{|n|=N(a),b+c=a} (1+\rr)\d^n\d_s\phi^{(a)}\d^n\big( \d_su^{(b)}\cdot\nab\phi^{(c)} \big)dxds\lesssim \ep_1^3\<t\>^{2H(a)\delta}.
    	\end{align}.

    \noindent We divide the left-hand side into three terms
    \begin{align*}
    {\rm LHS}(\ref{case4(b)})=& \int_0^t\int_{\R^3} \sum_{|n|=N(a)} (1+\rr)\d^n\d_s\phi^{(a)}\big( \d^n\d_su^{(a)}\cdot\nab\phi \big)dxds\\
    &+\int_0^t\int_{\R^3} \sum_{|n|=N(a)} (1+\rr)\d^n\d_s\phi^{(a)}\big(\sum_{n_1+n_2=n,|n_2|\geq 1}\d^{n_1} \d_su^{(a)}\cdot\d^{n_2}\nab\phi \big)dxds\\
    &+\int_0^t\int_{\R^3} \sum_{|n|=N(a),b+c=a,|b|<|a|} (1+\rr)\d^n\d_s\phi^{(a)}\d^n\big( \d_su^{(b)}\cdot\nab\phi^{(c)} \big)dxds\\
    =:& \mathfrak{I}_1+\mathfrak{I}_2+\mathfrak{I}_3
    \end{align*}

    \emph{Case 1: The contribution of $\mathfrak{I}_1$.}

    \noindent By integration by parts in time, $\mathfrak{I}_1$ can be rewritten as
    \begin{align*}
    \mathfrak{I}_1=&\int_{\R^3} \sum_{|n|=N(a)} (1+\rr)\d^n\d_s\phi^{(a)}\big( \d^nu^{(a)}\cdot\nab\phi \big)dx\big|_0^t
    -\int_0^t\int_{\R^3} \sum_{|n|=N(a)} \d_s\rr\d^n\d_s\phi^{(a)}\big( \d^nu^{(a)}\cdot\nab\phi \big)dxds\\
    &-\int_0^t\int_{\R^3} \sum_{|n|=N(a)} (1+\rr)\d^n\d_s\phi^{(a)}\big( \d^nu^{(a)}\cdot\nab \d_s\phi \big)dxds\\
    &-\int_0^t\int_{\R^3} \sum_{|n|=N(a)} (1+\rr)\d^n\d_s^2\phi^{(a)}\big( \d^nu^{(a)}\cdot\nab\phi \big)dxds\\
    =:& \mathfrak{I}_{11}+\mathfrak{I}_{12}+\mathfrak{I}_{13}+\mathfrak{I}_{14}.
    \end{align*}
    By (\ref{Dec_Phi}), (\ref{dtrr-L2}) and (\ref{Main_Prop_Ass1}) we can bound the first three terms $\mathfrak{I}_{11},\mathfrak{I}_{12}$ and $\mathfrak{I}_{13}$ by
    \begin{align*}
    \mathfrak{I}_{11}+\mathfrak{I}_{12}+\mathfrak{I}_{13}\lesssim &\sup_{s\in[0,t]}\lV\d_s\phi^{(a)}\rV_{H^{N(a)}}\lV  u^{(a)}\rV_{H^{N(a)}}\lV\nab\phi\rV_{L^{\infty}}\\
    &+\int_0^t \lV\d_s\rr\rV_{L^{\infty}} \lV\d_s\phi^{(a)}\rV_{H^{N(a)}}\lV \nab u^{(a)}\rV_{H^{N(a)}}\lV\nab\phi\rV_{L^{\infty}}ds\\
    &+\int_0^t  \lV\d_s\phi^{(a)}\rV_{H^{N(a)}}\lV \nab u^{(a)}\rV_{H^{N(a)}}\lV\nab\d_s\phi\rV_{L^{\infty}}ds\\
    \lesssim & \sup_{s\in[0,t]}\ep_1^3\<s\>^{H(a)\delta-1+2\delta}+\int_0^t\ep_1^2 \<s\>^{H(a)\delta-1+2\delta}\lV\nab u^{(a)}\rV_{H^{N(a)}}ds\lesssim \ep_1^3.
    \end{align*}
    For $\mathfrak{I}_{14}$, let $e\in\{(1,0,0),(0,1,0),(0,0,1)\}$, from integration by parts in space and $\phi^{(a)}$-equation in (\ref{Main_Sys_VecFie}) we have
    \begin{align*}
    \mathfrak{I}_{14}=&\int_0^t\int_{\R^3}\sum_{|n|=N(a)}\d^{n-e}\d_s^2\phi^{(a)} [\d^e\rr\d^n u^{(a)}\cdot\nab\phi+(1+\rr)\d^e(\d^nu^{(a)}\cdot\nab\phi)]dxds\\
    \lesssim & \int_0^t \lV \d_s^2\phi^{(a)}\rV_{H^{N(a)-1}}(1+\lV \nab\rr\rV_{L^{\infty}})\lV\nab u^{(a)}\rV_{H^{N(a)}}\lV\nab\phi\rV_{W^{1,\infty}}ds\\
    \lesssim & \int_0^t\ep_1\<s\>^{H(a)\delta}(1+\ep_1)\ep_1\<s\>^{-1+2\delta}\lV\nab u^{(a)}\rV_{H^{N(a)}}ds\lesssim \ep_1^3.
    \end{align*}

    \emph{Case 2: The contribution of $\mathfrak{I}_2$.}

    \noindent When $|a|\geq N_1-1$, by (\ref{Main_Prop_Ass1}), (\ref{dt_v_2norm_Lem}) and (\ref{Dec_Phi}) we may bound this contribution by
    \begin{align*}
    \mathfrak{I}_{2}\lesssim &\int_0^t \lV1+\rr\rV_{L^{\infty}}\lV\d_s\phi^{(a)}\rV_{H^{N(a)}}\lV\d_s u^{(a)}\rV_{H^{N(a)-1}}\lV\nab\phi\rV_{W^{N(a),\infty}}ds\\
    \lesssim &\int_0^t (1+\ep_1)\ep_1\<s\>^{H(a)\delta}(\lV\nab u^{(a)}\rV_{H^{N(a)}}+\lV\nab\rr^{(a)}\rV_{H^{N(a)-1}}+\ep_1^2\<s\>^{-1/2})\ep_1\<s\>^{-1+2\delta}ds\lesssim \ep_1^3.
    \end{align*}
    When $|a|\leq N_1-2$, by (\ref{Main_Prop_Ass1}), (\ref{dtu-L2}), (\ref{Dec_Phi}) and (\ref{dtrr-inf}) we have
    \begin{align*}
    \mathfrak{I}_2\lesssim & \int_0^t \lV1+\rr\rV_{L^{\infty}}\lV\d_s\phi^{(a)}\rV_{H^{N(a)}}\big[\lV\d_s u^{(a)}\rV_{H^{N(a)-1}}\lV\nab\phi\rV_{W^{N(a)/2,\infty}}+\lV\d_s u^{(a)}\rV_{W^{N(a)/2,\infty}}\lV\nab\phi\rV_{H^{N(a)}}\big]ds\\
    \lesssim & \int_0^t (1+\ep_1)\ep_1\<s\>^{H(a)\delta}\big[(\lV\nab u^{(a)}\rV_{H^{N(a)}}+\lV\nab\rr^{(a)}\rV_{H^{N(a)-1}}+\ep_1^2\<s\>^{-1/2})\ep_1\<s\>^{-1+2\delta}\\
    &+\ep_1^2\<s\>^{-1-l(a)/4+H(a)\delta}   \big]ds\\
    \lesssim &\ep_1^3.
    \end{align*}

    \emph{Case 3: The contribution of $\mathfrak{I}_3$.}

    \noindent This case follows as easily as \emph{Case 4(b).2} above. Before estimate $\mathfrak{I}_3$, we control the term $\d^n (\d_s u^{(b)}\cdot\nab\phi^{(c)})$ for $|n|=N(a),|b|<|a|$. When $|a|\geq N_1-1,|b|\leq |c|$, we have
    \begin{align*}
    \lV \d^n(\d_s u^{(b)}\cdot\nab\phi^{(c)})\rV_{L^2}\lesssim \lV \d_s u^{(b)}\rV_{W^{N(a),\infty}}\lV\nab\phi^{(c)}\rV_{H^{N(a)}}\lesssim \ep_1^2\<s\>^{-1-l(b)/4+H(c)\delta};
    \end{align*}
    when $|a|\geq N_1-1,|b|>|c|$, by (\ref{dt_v_2norm_Lem}) and (\ref{Dec_Phi}) we have
    \begin{align*}
    \lV \d^n(\d_s u^{(b)}\cdot\nab\phi^{(c)})\rV_{L^2}\lesssim& \lV \d_s u^{(b)}\rV_{H^{N(a)}}\lV\nab\phi^{(c)}\rV_{W^{N(a),\infty}}\\
    \lesssim &(\lV (\nab u^{(b)},\nab\rr^{(b)})\rV_{H^{N(a)}}+\ep_1^2\<s\>^{-1/2})\ep_1^2\<s\>^{-1+\delta+H(c+2)\delta};
    \end{align*}
    when $|a|\leq N_1-2$, we have
    \begin{align*}
    \lV \d^n(\d_s u^{(b)}\cdot\nab\phi^{(c)})\rV_{L^2}\lesssim& \lV \d_s u^{(b)}\rV_{H^{N(a)}}\lV\nab\phi^{(c)}\rV_{L^{\infty}}+\lV \d_s u^{(b)}\rV_{L^{\infty}}\lV\nab\phi^{(c)}\rV_{H^{N(a)}}\\
    \lesssim &(\lV (\nab u^{(b)},\nab\rr^{(b)})\rV_{H^{N(a)}}+\ep_1^2\<s\>^{-1/2})\ep_1^2\<s\>^{-1+\delta+H(c+2)\delta}\\
    &+\ep_1^2\<s\>^{-1-l(b)/4+H(c)\delta}.
    \end{align*}
    Combining these estimates we may bound this contribution by
    \begin{align*}
    \mathfrak{I}_3\lesssim \int_0^t \lV1+\rr\rV_{L^{\infty}}\lV\d_s\phi^{(a)}\rV_{H^{N(a)}}\lV\d^{N(a)}(\d_s u^{(b)}\cdot\nab\phi^{(c)})\rV_{L^2}ds \lesssim \ep_1^3.
    \end{align*}
    This concludes the bound (\ref{case4(b)}), and thus of \textbf{Step 4}.

     \textbf{Step 5. We prove the bound
    	\begin{equation*}
    	\int_0^t \bfR_5(s)ds\lesssim \ep_1^3\<t\>^{2H(a)\delta}.
    	\end{equation*}}

    \emph{Case 1: The contribution when $e=a,|n|=N(a)$.}

    \noindent We may rewrite this contribution by
    \begin{align*}
    &\int_0^t\int_{\R^3} \sum_{|n|=N(a)}(1+\rr)\d^n\d_s\phi^{(a)}\d^n\big( u\cdot\nab(u\cdot\nab\phi^{(a)}) \big)dxds\\
    =& \int_0^t\int_{\R^3} \sum_{|n|=N(a)}(1+\rr)\d^n\d_s\phi^{(a)} u\cdot\nab(u\cdot\nab\d^n\phi^{(a)})dxds\\
    &+\int_0^t\int_{\R^3} \sum_{|n|=N(a)}(1+\rr)\d^n\d_s\phi^{(a)}\sum_{\substack{n_1+n_2+n_3=n,\\|n_3|<|n|}} \d^{n_1}u\cdot\nab(\d^{n_2}u\cdot\nab\d^{n_3}\phi^{(a)}) dxds\\
    =:&\mathfrak{L}_1+\mathfrak{L}_2.
    \end{align*}
    By integration by parts, we can estimate the first term $\mathfrak{L}_1$ by
    \begin{align*}
    \mathfrak{L}_1=&\int_0^t\int_{\R^3} \sum_{|n|=N(a)}\big[-(1+\rr)u\cdot\nab\d^n\d_s\phi^{(a)} -(u\cdot\rr+(1+\rr)\nab\cdot u)\d^n\d_s\phi^{(a)}\big](u\cdot\nab\d^n\phi^{(a)})dxds\\
    =& \int_{\R^3}\sum_{|n|=N(a)}-\frac{1}{2}(1+\rr)|u\cdot\nab\d^n\phi^{(a)}|^2 dx\Big|_0^t+\int_0^t\int_{\R^3}\sum_{|n|=N(a)}\frac{1}{2}\d_s\rr|u\cdot\nab\d^n\phi^{(a)}|^2 dxds\\
    &+\int_0^t\int_{\R^3} \sum_{|n|=N(a)}\big[-(1+\rr)\d_su\cdot\nab\d^n\phi^{(a)} -(u\cdot\rr+(1+\rr)\nab\cdot u)\d^n\d_s\phi^{(a)}\big](u\cdot\nab\d^n\phi^{(a)})dxds.
    \end{align*}
    Applying (\ref{Main_Prop_Ass1}), (\ref{dtrr-L2}) and (\ref{dtu-L2}) we can bound this by
    \begin{align*}
    \lesssim & \sup_{s\in[0,t]}\lV  u\rV_{H^2}^2\lV \nab\phi^{(a)}\rV_{H^{N(a)}}^2+\int_0^t \lV\d_s\rr\rV_{H^2}\lV \nab u\rV_{H^2}^2\lV \nab\phi^{(a)}\rV_{H^{N(a)}}^2ds\\
    &+\int_0^t (\lV\d_s u\rV_{H^2}\lV\nab\phi^{(a)}\rV_{H^{N(a)}}+\lV(\nab\rr,\nab u)\rV_{H^2}\lV\d_s\phi^{(a)}\rV_{H^{N(a)}})\lV\nab u\rV_{H^2}\lV\nab\phi^{(a)}\rV_{H^{N(a)}}ds\\
    \lesssim & \ep_1^4\<t\>^{2H(a)\delta}+\int_0^t \ep_1^3\<s\>^{2H(a)\delta}\lV\nab u\rV_{H^2}^2ds\\
    &+\int_0^t  (\lV (\nab\rr,\nab u)\rV_{H^4}+\ep_1^2\<s\>^{-1/2})\ep_1\<s\>^{H(a)\delta}  \lV\nab u\rV_{H^2} \ep_1\<s\>^{H(a)\delta}ds\\
    \lesssim & \ep_1^4\<t\>^{2H(a)\delta}.
    \end{align*}
    For second term $\mathfrak{L}_2$, by H\"{o}lder and (\ref{Main_Prop_Ass1}) we have
    \begin{align*}
    \mathfrak{L}_2\lesssim \int_0^t \lV\d_s\phi^{(a)}\rV_{H^{N(a)}}\lV\nab u\rV_{H^{N(a)}}^2\lV\nab\phi^{(a)}\rV_{H^{N(a)}}ds\lesssim \ep_1^4\<t\>^{2H(a)\delta}.
    \end{align*}

    \emph{Case 2: The contribution when $e=a,|n|=0$ and $|e|<|a|$.}

    \noindent By H\"{o}lder and (\ref{Main_Prop_Ass1}) we control this by
    \begin{align*}
    &\int_0^t\int_{\R^3}(1+\rr)\d_s\phi^{(a)} u\cdot\nab(u\cdot\nab\phi^{(a)})\\
    &\ \ \ + \sum_{\substack{|n|\in\{0,N(a)\},\\b+c+e=a,|e|<|a|}}C_a^{b,c} (1+\rr)\d^n\d_s\phi^{(a)}\d^n\big( u^{(b)}\cdot\nab(u^{(c)}\cdot\nab\phi^{(e)}) \big)dxds\\
    \lesssim &\int_0^t (1+\lV\rr\rV_{H^2}) \lV\d_s\phi^{(a)}\rV_{L^2}\lV \nab u\rV_{H^2}^2\lV\nab\phi^{(a)}\rV_{H^2}+\\
    &+\sum_{b+c+e=a,|e|<|a|}(1+\lV\rr\rV_{H^2}) \lV\d_s\phi^{(a)}\rV_{H^{N(a)}}\lV\nab u^{(b)}\rV_{H^{N(a)}}\lV\nab u^{(c)}\rV_{H^{N(a)}}\lV\nab\phi^{(a)}\rV_{H^{N(a)+1}}ds\\
    \lesssim &\ep_1^4\<t\>^{2H(a)\delta}.
    \end{align*}
    This completes the proof of \textbf{Step 5}.

    \textbf{Step 6. We prove the bound
    	\begin{equation*}
    	\int_0^t \bfR_6(s)ds\lesssim \ep_1^3\<t\>^{2H(a)\delta}.
    	\end{equation*}}

    \noindent In order to deal with this contribution, it suffices to show that
    \begin{align*}
    \int_0^t\int_{\R^3} \sum_{|n|\in\{0,N(a),b+c+e=a\}} (1+\rr)\d^n\d_s\phi^{(a)}\d^n \big(\phi^{(b)}\nab\phi^{(c)}\nab\phi^{(e)}\big) dxds\lesssim \ep_1^4.
    \end{align*}
    When $|a|\geq N_1-1$, we may assume that $|b+c|\leq |e|$, then we bound the left-hand side by
    \begin{align*}
    & \int_0^t \lV\d_s\phi^{(a)}\rV_{H^{N(a)}}\lV\phi^{(b)}\rV_{W^{N(a),\infty}}\lV\nab\phi^{(c)}\rV_{W^{N(a),\infty}}\lV \nab\phi^{(e)}\rV_{H^{N(a)}}ds\\
    \lesssim &\int_0^t \ep_1^4\<s\>^{H(a)\delta-1/2+2\delta-1+\delta+H(c)\delta+H(e)\delta}ds\lesssim \ep_1^4.
    \end{align*}
    The case $|b+c|>|e|$ can be estimated similarly. When $|a|\leq N_1-2$, by H\"{o}lder we can bound this contribution by
    \begin{align*}
    &\int_0^t \lV\d_s\phi^{(a)}\rV_{H^{N(a)}}\big(\lV\phi^{(b)}\rV_{L^{\infty}}\lV\nab\phi^{(c)}\rV_{L^{\infty}}\lV\nab\phi^{(e)}\rV_{H^{N(a)}}+\lV\nab\phi^{(b)}\rV_{H^{N(a)}}\lV\nab\phi^{(c)}\rV_{L^{\infty}}\lV\nab\phi^{(e)}\rV_{L^{\infty}}\big)ds\\
    \lesssim & \int_0^t \ep_1\<s\>^{H(a)\delta}(\ep_1^3\<s\>^{-1/2+2\delta-1+\delta+H(c+2)\delta+H(e)\delta}+\ep_1^3\<s\>^{H(b)\delta-2+2\delta+H(c+2)\delta+H(e+2)\delta})ds\lesssim \ep_1^4.
    \end{align*}
    This concludes the treatment of \textbf{Step 6}. Therefore, this completes the proof of the lemma.

\end{proof}

\section{Bounds on the profile: Weighted $L^2$ norm}

In this section we prove (\ref{Main_Prop_result2}), namely
\begin{proposition}      \label{Prop_Psia}
	With the hypothesis in Proposition \ref{Main_Prop}, for any $t\in[0,T]$, $0\leq |a|\leq N_1-1$, we have
	\begin{equation*}
	\lV\mathcal{F}^{-1}(|\xi|\d_{\xi}\widehat{\Psi^{(a)}})\rV_{H^{N(|a|+1)}}\lesssim \ep_0\<t\>^{H(a+1)\delta}.
	\end{equation*}
\end{proposition}

\begin{proof}
	By the bound (\ref{xidxif}) and $S=t\d_t+x\cdot\d_x$, we get
	\begin{equation}                \label{keyIneq}
	\lV\mathcal{F}^{-1}(|\xi|\nab_{\xi}\widehat{\Psi^{(a)}}(\xi))\rV_{H^N}\lesssim \lV S\Psi^{(a)}\rV_{H^N}+\lV\Om \Psi^{(a)}\rV_{H^N}+\lV \Psi^{(a)}\rV_{H^N}+\lV t\d_t \Psi^{(a)}\rV_{H^N},
	\end{equation}
	where
	\begin{equation*}
	S\Psi^{(a)}=e^{-it|\nab|}S\Phi^{(a)},\ \Om \Psi^{(a)}=e^{-it|\nab|}\Om\Phi^{(a)},
	\end{equation*}
	it suffices from (\ref{Main_Prop_result1}) to prove that
	\begin{equation}   \label{dtPsi_ep0}
	\lV \d_t \Psi^{(a)}(t)\rV_{H^{N(|a|+1)}}\lesssim \ep_0\<t\>^{-1+H(|a|+1)\delta}.
	\end{equation}
	We only prove the case when $|t|\geq 1$, from $\phi^{(a)}$-equation in (\ref{Main_Sys_VecFie}), it suffices to prove
	\begin{gather}       \label{dtPsi_Term1}
	\lV \d_t u^{(b)}\cdot\nab\phi^{(c)}\rV_{H^{N(|a|+1)}}\lesssim \ep_1^2\<t\>^{-1},\\\label{dtPsi_Term0}
	\lV \rr^{(b)}\Delta\phi^{(c)}\rV_{H^{N(|a|+1)}}\lesssim \ep_1^2\<t\>^{-1+H(|a|+1)\delta} ,\\ \label{dtPsi_Term2}
	\lV u^{(b)}\cdot\nab\d_t\phi^{(c)}\rV_{H^{N(|a|+1)}}\lesssim \ep_1^2\<t\>^{-1+H(|a|+1)\delta},
	\end{gather}
	for $b+c=a$, $k\in\Z$, and
	\begin{gather} \label{dtPsi_Term3}
	\lV u^{(b)}\cdot\nab(u^{(c)}\cdot\nab\phi^{(e)})\rV_{H^{N(|a|+1)}}\lesssim \ep_1^3\<t\>^{-1},    \\     \label{dtPsi_Term4}
	\lV \frac{\Phi^{(b)}}{|\nab|}\Phi^{(c)}\Phi^{(e)}\rV_{H^{N(|a|+1)}}\lesssim \ep_1^3\<t\>^{-1},
	\end{gather}
	for $b+c+e=a$, $k\in\Z$.
	
	\textbf{Step 1: Proof of (\ref{dtPsi_Term1}).}
	
	When $|a|= N_1-1$, $|b|\geq |c|$, by (\ref{dtu-L2}), (\ref{drra,v,w_L2dec}) and (\ref{Dec_Phi}), we have
	\begin{align*}
	\lV \d_t u^{(b)}\cdot\nab\phi^{(c)}\rV_{H^{N(|a|+1)}}\lesssim  \lV \d_t u^{(b)}\rV_{H^{N(|a|+1)}}\sum_k 2^{N(|a|+1)k^+}\lV P_k \nab\phi^{(c)}\rV_{L^{\infty}}
	\lesssim  \ep_1^2 \<t\>^{-3/2+(6+|b|)\delta}.
	\end{align*}
	When $|a|= N_1-1$, $|b|<|c|$, using (\ref{dtrr-inf}), it follows that
	\begin{equation*}
	\lV \d_t u^{(b)}\cdot\nab\phi^{(c)}\rV_{H^{N(|a|+1)}}\lesssim  (\sum_{k\geq 0}2^{N(|a|+1)k}\lV P_k \d_t u^{(b)}\rV_{L^{\infty}}+\lV \d_t u^{(b)}\rV_{L^{\infty}}) \lV\nab \phi^{(c)}\rV_{H^{N(|a|+1)}}\lesssim \ep_1^2 \<t\>^{-1-l(b)/4+H(c)\delta},
	\end{equation*}
	When $|a|\leq N_1-2$, we obtain from (\ref{dtu-L2}), (\ref{drra,v,w_L2dec}), (\ref{Dec_Phi}) and (\ref{dtrr-inf})
	\begin{align*}
	\lV \d_t u^{(b)}\cdot\nab\phi^{(c)}\rV_{H^{N(|a|+1)}}\lesssim \lV\d_t u^{(b)}\rV_{H^{N(|a|+1)}}\lV \nab \phi^{(c)}\rV_{L^{\infty}}+\lV\d_t u^{(b)}\rV_{L^{\infty}}\lV \nab \phi^{(c)}\rV_{H^{N(|a|+1)}}\lesssim \ep_1^2 \<t\>^{-1-l(b)/4+H(c)\delta},
	\end{align*}
	Hence, the bound (\ref{dtPsi_Term1}) follows.

	\textbf{Step 2: Proof of (\ref{dtPsi_Term0}).}
	
	When $|a|=|b|=N_1-1,c=0$, by (\ref{Dec_Phi}) we have
	\begin{align*}
	\lV \rr^{(a)}\Delta\phi\rV_{H^{N(|a|+1)}}\lesssim  \lV u^{(a)}\rV_{H^{N(|a|+1)}} \sum_{k}2^{N(|a|+1)k^++k}\lV P_k\nab\phi\rV_{L^{\infty}}\lesssim \ep_1^2\<t\>^{-1+H(|a|+1)\delta}.
	\end{align*}
	Then it suffices to estimate (\ref{dtPsi_Term0}) when $|b|\leq N_1-2$.
	
	Decomposing dyadically in frequency, we have
	\begin{align*}
	\lV \rr^{(b)}\Delta\phi^{(c)}\rV_{H^{N(|a|+1)}}\lesssim (\sum_k 2^{2N(|a|+1)k^+}\mathfrak{X}_k^2)^{1/2},
	\end{align*}
	where
	\begin{equation*}
	\mathfrak{X}_k:=\lV\sum_{k_1,k_2}\varphi_k(\xi)\int_{\R^3} e^{it|\eta|}|\eta| \widehat{P_{k_1}\rr^{(b)}}(\xi-\eta)\widehat{P_{k_2}\Psi^{(c)}}(\eta)d\eta\rV_{L^2}.
	\end{equation*}
	We further divided $\mathfrak{X}_k$ into high-low, low-high, and high-high case,
	\begin{equation*}
	\mathfrak{X}_k\lesssim \mathfrak{X}^{hl}_k+\mathfrak{X}^{lh}_k+\mathfrak{X}^{hh}_k,
	\end{equation*}
	where
	\begin{gather*}
	\mathfrak{X}^{hl}_k:=\sum_{k_2\leq k-5}\lV\varphi_k(\xi)\int_{\R^3} e^{it|\eta|}|\eta| \widehat{P_{[k-3,k+3]}\rr^{(b)}}(\xi-\eta)\widehat{P_{k_2}\Psi^{(c)}}(\eta)d\eta\rV_{L^2},\\
	\mathfrak{X}^{lh}_k:=\lV\varphi_k(\xi)\int_{\R^3} e^{it|\eta|}|\eta| \widehat{P_{<k+6}u\rr^{(b)}}(\xi-\eta)\widehat{P_{[k-5,k+5]}\Psi^{(c)}}(\eta)d\eta\rV_{L^2},\\
	\mathfrak{X}^{hh}_k:=\sum_{k_1,k_2;k_2\geq k+5}\lV\varphi_k(\xi)\int_{\R^3} e^{it|\eta|}|\eta| \widehat{P_{[k_2-3,k_2+3]}\rr^{(b)}}(\xi-\eta)\widehat{P_{k_2}\Psi^{(c)}}(\eta)d\eta\rV_{L^2}.
	\end{gather*}
	
	{\bf High-low interaction: The contribution of $\mathfrak{X}^{hl}_k$.}
	
	\noindent Integration by parts in $\eta$ yields
	\begin{align*}
	\mathfrak{X}^{hl}_k\lesssim &  t^{-1}\sum_{k_2<k-5}\Big[\lV\int_{\R^3} \nab_{\eta}(\eta)\widehat{P_{[k-3,k+3]} \rr^{(b)}}(\xi-\eta) \widehat{P_{k_2}\Phi^{(c)}}(\eta) d\eta\rV_{L^2}\\
	&+\lV \int_{\R^3} \eta \nab_{\eta} \widehat{P_{[k-3,k+3]} \rr^{(b)}}(\xi-\eta) \widehat{P_{k_2}\Phi^{(c)}}(\eta) d\eta\rV_{L^2}\\
	&+\lV \int_{\R^3} \eta  \widehat{P_{[k-3,k+3]} \rr^{(b)}}(\xi-\eta) e^{it|\eta|}\nab_{\eta}\widehat{P_{k_2}\Psi^{(c)}}(\eta) d\eta\rV_{L^2}\Big].
	\end{align*}
	When $|a|=|b|\leq N_1-2, c=0$, by (\ref{Symbol_Pre}) we have
	\begin{align*}
	\big[\sum_k 2^{2N(|a|+1)k^+}(\mathfrak{X}^{hl}_k)^2\big]^{1/2}
	\lesssim & t^{-1}  [\lV \rr^{(b)}\rV_{H^{N(|a|+1)}}\lV \Phi^{(c)}\rV_{H^2}+\lV \FF^{-1}( |\xi|\d_{\xi}\widehat{\rr^{(b)}})\rV_{H^{N(|a|+1)}}\sum_{k_2}\lV \Phi^{(c)}\rV_{L^{\infty}}\\
	&+\lV \rr^{(b)}\rV_{H^{N(|a|+1)}}\lV\FF^{-1}(|\xi|\d_{\xi}\widehat{\Psi^{(c)}})\rV_{H^{N(|a|+1)}}]
	\end{align*}
	Since from (\ref{keyIneq}), (\ref{Main_Prop_Ass1}), (\ref{dtrr-L2}), and (\ref{drra,v,w_L2dec}) we obtain
	\begin{align*}
	&\lV |\xi|\d_{\xi}\widehat{\rr^{(b)}}(\xi)\rV_{H^{N(|a|+1)}}\lesssim \ep_1+t\lV\d_t\rr^{(b)}\rV_{H^{N(|a|+1)}}\\
	\lesssim &\ep_1+\sum_{|\tilde{b}|\leq |b|}\lV (\nab\rr^{(\tilde{b})},\nab u^{(\tilde{b})})\rV_{H^{N(|a|+1)}}\lesssim  \ep_1+\ep_1\<t\>^{1/2+(2+|b|)\delta}.
	\end{align*}
	Hence, combining this and (\ref{Main_Prop_Ass1}), (\ref{Main_Prop_Ass2}) we may bound this contribution by
	\begin{align*}
	\big[\sum_k 2^{2N(|a|+1)k^+}(\mathfrak{X}^{hl}_k)^2\big]^{1/2}\lesssim & t^{-1}(\ep_1^2\<t\>^{H(c)\delta}+\ep_1^2\<t\>^{-1/3}+\ep_1^2\<t\>^{H(|c|+1)\delta})\\
	\lesssim & \ep_1^2\<t\>^{-1+H(|c|+1)\delta},
	\end{align*}
	When $|b|\leq N_1-2,|c|\geq 1$,	it follows from (\ref{Main_Prop_Ass1}), (\ref{Main_Prop_Ass2}) and (\ref{rv}) that
	\begin{align*}
	\big[\sum_k 2^{2N(|a|+1)k^+}(\mathfrak{X}^{hl}_k)^2\big]^{1/2}
	\lesssim &t^{-1}[\lV \rr^{(b)}\rV_{H^{N(|a|+1)}}\lV \Phi^{(c)}\rV_{H^2}+(\sum_k 2^{2N(|a|+1)k^+}\lV rP_{[k-3,k+3]}\rr^{(b)}\rV_{L^{\infty}}^2)^{1/2}\lV \Phi^{(c)}\rV_{H^2}\\
	&+\lV \rr^{(b)}\rV_{H^{N(|a|+1)}}\lV\FF^{-1}(|\xi|\d_{\xi}\Psi^{(c)})\rV_{H^{N(|a|+1)}}]\\
	\lesssim & t^{-1}(\ep_1^2\<t\>^{H(c)\delta}+\sum_{\alpha_1\leq 1,\alpha_2\leq 2}\lV \d_r\tilde{\Om}^{\alpha_1}\rr^{(b)}\rV_{H^{N(|a|+1)}}^{1/2}\lV \tilde{\Om}^{\alpha_2}\rr^{(b)}\rV_{H^{N(|a|+1)}}^{1/2}\ep_1\<t\>^{H(c)\delta}\\
	&+\ep_1^2\<t\>^{H(|c|+1)\delta})\\
	\lesssim &\ep_1^2 \<t\>^{-1+H(|a|+1)\delta},
	\end{align*}
	Therefore, we have
	\begin{equation}\label{hl-Term0}
	\big[\sum_k 2^{2N(|a|+1)k^+}(\mathfrak{X}^{hl}_k)^2\big]^{1/2}\lesssim \ep_1^2\<t\>^{-1+H(|a|+1)\delta}.
	\end{equation}
	
	{\bf Low-high and high-high interaction: The contribution of $\mathfrak{X}^{lh}_k$ and $\mathfrak{X}^{hh}_k$.}
	
	\noindent We only estimate $\mathfrak{X}^{lh}_k$ in detail, the term $\mathfrak{X}^{hh}_k$ can be estimated similarly. Integration by parts in $\eta$, it follows that
	\begin{align*}
	\mathfrak{X}^{lh}_k\lesssim &  t^{-1}\Big[\lV\int_{\R^3} \nab_{\eta}(\eta)\widehat{P_{<k+6} \rr^{(b)}}(\xi-\eta) \widehat{P_{[k-5,k+5]}\Phi^{(c)}}(\eta) d\eta\rV_{L^2}\\
	&+\lV \int_{\R^3} \eta \nab_{\eta} \widehat{P_{<k+6} \rr^{(b)}}(\xi-\eta) \widehat{P_{[k-5,k+5]}\Phi^{(c)}}(\eta) d\eta\rV_{L^2}\\
	&+\lV \int_{\R^3} \eta  \widehat{P_{<k+6} u^{(b)}}(\xi-\eta) e^{it|\eta|}\nab_{\eta}\widehat{P_{[k-5,k+5]}\Psi^{(c)}}(\eta) d\eta\rV_{L^2}\Big]
	\end{align*}
	Then by Young's inequality, (\ref{Main_Prop_Ass1}), (\ref{Main_Prop_Ass2}) and (\ref{rv}), we have
	\begin{equation}\label{lhhh-Term0}
	\begin{aligned}
	\big[\sum_k 2^{2N(|a|+1)k^+}(\mathfrak{X}^{lh}_k)^2\big]^{1/2}
	\lesssim & t^{-1}  [\lV \rr^{(b)}\rV_{H^2}\lV \Phi^{(c)}\rV_{H^{N(|a|+1)}}+\lV r\rr^{(b)}\rV_{L^{\infty}}\lV \Phi^{(c)}\rV_{H^{N(|a|+1)+1}}\\
	&+\lV \rr^{(b)}\rV_{H^2}\lV\FF^{-1}(|\xi|\d_{\xi}\widehat{\Psi^{(c)}})\rV_{H^{N(|a|+1)}}]\\
	\lesssim & t^{-1}(\ep_1^2 \<t\>^{H(c)\delta}+\ep_1^2\<t\>^{H(|c|+1)\delta})\\
	\lesssim & \ep_1^2 \<t\>^{-1+H(|a|+1)\delta}.
	\end{aligned}
	\end{equation}
	
	From (\ref{hl-Term0}) and (\ref{lhhh-Term0}), the bound (\ref{dtPsi_Term0}) for $|b|\leq N_1-2, |c|\geq 0$ is obtained immediately. This completes the proof of (\ref{dtPsi_Term0}).
	
	\textbf{Step 3: Proof of (\ref{dtPsi_Term2}).}
	
	If $|a|=|b|=N_1-1,c=0$. (\ref{Dec_Phi}) implies
	\begin{equation}             \label{dtPsi_Term2_case1}
	\lV u^{(b)}\cdot \nab\d_t\phi\rV_{H^{N(|a|+1)}}\lesssim  \lV u^{(|b|)}\rV_{H^{N(|a|+1)}} \sum_{k}2^{N(|a|+1)k^++k}\lV \d_t\phi\rV_{L^{\infty}}\lesssim \ep_1^2\<t\>^{-1+H(|a|+1)\delta}.
	\end{equation}
	Now it suffices to consider the case $|b|\leq N_1-2, |c|\geq 0$.
	
	Decomposing dyadically in frequency, we have
	\begin{align*}
	\lV u^{(b)}\cdot\nab\d_t\phi^{(c)}\rV_{H^{N(|a|+1)}}\lesssim (\sum_k 2^{2N(|a|+1)k^+}I_k^2)^{1/2},
	\end{align*}
	where
	\begin{equation*}
	I_k:=\lV\sum_{k_1,k_2}\varphi_k(\xi)\int_{\R^3} e^{it|\eta|}\eta \widehat{P_{k_1}u^{(b)}}(\xi-\eta)\widehat{P_{k_2}\Psi^{(c)}}(\eta)d\eta\rV_{L^2}.
	\end{equation*}
	We further divided $I_k$ into high-low, low-high, high-high case,
	\begin{equation*}
	I_k\lesssim I^{hl}_k+I^{lh}_k+I^{hh}_k,
	\end{equation*}
	where
	\begin{gather*}
	I^{hl}_k:=\sum_{k_2\leq k-5}\lV\varphi_k(\xi)\int_{\R^3} e^{it|\eta|}\eta \widehat{P_{[k-3,k+3]}u^{(b)}}(\xi-\eta)\widehat{P_{k_2}\Psi^{(c)}}(\eta)d\eta\rV_{L^2},\\
	I^{lh}_k:=\lV\varphi_k(\xi)\int_{\R^3} e^{it|\eta|}\eta \widehat{P_{<k+6}u^{(b)}}(\xi-\eta)\widehat{P_{[k-5,k+5]}\Psi^{(c)}}(\eta)d\eta\rV_{L^2},\\
	I^{hh}_k:=\sum_{k_1,k_2;k_2\geq k+5}\lV\varphi_k(\xi)\int_{\R^3} e^{it|\eta|}\eta \widehat{P_{[k_2-3,k_2+3]}u^{(b)}}(\xi-\eta)\widehat{P_{k_2}\Psi^{(c)}}(\eta)d\eta\rV_{L^2}.
	\end{gather*}
	
	{\bf High-low interaction: The contribution of $I^{hl}_k$.}
	
	\noindent Integration by parts in $\eta$ yields
	\begin{align*}
	I^{hl}_k\lesssim &  t^{-1}\sum_{k_2<k-5}\Big[\lV\int_{\R^3} \d_{\eta_j}\big(\frac{\eta_j}{|\eta|}\eta\big)\widehat{P_{[k-3,k+3]} u^{(b)}}(\xi-\eta) \widehat{P_{k_2}\Phi^{(c)}}(\eta) d\eta\rV_{L^2}\\
	&+\lV \int_{\R^3} \frac{\eta_j}{|\eta|}\eta \d_{\eta_j} \widehat{P_{[k-3,k+3]} u^{(b)}}(\xi-\eta) \widehat{P_{k_2}\Phi^{(c)}}(\eta) d\eta\rV_{L^2}\\
	&+\lV \int_{\R^3} \frac{\eta_j}{|\eta|}\eta  \widehat{P_{[k-3,k+3]} u^{(b)}}(\xi-\eta) e^{it|\eta|}\d_{\eta_j}\widehat{P_{k_2}\Psi^{(c)}}(\eta) d\eta\rV_{L^2}\Big].
	\end{align*}
	When $|a|=|b|\leq N_1-2, c=0$, by (\ref{keyIneq}), (\ref{dtu-L2}) and (\ref{drra,v,w_L2dec}), one obtains
	\begin{align*}
	\lV |\xi|\d_{\xi}\widehat{u^{(b)}}(\xi)\rV_{H^{N(|a|+1)}}\lesssim \ep_1+t\lV\d_t u^{(b)}\rV_{H^{N(|a|+1)}}\lesssim \ep_1\<t\>^{1/2+(2+|b|)\delta}.
	\end{align*}
	Then by (\ref{Symbol_Pre}), (\ref{Main_Prop_Ass1}), (\ref{Main_Prop_Ass2}), (\ref{Dec_Phi}) we may bound this contribution by
	\begin{align*}
	\big[\sum_k 2^{2N(|a|+1)k^+}(I^{hl}_k)^2\big]^{1/2}
	\lesssim & t^{-1}  [\lV u^{(b)}\rV_{H^{N(|a|+1)}}\lV \Phi^{(c)}\rV_{H^2}+\lV \FF^{-1}(|\xi|\d_{\xi}\widehat{u^{(b)}})\rV_{H^{N(|a|+1)}}\sum_{k_2}\lV \Phi^{(c)}\rV_{L^{\infty}}\\
	&+\lV u^{(b)}\rV_{H^{N(|a|+1)}}\lV\FF^{-1}(|\xi|\d_{\xi}\widehat{\Psi^{(c)}})\rV_{H^{N(|a|+1)}}]\\
	\lesssim & t^{-1}(\ep_1^2\<t\>^{H(c)\delta}+\ep_1^2\<t\>^{-1/3}+\ep_1^2\<t\>^{H(|c|+1)\delta})\\
	\lesssim & \ep_1^2\<t\>^{-1+H(|c|+1)\delta},
	\end{align*}
	When $|b|\leq N_1-2,|c|\geq 1$,	it follows from (\ref{Main_Prop_Ass1}), (\ref{Main_Prop_Ass2}) and (\ref{rv}) that
	\begin{align*}       \label{dtPsi_Term2_case2}
	&\big[\sum_k 2^{2N(|a|+1)k^+}(I^{hl}_k)^2\big]^{1/2}\\
	\lesssim &t^{-1}[\lV u^{(b)}\rV_{H^{N(|a|+1)}}\lV \Phi^{(c)}\rV_{H^2}+(\sum_k2^{2N(|a|+1)k^+}\lV rP_{[k-3,k+3]}u^{(b)}\rV_{L^{\infty}}^2)^{1/2}\lV \Phi^{(c)}\rV_{H^2}\\
	&+\lV u^{(b)}\rV_{H^{N(|a|+1)}}\lV|\xi|\d_{\xi}\Psi^{(c)}\rV_{H^{N(|a|+1)}}]\\
	\lesssim & t^{-1}(\ep_1^2\<t\>^{H(c)\delta}+\sum_{\alpha_1\leq 1,\alpha_2\leq 2}\lV \d_r\tilde{\Om}^{\alpha_1}v^{(b)}\rV_{H^{N(|a|+1)}}^{1/2}\lV \tilde{\Om}^{\alpha_2}v^{(b)}\rV_{H^{N(|a|+1)}}^{1/2}\ep_1\<t\>^{H(c)\delta}\\
	&+\ep_1^2\<t\>^{H(|c|+1)\delta})\\
	\lesssim &\ep_1^2 \<t\>^{-1+H(|a|+1)\delta},
	\end{align*}
	Therefore, we have
	\begin{equation}\label{hl}
	\big[\sum_k 2^{2N(|a|+1)k^+}(I^{hl}_k)^2\big]^{1/2}\lesssim \ep_1^2\<t\>^{-1+H(|a|+1)\delta}.
	\end{equation}
	
	{\bf Low-high and high-high interaction: The contribution of $I^{lh}_k$ and $I^{hh}_k$.}
	
	\noindent We only bound $I^{lh}_k$ in detail, the term $I^{hh}_k$ can be estimated similarly. Integration by parts in $\eta$, it follows that
	\begin{align*}
	I^{lh}_k\lesssim &  t^{-1}\Big[\lV\int_{\R^3} \d_{\eta_j}\big(\frac{\eta_j}{|\eta|}\eta\big)\widehat{P_{<k+6} u^{(b)}}(\xi-\eta) \widehat{P_{[k-5,k+5]}\Phi^{(c)}}(\eta) d\eta\rV_{L^2}\\
	&+\lV \int_{\R^3} \frac{\eta_j}{|\eta|}\eta \d_{\eta_j} \widehat{P_{<k+6} u^{(b)}}(\xi-\eta) \widehat{P_{[k-5,k+5]}\Phi^{(c)}}(\eta) d\eta\rV_{L^2}\\
	&+\lV \int_{\R^3} \frac{\eta_j}{|\eta|}\eta  \widehat{P_{<k+6} u^{(b)}}(\xi-\eta) e^{it|\eta|}\d_{\eta_j}\widehat{P_{[k-5,k+5]}\Psi^{(c)}}(\eta) d\eta\rV_{L^2}\Big]
	\end{align*}
	Then by (\ref{Symbol_Pre}), (\ref{Main_Prop_Ass1}), (\ref{Main_Prop_Ass2}) and (\ref{rv}), we have
	\begin{equation}\label{lhhh}
	\begin{aligned}
	\big[\sum_k 2^{2N(|a|+1)k^+}(I^{lh}_k)^2\big]^{1/2}
	\lesssim & t^{-1}  [\lV u^{(b)}\rV_{H^{N(|a|+1)}}\lV \Phi^{(c)}\rV_{H^2}+\lV ru^{(b)}\rV_{L^{\infty}}\lV \Phi^{(c)}\rV_{H^{N(|a|+1)+1}}\\
	&+\lV u^{(b)}\rV_{H^{N(|a|+1)}}\lV\FF^{-1}(|\xi|\d_{\xi}\widehat{\Psi^{(c)}})\rV_{H^{N(|a|+1)}}]\\
	\lesssim & t^{-1}(\ep_1^2 \<t\>^{H(c)\delta}+\ep_1^2\<t\>^{H(|c|+1)\delta})\\
	\lesssim & \ep_1^2 \<t\>^{-1+H(|a|+1)\delta}.
	\end{aligned}
	\end{equation}
	
	From (\ref{hl}) and (\ref{lhhh}), the bound (\ref{dtPsi_Term2}) for $|b|\leq N_1-2, |c|\geq 0$ is obtained immediately. This completes the proof of (\ref{dtPsi_Term2}).

	\textbf{Step 4: Proof of (\ref{dtPsi_Term3}).}
	
	When $|e|\leq N_1-2$, using (\ref{drra,v,w_a_inf}) and (\ref{Dec_Phi}), we get
	\begin{align*}
	\lV u^{(b)}\cdot\nab(u^{(c)}\cdot\nab\phi^{(e)})\rV_{H^{N(|a|+1)}}
	\lesssim &\lV u^{(b)}\rV_{H^{N(|a|+1)}}(\lV\nab u^{(c)}\rV_{L^{\infty}}\lV \nab\phi^{(e)}\rV_{L^{\infty}}+\lV u^{(c)}\rV_{L^{\infty}}\lV \nab^2\phi^{(e)}\rV_{L^{\infty}})\\
	&+\lV u^{(b)}\rV_{L^{\infty}}(\lV u^{(c)}\rV_{H^{N(|a|+1)+1}}\lV \nab\phi^{(e)}\rV_{L^{\infty}}+\lV u^{(c)}\rV_{L^{\infty}}\lV \nab\phi^{(e)}\rV_{H^{N(|a|+1)+1}})\\
	\lesssim & \ep_1^3\<t\>^{-5/4},
	\end{align*}
	When $|e|\geq N_1-1$, by (\ref{drr,v,w_inf}), we have
	\begin{align*}
	\lV u\cdot\nab(u\cdot\nab\phi^{(e)})\rV_{H^{N(|a|+1)}}
	\lesssim & (\sum_{k\geq 0}2^{N(|e|+1)k+k}\lV P_k u\rV_{L^{\infty}}+\lV P_{<0}u\rV_{L^{\infty}})^2 \lV \nab\phi^{(e)}\rV_{H^{N(|e|+1)+1}}\lesssim \ep_1^3\<t\>^{-5/4},
	\end{align*}
	Hence, the bound (\ref{dtPsi_Term3}) follows.
	
	Finally, the bound (\ref{dtPsi_Term4}) is an consequence of (\ref{Dec_Phi}) and (\ref{Dec_Phi/nab}).
	This completes the proof of the proposition.
\end{proof}

{\bf Acknowledgement} L. Zhao was partially supported by the NSFC Grant of China  (No. 11771415).

\end{document}